\documentclass[12pt,twoside,final]{article}
\usepackage{amsmath,amsthm,amssymb}

\usepackage[nottoc]{tocbibind}

\usepackage[a4paper,left=20mm,right=20mm,top=30mm,bottom=30mm]{geometry}
\usepackage{imakeidx}
\usepackage{xr-hyper} 
\usepackage{xr}
\usepackage[all]{xy}

\usepackage{fancyhdr}
\usepackage[final]{hyperref}
\usepackage{cleveref}
\usepackage{seqsplit}
\usepackage{xstring}
\usepackage[xcdraw]{xcolor}
\definecolor{mycolor}{rgb}{0.122, 0.435, 0.698}

\hypersetup{
 colorlinks,
 citecolor=ocre,
 filecolor=ocre,
 linkcolor=ocre,
 urlcolor=ocre
}
\makeatletter
\DeclareOldFontCommand{\rm}{\normalfont\rmfamily}{\mathrm}
\DeclareOldFontCommand{\sf}{\normalfont\sffamily}{\mathsf}
\DeclareOldFontCommand{\tt}{\normalfont\ttfamily}{\mathtt}
\DeclareOldFontCommand{\bf}{\normalfont\bfseries}{\mathbf}
\DeclareOldFontCommand{\it}{\normalfont\itshape}{\mathit}
\DeclareOldFontCommand{\sl}{\normalfont\slshape}{\@nomath\sl}
\DeclareOldFontCommand{\sc}{\normalfont\scshape}{\@nomath\sc}
\makeatother

\usepackage[shortlabels]{enumitem}
\setlist{nolistsep} %

\newlist{asslist}{enumerate}{1} % also creates a counter called 'propenumi'
\setlist[asslist]{label=(\roman*), ref=\thethmT(\roman*)}
\crefalias{asslistenumi}{Assumption} 

\newlist{thmlist}{enumerate}{1} % also creates a counter called 'propenumi'
\setlist[thmlist]{label=(\alph*), ref=\thethmT(\alph*)}

\usepackage{tikz} % Required for drawing custom shapes

\usepackage{xcolor} % Required for specifying colors by name
\definecolor{ocre_old}{RGB}{243,102,25} % Define the orange color used for highlighting throughout the book
\definecolor{ocre}{rgb}{0.122, 0.435, 0.698}
\usepackage{amsmath,amsfonts,amssymb,amsthm} % For math equations, theorems, symbols, etc

\makeatletter
\newtheoremstyle{ocrenumbox}% % Theorem style name
{0pt}% Space above
{0pt}% Space below
{\sl}% % Body font
{}% Indent amount
{\small\bf\sffamily\color{ocre}}% % Theorem head font
{\;}% Punctuation after theorem head
{0.25em}% Space after theorem head
{\small\sffamily\color{ocre}\thmname{#1}\nobreakspace\thmnumber{\@ifnotempty{#1}{}\@upn{#2}}% Theorem text (e.g. Theorem 2.1)
\thmnote{\nobreakspace\the\thm@notefont\sffamily\bfseries\color{black}---\nobreakspace#3.}} % Optional theorem note

\newtheoremstyle{ocrenumhypbox}% % Theorem style name
{0pt}% Space above
{0pt}% Space below
{}% % Body font
{}% Indent amount
{\small\bf\sffamily\color{ocre}}% % Theorem head font
{\;}% Punctuation after theorem head
{0.25em}% Space after theorem head
{\small\sffamily\color{ocre}\thmname{#1}\nobreakspace\thmnumber{\@ifnotempty{#1}{}\@upn{#2}}% Theorem text (e.g. Theorem 2.1)
\thmnote{\nobreakspace\the\thm@notefont\sffamily\bfseries\color{black}---\nobreakspace#3.}} % Optional theorem note

\newtheoremstyle{blacknumex}% Theorem style name
{5pt}% Space above
{5pt}% Space below
{\sl}% Body font
{} % Indent amount
{\small\bf\sffamily}% Theorem head font
{\;}% Punctuation after theorem head
{0.25em}% Space after theorem head
{\small\sffamily{\tiny\ensuremath{\blacksquare}}\nobreakspace\thmname{#1}\nobreakspace\thmnumber{\@ifnotempty{#1}{}\@upn{#2}}% Theorem text (e.g. Theorem 2.1)
\thmnote{\nobreakspace\the\thm@notefont\sffamily\bfseries---\nobreakspace#3.}}% Optional theorem note

\newtheoremstyle{blacknumbox} % Theorem style name
{0pt}% Space above
{0pt}% Space below
{\normalfont}% Body font
{}% Indent amount
{\small\bf\sffamily}% Theorem head font
{\;}% Punctuation after theorem head
{0.25em}% Space after theorem head
{\small\sffamily\thmname{#1}\nobreakspace\thmnumber{\@ifnotempty{#1}{}\@upn{#2}}% Theorem text (e.g. Theorem 2.1)
\thmnote{\nobreakspace\the\thm@notefont\sffamily\bfseries---\nobreakspace#3.}}% Optional theorem note

\newtheoremstyle{ocrenum}% % Theorem style name
{5pt}% Space above
{5pt}% Space below
{\sl}% % Body font
{}% Indent amount
{\small\bf\sffamily\color{ocre}}% % Theorem head font
{\;}% Punctuation after theorem head
{0.25em}% Space after theorem head
{\small\sffamily\color{ocre}\thmname{#1}\nobreakspace\thmnumber{\@ifnotempty{#1}{}\@upn{#2}}% Theorem text (e.g. Theorem 2.1)
\thmnote{\nobreakspace\the\thm@notefont\sffamily\bfseries\color{black}---\nobreakspace#3.}} % Optional theorem note
\makeatother

\theoremstyle{ocrenumbox}
\newtheorem{thmT}{Theorem}[section]
\newtheorem{theoT}{Theorem}
\newtheorem{theoremeT}[thmT]{Theorem}
\newtheorem{lemT}[thmT]{Lemma}

\theoremstyle{ocrenumhypbox}
\newtheorem{hypT}[thmT]{Hypothesis}
\theoremstyle{blacknumex}

\theoremstyle{blacknumbox}
\newtheorem{definitionT}[thmT]{Definition}
\newtheorem{notationT}[thmT]{Notation}
\newtheorem{remarkT}[thmT]{Remark}

\theoremstyle{ocrenum}

\newtheorem{propT}[thmT]{Proposition}

\newtheorem{corollaryT}[thmT]{Corollary}
\RequirePackage[framemethod=default]{mdframed} % Required for creating the theorem, definition, exercise and corollary boxes

\newmdenv[skipabove=7pt,
skipbelow=7pt,
backgroundcolor=black!5,
linecolor=ocre,
innerleftmargin=5pt,
innerrightmargin=5pt,
innertopmargin=5pt,
leftmargin=0cm,
rightmargin=0cm,
innerbottommargin=5pt]{tBox}

\newmdenv[skipabove=7pt,
skipbelow=7pt,
rightline=false,
leftline=true,
topline=false,
bottomline=false,
backgroundcolor=ocre!10,
linecolor=ocre,
innerleftmargin=5pt,
innerrightmargin=5pt,
innertopmargin=5pt,
innerbottommargin=5pt,
leftmargin=0cm,
rightmargin=0cm,
linewidth=4pt]{eBox}	

\newmdenv[skipabove=7pt,
skipbelow=7pt,
rightline=false,
leftline=true,
topline=false,
bottomline=false,
linecolor=ocre,
innerleftmargin=5pt,
innerrightmargin=5pt,
innertopmargin=0pt,
leftmargin=0cm,
rightmargin=0cm,
linewidth=4pt,
innerbottommargin=0pt]{dBox}	

\newmdenv[skipabove=7pt,
skipbelow=7pt,
rightline=false,
leftline=true,
topline=false,
bottomline=false,
linecolor=gray,
backgroundcolor=black!5,
innerleftmargin=5pt,
innerrightmargin=5pt,
innertopmargin=5pt,
leftmargin=0cm,
rightmargin=0cm,
linewidth=4pt,
innerbottommargin=5pt]{cBox}

\newenvironment{theorem}{\begin{tBox}\begin{theoremeT}}{\end{theoremeT}\end{tBox}}
\newenvironment{thmbox}{\begin{tBox}\begin{theoremeT}}{\end{theoremeT}\end{tBox}}
\newenvironment{hyp}{\begin{tBox}\begin{hypT}}{\end{hypT}\end{tBox}}
\newenvironment{theo}{\begin{tBox}\begin{theoT}}{\end{theoT}\end{tBox}}

\newenvironment{defi}{\begin{dBox}\begin{definitionT}}{\end{definitionT}\end{dBox}}	
\newenvironment{notation}{\begin{dBox}\begin{notationT}}{\end{notationT}\end{dBox}}	
\newenvironment{remark}{\begin{dBox}\begin{remarkT}}{\end{remarkT}\end{dBox}}	
\newenvironment{lem}{\begin{dBox}\begin{lemT}}{\end{lemT}\end{dBox}}	
\newenvironment{prop}{\begin{dBox}\begin{propT}}{\end{propT}\end{dBox}}
		
\newenvironment{corollary}{\begin{dBox}\begin{corollaryT}}{\end{corollaryT}\end{dBox}}	
\newenvironment{cor}{\begin{dBox}\begin{corollaryT}}{\end{corollaryT}\end{dBox}}	

\makeatletter
\renewcommand{\@seccntformat}[1]{\llap{\textcolor{ocre}{\csname the#1\endcsname}\hspace{1em}}} 
\renewcommand{\section}{\@startsection{section}{1}{\z@}
{-4ex \@plus -1ex \@minus -.4ex}
{1ex \@plus.2ex }
{\normalfont\large \bf \color{ocre}}}
\renewcommand{\subsection}{\@startsection {subsection}{2}{\z@}
{-3ex \@plus -0.1ex \@minus -.4ex}
{0.5ex \@plus.2ex }
{\normalfont\large\bf\color{ocre} }}
\newcommand{\ul}{{\underline {l}}}

\def\titlerunning#1{\gdef\titrun{#1}}
\makeatletter
\def\author#1{\gdef\autrun{\def\and{\unskip, }#1}\gdef\@author{#1}}
\def\address#1{{\def\and{\\\hspace*{18pt}}\renewcommand{\thefootnote}{}%
\footnote {#1}}%
\markboth{Britta Sp\"ath}{\titrun}}
\makeatother
\def\email#1{e-mail: #1}
\def\subjclass#1{{\renewcommand{\thefootnote}{}%
\footnote{\emph{Mathematics Subject Classification (2010):} #1}}}

\newcommand{\otw}{\text{otherwise}}
\newcommand{\inv}{^{-1}}

\theoremstyle{definition}

\newtheorem{num}[thmT]{}{\color{ocre}}
\theoremstyle{plain}

\theoremstyle{definition}
 \newtheorem{condi}[thmT]{Condition}

\numberwithin{equation}{section}
\numberwithin{table}{section}

\def\norm#1#2{{\operatorname N}_{#1}(#2)}
\def\cent#1#2{{\operatorname C}_{#1}(#2)}

\newcommand{\Id}{\operatorname {Id}}

\newcommand{\wt}{\widetilde}
\newcommand{\Abstand}{\|}
\newcommand{\wh}{\widehat}
\newcommand{\pwh}{\protect\widehat}
\newcommand{\pwt}{\protect\widetilde}
\newcommand{\phat}{\protect\widehat}
\newcommand{\bN}{{\mathbf N}}

\newcommand{\HC}{{\mathrm {HC}}}

\newcommand{\wG}{{\widetilde G}}

\newcommand{\Stab}{\operatorname{Stab}}

\newcommand{\Kcirc}{{K_0}}
\newcommand{\Kcircd}{{K_{0,d}}}

\newcommand{\la}{\ensuremath{\lambda}}

\newcommand{\bT}{{\mathbf T}}
\newcommand{\odd}{{\rm odd}}
\newcommand{\even}{{\rm even}}
\newcommand{\type}{{\mathsf {type}}}
\newcommand{\tDlsc}{\tD_{l,\mathrm{sc}}}

\newcommand{\bB}{{\mathbf B}}
\newcommand{\bP}{{\mathbf P}}
\newcommand{\bU}{{\mathbf U}}
\newcommand{\bG}{{{\mathbf G}}}
\newcommand{\bGI}{{{\mathbf G}_{I}}}

\newcommand{\bH}{{{\mathbf H}}}
\newcommand{\bL}{{{\mathbf L}}}
\newcommand{\bK}{{{\mathbf K}}}

\newcommand{\EL}{E_L}

\newcommand{\bS}{{\mathbf S}}
\newcommand{\HF}{{{\bH}^F}}

\newcommand{\bX}{{\mathbf X}}
\newcommand{\wbG}{\wt{\mathbf G}}
\newcommand{\Irr}{\mathrm{Irr}}

\newcommand{\Char}{\mathrm{Char}}

\newcommand{\bZ}{{\mathbf Z}}

\newcommand{\SL}{\operatorname{SL}}

\newcommand{\GL}{\operatorname{GL}}

\newcommand{\SO}{\operatorname{SO}}

\newcommand{\ZZ}{\ensuremath{\mathbb{Z}}}

\newcommand{\TT}{\ensuremath{\mathbb{T}}}
\newcommand{\KK}{\ensuremath{\mathbb{K}}}
\newcommand{\MM}{\ensuremath{\mathbb{M}}}

\newcommand{\ov}{\overline }
\newcommand{\pov}{\protect\overline }
\newcommand{\whL}{{\wh L}}
\newcommand{\wtL}{{\wt L}}
\newcommand{\R}{\operatorname{R}}
\newcommand{\obG}{{\overline {\mathbf G}}}

\newcommand{\xx}{\mathbf x }
\newcommand{\n}{\mathbf n }\newcommand{\nn}{\mathbf n }
\newcommand{\h}{\mathbf h }\newcommand{\hh}{\mathbf h }

\renewcommand{\o}{\overline}

\newcommand{\oN}{{\o N}}

\newcommand{\Cent}{\ensuremath{{\rm{C}}}}
\newcommand{\NNN}{\ensuremath{{\mathrm{N}}}}
\newcommand{\ZZZ}{\ensuremath{{\mathrm{Z}}}}
\newcommand{\Sym}{{\mathcal{S}}}
\newcommand{\Young}{{\mathcal{Y}}}

\def\restr#1|#2{\left.#1\right\rceil_{#2}}

\usepackage{imakeidx}
\makeindex[columns=3,intoc]

\def\III#1{\index{#1@$#1$}{\color{ocre}#1}}
\def\II#1@#2{\index{#1@$#2$}{{\color{ocre}#2}}}

\newcommand{\spannFnull}{\spann<F_0>}
\newcommand{\spannFp}{\spann<F_p>}

\newcommand{\tD}{\ensuremath{\mathrm{D}}}
\newcommand{\Cy}{\mathrm C}
\newcommand{\tC}{\mathrm C}
\newcommand{\tB}{\mathrm B}
\newcommand{\cE}{\mathcal E}

\newcommand{\calM}{\mathcal M}
\newcommand{\calL}{\mathcal L}

\newcommand{\calC}{\mathcal C}

\newcommand{\al}{{\alpha}}
\newcommand{\eps}{{\epsilon}}
\newcommand{\subsubset}{\subset \subset}

\newcommand{\spannh}{\spann<h_0>}

\newcommand{\FF}{{\mathbb{F}}}
\newcommand{\FFtimes}{{\mathbb{F}^\times}}
\newcommand{\si}{\ensuremath{\sigma}}

\newcommand{\GF}{{{\bG^F}}}
\newcommand{\oGF}{{{\obG^F}}}

\newcommand{\KF}{{{\bK^F}}}

\newcommand{\wGF}{{{{\wbG}^F}}}

\newcommand{\cO}{{\mathcal O}}

\makeatletter
\def\Set#1{\Set@h#1@}
\def\Lset#1{\Lset@h#1@}
\def\Set@h#1|#2@{\left\{\left.#1\vphantom{#2}\hskip.1em\,\right|\,\relax #2\right\}}
\def\Lset@h#1@{\left\{#1\right\}}
\def\CALC#1{\CALC@h#1@}
\def\CALC@h#1|#2@{\calC^{#1}(#2)}
\def\CALCrad#1{\CALCrad@h#1@}
\def\CALCrad@h#1|#2@{\calC_\radic^{#1}(#2)}

\def\CALCNC#1{\CALCNC@h#1@}
\def\CALCNC@h#1|#2@{\calC_{\radic,nc}^{#1}(#2)}

\def\restr#1|#2{\left.#1\right\rceil_{#2}}
\def\spann<#1>{\left\langle#1\right\rangle}

\def\Spann<#1>{\Spann@h#1@}
\def\Spann@h#1|#2@{\left\langle\left.#1\vphantom{#2}\hskip.1em\right.\mid\relax #2 \right\rangle}
\def\Set#1{\Set@h#1@}
\def\Set@h#1|#2@{\left\{\left.#1\vphantom{#2}\hskip.1em\,\right.
 \mid\relax #2\right\}}
\def\set#1{\set@h#1@}
\def\set@h#1@{\left\{#1\right\}}
\def\spann<#1>{\left\langle#1\right\rangle}
\makeatother

\newcommand{\uad}{{\underline{a_d}}}

\newcommand{\Aut}{\mathrm{Aut}}

\newcommand{\Out}{\ensuremath{\mathrm{Out}}}

\newcommand{\Z}{\operatorname Z}

\newcommand{\calP}{\mathcal P}

\newcommand{\DD}{\mathbb D}

\newcommand{\forevery}{{\text{\quad\quad for every }}}
\newcommand{\und}{{\text{ and }}}
\newcommand{\with}{{\text{ with }}}

\newcommand{\ra}{\rightarrow}
\newcommand{\lra}{\longrightarrow}

\newcommand{\tA}{\mathrm A}

\newcommand{\Norm}{\operatorname{N}}

\newcommand{\cusp}{\Irr_{cusp}}
\newcommand{\wrt}{{with respect to\ }}
 
\titlerunning{Inductive McKay Condition in type D, I}
\title{Extensions of characters in type D and the inductive McKay condition, I}

\author{Britta Sp\"ath\thanks{This material is partially based upon work supported by the National Science Foundation under Grant No. DMS-1440140 while the author was in residence at the Mathematical Sciences Research Institute in Berkeley, California, during the Spring 2018 semester.
Some research was conducted in the framework of the research training group
\emph{GRK 2240: Algebro-Geometric Methods in Algebra, Arithmetic and Topology}, funded by the DFG. 
}}
\begin{document} 
\maketitle

\abstract{This is a contribution to the study of $\Irr(G)$ as an $\Aut(G)$-set for $G$ a finite quasisimple group. Focusing on the last open case of groups of Lie type $\tD$ and $^2\tD$, a crucial property is the so-called $A'(\infty)$ condition expressing that diagonal automorphisms and graph-field automorphisms of $G$ have transversal orbits in $\Irr(G)$. This is part of the stronger $A(\infty)$ condition introduced in the context of the reduction of the McKay conjecture to a question about quasisimple groups. Our main theorem is that a minimal counter-example to condition $A(\infty)$ for groups of type $\tD$ would still satisfy $A'(\infty)$.  This will be used in a second paper to fully establish $A(\infty)$ for any type and rank. The present paper uses Harish-Chandra induction as a parametrization tool. We give a new, more effective proof of the theorem of Geck and Lusztig ensuring that cuspidal characters of any standard Levi subgroup of $G=\tD_{ l,\mathrm{sc}}(q)$ extend to their stabilizers in the normalizer of that Levi subgroup. This allows us to control the action of automorphisms on these extensions. From there, Harish-Chandra theory leads naturally to a detailed study of associated relative Weyl groups and other extendibility problems in that context.}

\address{School of Mathematics and Natural Sciences University of Wuppertal, Gau\ss str. 20, 42119 Wuppertal, Germany, \email{bspaeth@uni-wuppertal.de}}\subjclass{ 20C20 (20C33 20C34)}
{\small{\tableofcontents}}

\section{Introduction}

After the classification of finite simple groups and with the knowledge on their representations having also greatly expanded in the last decades it seems overdue to determine for each quasisimple group $G$ the action of its outer automorphism group $\Out(G)$ on its set of irreducible (complex) characters $\Irr(G)$. This is important in order to use our results on representations of simple groups to get theorems about arbitrary finite groups. A crucial example is the McKay conjecture asserting $$|\Irr_{p'}(X)|=|\Irr_{p'}(\NNN_{X}(P))| $$ for $p$ a prime, $X$ a finite group, $P$ one of its Sylow $p$-subgroups and $\Irr_{p'}(X)$ the set of irreducible characters of $X$ of degree prime to $p$. It is fairly clear that once this is solved for a normal subgroup $Y$ of $X$ the next step to deduce something for $X$ is to determine the action of $X$ on at least $\Irr_{p'}(Y)$. The McKay conjecture  has been reduced to a so-called inductive McKay condition about quasisimple groups by Isaacs--Malle--Navarro \cite{IMN} and the first requirement is an Out$(X)$-equivariant bijection realizing McKay's equality. Knowing the action of $\Out(G)$ on $\Irr(G)$ for all quasisimple groups $G$ would also have applications to other conjectures about characters with similar reductions such as the Alperin--McKay conjecture or the Dade conjecture (see \cite{S13}, \cite{S17}) or even conjectures about modular characters (see \cite{NT11}) through the unitriangularity of decomposition matrices (see \cite{BDT20}).

For alternating and sporadic groups the action of $\Out(G)$ on $\Irr(G)$ is easy to deduce from the available description of $\Irr(G)$. When $G$ is the universal covering group of a finite simple group of Lie type, this is a question in \cite[A.9]{GM}. Previous research on the subject has left open only the case of groups of type $\tD$, see \cite[2.5]{CS18B}. The present paper is the first part of a solution to that problem. A second part \cite{S21} will finish the determination of $\Irr(G)$ as an $\Out(G)$-set. The splitting is due to the quite different methods used here and in \cite{S21}. 
A third part will focus on applications to the McKay conjecture \cite{S21b}.

In order to be more specific about intermediate goals and results, let us introduce some notation. Let $G=\bG^F$ for $F\colon\bG\to\bG$ a Frobenius endomorphism of a simply-connected simple algebraic group $\bG$. Upon choosing an $F$-stable maximal torus and a Borel subgroup containing it, one can define a group $E$ of so-called field and graph automorphisms of $G$. One can also define a reductive group $\wbG$ realizing a regular embedding for $\bG$, that is $\bG=[\wbG ,\wbG]$ with connected $\ZZZ (\wbG)$ and also assume that $F$ extends to a Frobenius endomorphism of $\wbG$ with $E$ also acting on $\wG:=\wbG^F$. Then $\Aut(G)$ is induced by the direct product $\wG\rtimes E$ (see for instance \cite[2.5.12]{GLS3}). 

The determination of the action of $\wG\rtimes E$ on $\Irr(G)$ mostly relies on establishing that $\wG$-orbits and $E$-orbits are somehow transversal. More precisely, one aims at showing the following property

\noindent $A'({\infty})$ : {\it  There exists an $E$-stable $\wG$-transversal in $\Irr(G)$. }

This, combined with the present knowledge of $\Irr(\wG)$, is enough to determine $\Irr(G)$ as an $\Out(G)$-set (see \cite[2.5]{CS18B}). However, in order to deduce any valuable statement about representations of almost-simple groups, it is also important to answer extendibility questions. For instance a difficult theorem of Lusztig essentially focusing on the case of type $\tD$ shows that any element of $\Irr(G)$ extends to its stabilizer in $\wG$ (see \cite{L88}, \cite{L08}). This notably leads to the determination of the action of $E$ on the set of $\wG$-orbits in $\Irr (G)$. 

The following strengthening of $A'({\infty})$ was introduced in \cite{S12} in order to check the inductive McKay condition for the defining characteristic.

\noindent $A({\infty})$ : {\it  There exists an $E$-stable $\wG$-transversal $\mathbb T$ in $\Irr(G)$ and any $\chi\in\mathbb T$ extends to an irreducible character of its stabilizer $G\rtimes E_\chi$. }

The aim of the present paper and its sequel  \cite{S21} is to prove $A({\infty})$ for $G$ of type $\tD$ and $^2\tD$. In the present paper $G$ will be indeed some $\tD_{ l, \mathrm{sc}}(q)$ ($l\geq 4$, $q$ a power of an odd prime); the case of twisted types $^2\tD$ will also be deduced in  \cite{S21}.  

Our main theorem here can be seen as showing that a putative counterexample to $A({\infty})$ with minimal $l$ still satisfies $A'({\infty})$. 

\begin{theo}\label{thm1}	Let $G= \tD_{l,sc}(q)$ ($l\geq 4$, $q$ a power of an odd prime), let $\wt G$ and $E$ as above (see also \ref{not}). 
	If any $\tD_{l',sc}(q)$ for $4\leq l'<l$ satisfies $A({\infty})$, then $G$ satisfies $A'({\infty})$.
\end{theo}

More precisely we assume \Cref{hyp_cuspD_ext}, i.e. that condition A$(\infty)$ holds for the cuspidal characters of any $G'=\tD_{l',sc}(q)$ with $4\leq l'<l$.

Our proof uses as a starting point a theorem of Malle \cite{Ma17} showing the existence part $A'({\infty})$ of the above statement for cuspidal characters. Then, our strategy is through the parametrization of $\Irr(G)$ given by Harish-Chandra theory. In particular we take the standard Levi supplement $\bL$ of an $F$-stable parabolic subgroup $\bP$ containing our chosen Borel subgroup and consider parabolic induction R$^\bG_{\bL}\lambda$ of cuspidal characters $\lambda\in\cusp (\bL^F)$. 

An essential ingredient of that parametrization is the deep theorem by Lusztig and Geck (\cite[8.6]{L} and \cite[Cor.~2]{GeckHC}) that any $\lambda\in \cusp (\bL^F)$ extends to its stabilizer in $N:=\NNN_{\bG}(\bL)^F$. In order to put that parametrization to use for our purpose of 
tracking automorphism actions  it is important to find an equivariant version of that statement. This  does not seem possible from the available proofs, so we devise a new one in this paper, showing namely
with the same notation for $G=  \tD_{l,sc}(q)$, $\wG$, $\bL$, $N$, $E$

\begin{theo} \label{theoC} Let $\la\in\cusp(\bL^F)$. Assume 
 \Cref{hyp_cuspD_ext} holds for $\tD_{l',sc}(q)$ ($4\leq l'<l$). Then some ${(\ZZZ(\wbG)\bL)}^F$-conjugate $\la_0$ of $\la$ has an $(N\NNN_E(\bL))_{\la_0} $-stable extension to $N_{\la_0}$.
\end{theo}

Studying the group structure of $N$, our proof uses essentially the Steinberg relations for the structure of $G$, not its realization as spin group, making probably more uniform a case-by-case but effective proof of Geck--Lusztig's theorem for other quasisimple groups of Lie type (see \cite[4.3]{BS} and \cite[4.13]{CSS} for types A and C). 

We should point out that the above extendibility property is part of the following wider problem where $(\bH ,F)$ is a reductive group defined over a finite field and $F$ is its associated Frobenius endomorphism.

\noindent (P) {\it Let $\bS$ be an $F$-stable torus of $\bH$. Does every $\psi\in\Irr(\Cent_{\bH}(\bS)^F)$ extend to its stabilizer in $\NNN_{\bH}(\bS)^F$ ?}

This was answered in the affirmative in the case where $\bS$ is a Sylow $d$-torus ($d\geq 1$) in the sense of \cite[25.6]{MT}, see \cite{S09,S10a,S10b}. Lusztig's theorem on the case where $\bS$ is split and $\psi$ is cuspidal was important in \cite{L} to turn Deligne--Lusztig theory into a parametrization of $\Irr(\bH^F)$ when $\bH$ has connected center. It seems that even partial answers to (P) have quite interesting applications (see also \cite[\S 15]{Cedric}, \cite[2.9]{Ma14}).
 
Let for now $\cusp(N)$ be the set of characters of $N$ whose restriction to $L$ is a sum of cuspidal characters.
Theorem B then can be seen as the starting point of a specific parametrization of $\cusp(N)$ bearing similarities with the parametrizations of characters of normalizers of Sylow $d$-tori given in the author's work just mentioned but with a special emphasis on outer automorphism actions. 

Through preparations gathered at the start of the paper and similar to a method developed in \cite{MS16} where $\bL$ was a torus, our main goal Theorem A reduces to a weak analogue of it for $\cusp(N)$. This is Theorem~\ref{thm_loc*}. It is checked through a strategy prescribed by Clifford theory. In particular this entails a quite detailed analysis of the relative Weyl groups $$W(\lambda) :=N_\lambda /\bL^F$$ and their various embeddings related to $\wG$ and $E$.

\subsection{Structure of the paper.} In Section \ref{sec2} we recall notation on quasisimple groups of Lie type, their automorphisms and the conditions $A({\infty})$, $A'({\infty})$. Then we collect the basic facts about cuspidal characters and Harish-Chandra theory for finite groups of Lie type. This leads to Theorem~\ref{thm_MS} which sums up the methods from \cite{MS16} to establish condition $A({\infty})$ through Harish-Chandra theory.
This is roughly the road map for the rest of the paper, in particular splitting the task into two halves that will be addressed in Sections \ref{secext}-\ref{sec4_neu} and Sections \ref{secWla}-\ref{ssec6E}.

The rest of the paper is specific to type $\tD$ (untwisted) in odd characteristic. After recalling a method from \cite{CSS} for constructing extensions, the main objective of Sections \ref{secext}-\ref{sec4_neu} is Theorem~B. Section \ref{secext} is a description of certain group theoretical aspects of the groups $L:=\bL^F$ and $N$, using also the classic embedding $\bG\leq\ov\bG$ of type $\tD_l$ into type $\tB_l$. 
The root system $\Phi'$ of $\bL$ is the direct product of irreducible root systems of type $\tA$ and type $\tD$. Roughly speaking, the factors of type $\tA_{d-1}$ form a root system  $\Phi_d$ and the factor of type $\tD$ gives $\Phi_{-1}$. Along the way we introduce a set $\DD$ determining the types occurring as factors of $\Phi'$. This description will be used in the whole paper. For each $d\in\DD$ we describe a normal inclusion $H_d\lhd V_d\leq \ov N:=\NNN_{\ov\bG}(\bL)^F$ where $H_d=L\cap V_d$ is an elementary abelian $2$-group and $L\spann<V_d\mid d\in\DD>=\ov N$. This normal inclusion $H_d\lhd V_d$ concentrates the equivariant extendibility problem we have to solve. 

In Section \ref{sec4_neu} we draw the consequences of the structure of $N$  in terms of characters. One has to take care of all the factors involved and deal with the inclusion in type $\tB$ which provides the graph automorphism specific to type $\tD$. Concerning the diagonal automorphisms we avoid choosing a regular embedding $\wG$ and instead consider inclusions $L\lhd \calL^{-1}(Z)\cap \bL$ where $Z\leq\ZZZ(\bG)$ and $\calL$ is the Lang map $x\mapsto x^{-1}F(x)$ on $\bG$.

Theorem~B being proved, we study in Section \ref{sec4} how automorphisms act on cuspidal characters in types $\tA$ and $\tD$, making use in the latter case of Malle's theorem \cite{Ma17} mentioned above and some results about semisimple characters already used in the study of the McKay conjecture for the defining characteristic (see \cite[\S 8]{Maslowski}).

In Section \ref{secWla}, the most technical of the paper, the objective is to prove Theorem~\ref{thm_loc*}, showing that $\cusp(N)$ satisfies a version of $A({\infty})$. As already shown in Section \ref{sec2}, this translates into requirements on $\Irr(N_\la/L)$, the characters of the relative Weyl group $W(\la)$ associated to a cuspidal character $\la$ of $L$. The comparison of the action of diagonal versus graph-field automorphisms on $\cusp(N)$ relates with the induced action of related characters of relative Weyl groups. The proof splits  naturally into the various cases for the stabilizer of $\lambda$ in $L\cap\calL^{-1}(\ZZZ(\bG))/L$. This leads to Proposition~\ref{prop_sec4B} and Proposition~\ref{prop5mixed} describing the situation in the two main cases. In the proofs graph-field automorphisms are taken care of by embedding the relative Weyl group $W(\lambda)$ into overgroups $K(\lambda)$ and $\wh K(\lambda)$ (see Notation~\ref{not52}) for field automorphisms and the embedding into type $\tB$ for the graph automorphism of order 2. 

In Section \ref{ssec6E} we essentially put together all the material of the preceding section to establish Theorem~\ref{thm_loc*} and with some extra effort Theorem~A.

\medskip\noindent{\bf Acknowledgements.} This material is partly based upon work supported by the NSF under Grant DMS-1440140 while the author was in residence at the MSRI, Berkeley CA. Some research was conducted in the framework of the research training group
\emph{GRK 2240: Algebro-Geometric Methods in Algebra, Arithmetic and Topology}, funded by the DFG. I thank Julian Brough, Lucas Ruhstorfer, Gunter Malle and the anonymous referee for their remarks on versions of the paper.

\section{Basic Considerations}\label{sec2}
We first gather here some notation around characters, recall Condition $A(\infty)$ and give a rephrasement that provides alternative approaches for the proof of \Cref{thm1}. In \ref{ssec2A} we collect relevant results from Harish-Chandra theory. We conclude with general considerations on cuspidal characters in 1.C. 

\subsection{Notation and Condition $A(\infty)$}
In general we follow the notation about characters as introduced in \cite{Isa}. Additionally we use some terminology from \cite{S09, S10a, S10b} that is recalled in the following paragraph.
\begin{notation}%Character theory
\index{maximal extendibility}\index{extension map}
\index{extendibility! maximal}\index{map! extension}
Let $X\lhd Y$ be finite groups and $\TT\subseteq \Irr(X)$. An \textit{extension map \wrt $X\lhd Y$ for $\TT$} is a map $\Lambda:\TT\lra \coprod_{X\leq I \leq Y } \Irr(I)$ such that every $\la\in\TT$ is mapped to an extension of $\la$ to $\II Ylambda@{Y_\la}$ the inertia subgroup of $\lambda$ in $Y$. We say that \textit{maximal extendibility holds \wrt $X\lhd Y$ for $\TT$ }if such an extension map exists, see also \cite[Def.~5.7]{CS17A}. In such a case the map can be chosen $Y$-equivariant, provided $\TT$ is $Y$-stable, see \cite[Thm.~4.1]{CS17A}.
Whenever $\TT=\Irr(X)$ we omit to mention $\TT$. For $\la\in\Irr(X)$ and $\psi\in\Irr(Y)$ we write $\II lambdaY@{\lambda^Y}$ for the character induced \index{induced character} to $Y$ and $\II psiX@{\psi \rceil_X}$ for the {restricted character}\index{restricted character}. \index{character! restricted}\index{character! induced} For any generalized character $\kappa$ we denote by $\II Irrkappa@{\Irr(\kappa)}$ the set of (irreducible) constituents of $\kappa$. If $\si\in\Aut(X)$ and $\la\in\Irr(X)$ we write 
$\II lambdasigma@{\lambda^\sigma}= \,^{\sigma^{-1} }\la$ for the character with $\,^{\sigma^{-1}} \la(x)= \la^\si(x)=\la(\si(x))$ for $x\in X$.

If two subgroups $H_1,H_2\leq Y$ satisfy $[H_1,H_2]=1$, and $\la_i\in \Irr(H_i)$ for $i=1, 2$ with $\Irr(\restr \la_1|{H_1\cap H_2})=\Irr(\restr \la_2|{H_1\cap H_2})$, then there exists a unique character $\phi\in\Irr(\spann<H_1,H_2>)$ with $\Irr(\restr\phi|{H_i})=\{\la_i\}$ according to \cite[\S 5]{IMN} and we write $\II lambda1@{\lambda_1\cdot \lambda_2}$ for this character. 
Let $\mathbb I$ be a finite set, $Z$, $H$ and $H_i$ ($i\in \mathbb I$) finite groups with $Z\leq H_i\leq H$. If $[H_i,H_{i'}]=1$ for every $i,i'\in \mathbb I$ with $i\neq i'$ and $H_i\cap \spann<H_{j}\mid j\in \mathbb I\setminus\{i\}>=Z$, we consider $  \spann<H_{i}\mid i\in \mathbb I> \leq H$ the central product of the groups $H_i$. Given $\nu\in\Irr(Z)$ and $\la_i\in\Irr(H_i\mid \nu)$ we denote by 
$\II OiIlambdai@{\odot_{i \in\mathbb I}\lambda_i}
\in\protect\Irr(\spann<H_{i}\mid i\in \mathbb I> )$ the character  $\phi\in\Irr(\spann<H_{i}\mid i\in \mathbb I> )$ with $\Irr(\restr \phi|{H_i})=\{\la_i\}$ for every $i\in\mathbb I$, see also \cite[\S 5]{IMN}. 
\end{notation}
Next we introduce the groups and automorphisms considered in the following.
\begin{notation}[Simple Groups of Lie type]\label{not}
 Let $\II G@\bG$ be a simple linear algebraic group of simply-connected type over an algebraic closure $\II FF@{\FF}$ of $\FF_p$ for $p$ a prime. Additionally let $\III F:\bG\ra \bG$ be a Frobenius endomorphism defining an $\FF_q$-structure on $\bG$ for $q$ a power of $p$. The automorphisms of $\GF$ are restrictions to $\GF$ of bijective endomorphisms of $\bG$ commuting to $F$ (see \cite[\S 1.15]{GLS3}), so it makes sense to consider stabilizers $\Aut(\GF)_\bH$ for $F$-stable subgroups $\bH\leq \bG$.Let $\II T0@{\bT_0}$ be an $F$-stable maximally split torus and $\II B@{\bB}$ be an $F$-stable Borel subgroup of $\bG$ with $\bT_0\subseteq \bB$ and $\II Nb0@{\bN_0}:=\NNN_\bG(\bT_0)$. 
According to \cite[Thm.~24.11]{MT} the group $\III G:=\GF$ has a split $BN$-pair with respect to $\II B@ B:=\bB^F$, $\II T0@ {T_0}:=\bT_0^F$ and $\II N0 @{N_0}:=\bN_0^F$. Let $\II EGF@{E(\GF)}$, often just $\III E$, be the subgroup of $\Aut(\GF)_{(\bB,\bT_0)}$ generated by the restrictions to $\GF$ of graph automorphisms and some Frobenius endomorphism $\II F0 @{F_0}$ stabilizing $\bT_0$ and $\bB$ as in \cite[Thm.~2.5.1]{GLS3} and \cite[\S 2.A]{CS18B}.

Let ${\bG\leq \pwt\bG}$ be a regular embedding, i.e., a closed inclusion of algebraic groups with $\II{Gtilde}@{\pwt \bG}=\Z( \wbG)\bG$ and connected $\Z( \wt\bG)$. Then $\II T0tilde@{\pwt \bT_0}:=\Z(\wbG)\bT_0$ is a maximal torus of $\wt\bG$. Let $\II T0tilde@{\pwt T_0}:=\wt \bT_0^F$. Assume that $F:\wbG\ra \wbG$ is a Frobenius endomorphism extending the one of $\bG$, see also \cite[\S 2]{MS16}. Then $\wbG^F$ has again a split $BN$-pair with respect to the groups $\II Btilde @{\pwt B}:= \wt T_0 B$ and $\II N0tilde @{\pwt N'_0}:=\wt T_0 N_0$, see \cite[Thm.~24.11]{MT}. 
Often the action of $\wt N'_0$ on $\bG$ will be studied via the group $\II{Ntilde0}@{\pwt N_0}:=\{x\in \NNN_\bG(\bT_0)\mid x\inv F(x) \in \Z(\bG)\}$, which will be shown to induce the same automorphisms on $\bG$, see \ref{rem_whG}.

Via the convention given in \cite[\S 2]{MS16}, $E(\GF)$ also acts on $\wbG^F$ and the semi-direct product $\wbG^F\rtimes E(\GF)$ induces on $\GF$ the whole automorphism group $\Aut(\GF)$. 
\end{notation}
We recall the conditions $\II Ainf@{A(\infty)} $ and $\II A'inf@{A'(\infty)}$ from \cite[Def.~2.2]{CS18B}:
\begin{condi}[On stabilizers of irreducible characters of $\GF$]\label{recallAinfty} 
%\index{$A(\infty)$}\index{$A'(\infty)$}  
\noindent
\begin{itemize}
\item[{{$A(\infty)$}}:] There exists some $E$-stable $\wGF$-transversal $\TT$ in $\Irr(\GF)$, such that every $\chi\in\TT$ extends to $\GF E_\chi$.
\item [{$A'(\infty)$}:] There exists some $E$-stable $\wGF$-transversal $\TT$ in $\Irr(\GF)$.
\end{itemize}
\end{condi}
\noindent
Condition $A'(\infty)$ implies a weak version of Assumption 2.12(v) of \cite{S12}. 
\begin{lem}\label{*cond_trans}
Let $\wh Y$ and $\wt X$ be two subgroups of a group $Z$ with $\wt X\lhd Z$ and $Z=\wh Y \wt X$. For $X:=\wt X\cap \wh Y$ let $\calM\subseteq \Irr(X)$ be $Z$-stable. Then the following are equivalent: 
\begin{asslist}%
\item there is a $\wh Y$-stable $\wt X$-transversal $\calM_0$ in $\calM$;
\item every $\zeta'\in\calM$ is $\wt X$-conjugate to some $\zeta$ such that $( \wt X \wh Y)_\zeta= \wt X_\zeta \wh Y_\zeta$;
\item every $\zeta'\in \calM$ satisfies $( \wt X\wh Y )_{\zeta'}=(\wh Y^x)_{\zeta'} \wt X_{\zeta'}$ for some $x\in \wt X$.
\end{asslist}
\end{lem}
\begin{proof} This follows from Remark 3.3 of \cite{CSS}. 
\end{proof}

\subsection{Action of $\Aut(G)$ on Harish-Chandra induced characters}\label{ssec2A}
Using detailed analysis of Harish-Chandra induction, the results of \cite{MS16} describe the action of $\Aut(\GF)$ in terms of cuspidal characters and their relative Weyl groups. The action is expressed in terms of the labels given by Howlett--Lehrer theory.

\begin{notation} \label{not_15}
Let $\II L@ L$ be a standard Levi subgroup of $G$ with respect to $B$ and $T_0$, i.e. $L=\bL^F$ for some standard Levi subgroup $\II Lbold@{\bL}$ of $\bG$ such that $\bT_0\leq \bL$ and $\bL \bB$ is an $F$-stable parabolic subgroup. We set $\II N@ N:=\NNN_\bG(\bL)^F$, $\III W:=N/L$ and we abbreviate $$\II EL@{\EL}={E(\GF)}_{\bL}.$$ We write $\II IrrcuspL@{\protect \cusp(L)}$ for the set of cuspidal characters of $L$ as defined in \cite[9.1]{Ca85} and $\II IrrcuspN@{\protect\cusp(N)}:=\bigcup_{\la\in\cusp(L)}\Irr(\la^N)$. Let us denote by $\II RLG@{\R_L^G}$ the \index{Harish-Chandra induction} {Harish-Chandra induction}
from $L$ to $G$. For $\lambda\in\cusp(L)$ let $$\II IrrGLla@{\Irr(G\mid (L,\la))}:=\Irr(\R_L^G(\la)),$$ (sometimes denoted as ${\cE(G,(L,\lambda))}$ in the literature). Let also $\II IrrGLM@{\Irr(G\mid (L,\TT))}:=\bigcup_{\la\in\TT}\Irr(G\mid (L,\la))$ for $\TT\subseteq \cusp(L)$.
\end{notation}

\begin{num}\label{ssec2Anot}
Let $\II AutGFLHC@{\Aut(\GF)_{L,\HC}}$ be the subgroup of $\Aut(\GF)$ generated by the automorphisms of $\GF$ induced by $N$ and $\Aut(\GF)_{(\bB\bL,\bL)}$. Note $\EL\leq \Aut(\GF)_{L,\HC}$. According to Howlett--Lehrer theory (see \cite[\S 10]{Ca85}), fixing an extension $\wt \la\in\Irr(N_\la)$ of $\la\in\cusp(L)$ defines a unique labelling of $\Irr(G\mid (L,\la))$ by $\Irr(W(\la))$ where $\II Wlambda@{W(\lambda)}:=N_\la/L$. We write $\II RLGlaeta@{\R_L^G(\la)_\eta}$ for the character of $\Irr(G\mid (L,\la))$ associated to $\eta\in\Irr(W(\la))$ via the extension $\wt \la$.

Accordingly the parametrization of $\Irr(G\mid (L,\cusp(L)))$ depends on an extension map $\II LambdaL@{\Lambda_L}$ \wrt $L\lhd N$ for $\cusp(L)$. For $\la\in\cusp(L)$ let $\III{R(\lambda)}\lhd W(\la)$ be defined as in \cite[Prop.~10.6.3]{Ca85}. If $\la\in\cusp(L)$ and $\si\in \Aut(\GF)_{L,\HC}$ let $\II deltalambdasigma@{\delta_{\la,\si}}$ be the unique linear character of $W(^\si\la)$ satisfying 
\begin{align}\label{deltalasi}
^\si\Lambda_L(\la)&=\Lambda_L(^\si\la)\delta_{\la,\si}.
\end{align} 
\end{num}
\noindent We only use the formula with some simplifying assumptions on $R(\la)$ and $\delta_{\la,\si}$. 

\begin{thmbox}[Malle--Sp\"ath {\cite[Thm.~4.6 and 4.7]{MS16}}]\label{HC_form}
Let %$\II T@\TT\subseteq \cusp(L)$ be $\wh N $-stable and 
$\si\in\Aut(\GF)_{L,\HC}$ and $\Lambda_L$ be an $N$-equivariant extension map \wrt $L\lhd N$ for $\cusp(L)$. Assume that $\R_L^G(\la)_{\eta}$ ($\la\in\cusp(L)$, $\eta\in\Irr(W(\la))$) is defined using $\Lambda_L$ and
\begin{align}\label{incl}
R(^\si\la) \leq\ker(\delta_{\la,\si}) \forevery \la \in\cusp(L) .
\end{align}
Then $^{\si}(\R_L^G(\la)_{\eta})= \R_L^G({}^{\si}\la)_{^\si\eta \delta_{\la,\si}^{-1}}$
for every $\la\in\TT$ and $\eta\in\Irr(W(\la))$. 
\end{thmbox}
In Section 5 of \cite{MS16}, the analogue of \Cref{thm1} was proven for characters in $\Irr(G \mid (T_0, \cusp(T_0))$ by studying $\cusp(\NNN_\bG(\bT_0)^F)$. For other standard Levi subgroups the strategy from \cite{MS16} leads naturally to the following statement where we focus on a single $L$ and its stabilizer in $E$. Sections \ref{secext}--\ref{secWla} will ensure the assumptions for the groups from \ref{not} whenever $\GF =\tD_{ l, \mathrm{sc}}(q)$. 
\begin{theorem}\label{thm_MS}
\label{not_L_N_E1}
Let $\II Ltilde @{\pwt L'}:=\wt T_0 L$, $\II Ntilde@{\pwt N'}:=\wt T_0 N$ and $\II Nhat@{\phat N}:=N\EL$. 
Assume there exist
\begin{asslist}
\item \label{ext_map} 
an $\wh N$-stable ${\wt L'}$-transversal $\TT$ in $\cusp(L)$,
an $N$-equivariant extension map $\Lambda_{L, \TT}$ \wrt $L\lhd N$ for $\TT$ such that any $\la\in\TT$ satisfies Equation \eqref{incl}; and
\item \label{loc*}
some $\EL$-stable ${\wt N'}$-transversal in $\cusp(N)$.
\end{asslist}
Then there exists an $\EL$-stable $\wbG^F$-transversal in $\Irr(\GF\mid(L,\cusp(L)))$. 
\end{theorem}
For the proof of Theorem \ref{thm_MS} we parametrize $\cusp(N)$ via a set $\calP(L)$ using an extension map $\Lambda_L$ with respect to $L\lhd N$ for $\cusp(L)$ deduced from $\Lambda_{L, \TT}$.
\begin{notation}\label{not29}
Assume that $\TT$ is an $\wh N$-stable ${\wt L'}$-transversal in $\cusp(L)$.
For each $\la\in \TT$ we denote by $\cO_\la$ its $N$-orbit in $\cusp(L)$. Note $\cO_\la\subseteq \cusp(L)$. Let $M(\la)\subseteq {\wt L'}$ be a set of representatives of the ${\wt L'}_\la$-cosets in $\wt L'$. We define an extension map $\Lambda_{L}$ on $\cO_\la$ by 
$$\Lambda_L(\la'^m)= \Lambda_{L,\TT}(\la')^m \forevery \la'\in \cO_\la \und m\in M(\la).$$
Hence, $\Lambda_L$ is defined, but depends on the choice of $M(\la)$. 
The map $\Lambda_L': \cusp(L)\lra \coprod_{L\leq I \leq N} \Irr(I)$ with $\Lambda_L'(\mu):=\restr \Lambda_L(\mu)|{N_{\wt \mu}}$ for every $\mu\in \cusp(L)$ is well-defined, where $\wt \mu\in \Irr({\wt L'}_\mu)$ is an extension of $\mu$. In contrast to $\Lambda_L$ we see that $\Lambda_L'$ is independent of the choice of $M(\la)$. Observe $[N/L, {\wt L'}/L]=1$. The map $\Lambda_L'$ is even $\wh N {\wt L'}$-equivariant since $\Lambda_L$ is $N$-equivariant and $\Lambda_{L,\TT}$ is $\wh N$-equivariant. 

We write $\II PcircL@{\calP '(L)}$ for the set of pairs $(\la,\eta)$ with $\la\in\cusp(L)$ and $\eta\in\Irr(W(\la))$. The groups $N$ and $W$ act naturally via conjugation on ${\calP ' (L)}$. We denote by $\II PL@{\calP(L)}$ the set of $N$-orbits in $\calP '(L)$ and by $\II laoveta@{\ov{(\la,\eta)}}$ the $N$-orbit containing $(\la,\eta)$. Since $L$ is mostly clear from the context, we omit it, writing $\II Pcirc@{\calP '}$ and $\II P@{\calP}$.
\end{notation}
The parametrization of $\cusp(N)$ is given by the following.
\begin{prop}\label{propcuspN} 
Let $\Lambda_{L}$, $\calP '$ and $\calP$ be as in  \ref{not_L_N_E1} and \ref{not29}.
%	 be the $N$-equivariant extension map \wrt $L\lhd N$ for $\cusp(L)$
\begin{enumerate}
	\item Then the map
	$$\II UPsilon@\Upsilon: \calP \lra \cusp(N) \with \ov{(\la,\eta)}\longmapsto (\Lambda_L(\la )\eta)^N $$ is a well-defined bijection. 
	\item $^\si\Upsilon(\ov {(\la,\eta)})=\Upsilon(\ov{(^\si\la,{}^\si\eta\delta_{\la,\si })})$ for every $\si\in\Aut(G)_{L,\mathrm{HC} }$ and $\ov {(\la,\eta)}\in \calP$, where $\delta_{\la,\si}\in\Irr(W(^\si\la))$ is as given in \ref{ssec2Anot}.
\end{enumerate}	
\end{prop}
\begin{proof} Clifford theory together with Gallagher's Lemma \cite[6.17]{Isa} proves part (a). The definition of $\delta_{\la,\si}$ in Equation \eqref{deltalasi} from \ref{ssec2Anot} leads to part (b).\end{proof}

In combination with Theorem \ref{HC_form} we obtain a proof of \Cref{thm_MS}.
\begin{proof}[Proof of \Cref{thm_MS}]
For the application of \Cref{HC_form} we have to ensure that under our assumptions the Equation \eqref{incl} holds for characters $\la\in \TT$ and $\si\in\Aut(\GF)_{L,\HC}$. For every $\la\in\cusp(L)$, the character $\Lambda_L(\la)$ is an extension of $\Lambda_L'(\la)$. Accordingly $\delta_{\la,\si}$ defined as the unique linear character of $W(^\si \la)$ such that $^\si\Lambda_L(\la)=\Lambda_L(^\si\la)\delta_{\la,\si}$ satisfies as well $\restr ^\si\Lambda_L(\la)|{N_{\wt \la ^\si }}=\restr \Lambda_L(^\si\la)\, \, \delta_{\la,\si}|{N_{\wt \la ^\si }}$. 
Since $\Lambda_L'(\la)$ is ${\wt N'} \EL$-equivariant, we see that $\restr \delta_{\la,\si}|{N_{\wt \la ^\si }}$ is trivial. Accordingly $\ker(\delta_{\la,\si})\geq N_{^\si \wt \la}/L$ for every $\la\in \cusp(L)$ and $\si\in \Aut(\GF)_{L,\HC}$ where $\wt \la$ denotes an extension of $\la$ to ${\wt L'}_\la$. Recall $W(^\si \wt \la)=N_{^\si \wt \la}/L$. 
In combination with the inclusion $R(^\si \la) \leq W(^\si \wt \la) $ from \cite[Lem.~4.14]{CSS} we obtain the required containment \eqref{incl}. 

Via Harish-Chandra induction the map
\[\II Upsilon'@{\Upsilon'}: \calP \lra \Irr( G\mid (L,\cusp(L)))\with \ov {(\la,\eta)}\longmapsto \R_L^G(\la)_{\eta}\]
is well-defined according to \cite[Thm.~4.7]{MS16} and bijective. Hence $\Upsilon'\circ\Upsilon^{-1}$ is a bijection between $\cusp(N)$ and $\Irr( G\mid (L,\cusp(L)))$. Via $\Upsilon$ and $\Upsilon'$ the group $\Aut(\GF)_{L,\HC}$ and hence ${\wt N'} \EL$ act on $\calP$. By the description of this action given in \ref{HC_form} and \ref{propcuspN} these actions coincide. Hence $\Upsilon'\circ\Upsilon^{-1}$ is ${\wt N'} \EL$-equivariant. 
By Assumption \ref{loc*} every $\psi_0\in\cusp(N)$ has an ${\wt L'}$-conjugate $\psi$ such that $({\wt N'}\EL )_\psi= {\wt N'}_\psi (\EL)_{\psi} $. Hence every $\chi_0\in \Irr(G\mid (L,\cusp(L)))$ has an $\wt N'$-conjugate $\chi$ with $(\wbG^F \EL)_\chi=G(\wt N' \EL )_\chi=G(\wt N'_\chi) (\EL)_{\chi}=\wbG^F_\chi (\EL)_{\chi} $. This implies the statement, see \Cref{*cond_trans}. \end{proof}

In the following sections we verify the assumptions of Theorem \ref{thm_MS}: We prove Assumption \ref{loc*}, i.e., that every $\psi\in\cusp(N)$ is ${\wt L'}$-conjugate to some $\psi_0$ with $({\wt N'}\EL)_{\psi_0}={\wt N'}_{\psi_0} (\EL)_{\psi_0}$, and prove the existence of an extension map as required in \ref{ext_map}. Note that by Lusztig \cite{L} and Geck \cite{GeckHC} an extension map exists. Their proofs are indirect and we don't see how the required properties can be deduced from their proofs. In later sections we give an independent explicit construction of the required extension map.

\subsection{Action on characters of normalizers of Levi subgroups}\label{ssec1C}
In the following we discuss some basic considerations that will be applied to ensure Assumption \ref{loc*}. In the case where $L=T_0$ the assumption \ref{loc*} holds, whenever the underlying group $\GF$ is of simply-connected type, see \cite[Proof of Cor.~5.3]{MS16}. The assumption on the characters $\cusp(N)$ is very similar to the results \cite[Prop.~5.13]{CS17A}, \cite[Thm.~5.1]{CS17C} and \cite[5.E]{CS18B} on $\Irr(\NNN_\bH(\bS)^F)$ for Sylow $\Phi_d$-tori $\bS$ of $(\bH,F)$, where $\bH$ is a simple simply-connected group of type different from $\tD_l$ and $d$ is a positive integer. The proof there relies on Theorem 4.3 of \cite{CS17C} and we use here a similar strategy. The following proposition gives the road map for the verification of Assumption \ref{loc*}.\medskip

We set $\II W phi @{W(\phi )}=N_{\phi }/L$ for every $L\leq M\leq \wt T_0L$ and $\phi\in\Irr(M)$.  
\begin{prop}\label{prop23}
Let $\wh N$, $\wt L'=\wt T_0L$ be as in \ref{thm_MS}, $\TT$ and $\Lambda_{L,\TT}$ as in Assumption \ref{ext_map} and $\Upsilon$ from \Cref{propcuspN}. Let $\la \in \TT $, $\wt \la \in\Irr(\wt L'_ \la\mid \la)$,  $\eta\in\Irr(W(\la))$ and $\eta_0\in \Irr(\restr\eta|{W(\wt\la)})$. We set $\II What @{{\pwh W}}:=\wh N/L =N\EL/L$ and $\II Khatlambda @{\pwh K(\la )}:=\wh W_{\la}$. If $\eta$ is $\wh K(\la)_{\eta_0}$-stable, then $$(\wh N {\wt L'} )_{\Upsilon (\ov{(\la,\eta)})}=\wh N_{\Upsilon (\ov{(\la,\eta)})} {\wt L'}_{\Upsilon (\ov{(\la,\eta)})}.$$
\end{prop}
We adapt the arguments from the proof of \cite[Thm.~4.3]{CS17C}, where $\eta$ is assumed to be  $\NNN_{W\rtimes \EL}(W(\wt \la))_{\eta_0}$-stable. Note that $\wh K(\la)$ normalizes $W(\la)$  but this group is in general different from $\NNN_{W\rtimes \EL}(W(\wt \la))$. 
%\label{prop_TT_alt} \label{prop23ii_alt} 
\begin{proof}
Recall $\psi =\Upsilon (\ov{(\la,\eta)})=(\Lambda_{L,\TT}(\la)\eta)^N$. By the assumptions on $\TT$, $(\wh N {\wt L'})_\la=\wh N_\la {\wt L'}_\la$ for every $\la\in\TT$. 

Let $\wt \la\in\Irr({\wt L'}_\la\mid \la)$ and $\eta_0\in\Irr(\restr\eta |{W(\wt \la)})$. According to \cite[15.11]{CE04}, $\wt \la$ is an extension of $\la$. 
The group ${\wt L'}_\la/(L\Z(\wG))$ acts by multiplication with linear characters of $W(\la)/W(\wt\la)$ on $\Irr(W(\la)\mid \eta_0)$. 
Computing the action of $W(\la)/W(\wt \la)$ on $\Irr(\wt L'_\la\mid \la)$ shows that the action of $\wt L'_\la/L$ on $\Irr(W(\la)\mid \eta_0)$ is transitive. Hence the characters $\{ (\Lambda_L(\la)\eta')^N \mid \eta'\in \Irr(W(\la)\mid \eta_0)\}$ are the ${\wt L'}_\la$-conjugates of $\psi$.

Let $l\in {\wt L'}$ and $\wh n\in \wh N$ with $\psi^{l}=(\psi)^{\wh n}$. Note that $\psi^{\wh n}\in \Irr(N\mid \TT)$ since $\TT$ is $\wh N$-stable. Then $\Irr(\restr \psi^{l}|{L})$ is the $N$-orbit of $\la^{l}$. Recall $\TT$ is an $\wt L'$-transversal. If $l\notin {\wt L'}_\la$ then $\la^{l}\neq \la$ and $\la^{l}\notin \TT$, in particular $\psi^{l}\notin\Irr(N\mid \TT)$. This implies $l\in {\wt L'}_\la$ and $\psi^{l}=(\Lambda_L(\la)\eta \nu)^N$ for some linear character $\nu$ of $\Irr(W(\la)/W(\wt \la))$. Accordingly $(\psi)^{\wh n}\in\Irr(N\mid \la)$ and hence $(\psi)^{\wh n}=(\psi)^{\wh n'}$ for some $\wh n'\in\wh N_\la$. Note that 
$$(\psi)^{\wh n'}=((\Lambda_L(\la)\eta)^{\wh n'})^N= (\Lambda_L(\la) \eta^{\wh n'})^N.$$ 
The equality $\psi^{l}=\psi^{\wh n'}$ implies $\eta^{\wh n'}=\eta\nu$ and hence $\wh n' L\in W(\la) \wh K(\la)_{\eta_0}$. As $\eta$ is $\wh K(\la)_{\eta_0}$-stable, $\eta^{\wh n'}=\eta$ and hence $\psi^{\wh n'}=\psi$. This shows $(\wh N {\wt L'} )_\psi=\wh N_\psi {\wt L'}_{\psi}$.
\end{proof}
\noindent 
The above proposition allows us to prove the following result showing how to construct an $\wh N$-stable $\wt L'$-transversal in $\cusp(N)$.
\begin{prop}\label{prop23_neu}
In the situation of \Cref{not_L_N_E1} assume 
\begin{asslist}
 \item \label{prop_TT} $(\wh N {\wt L'})_\la=\wh N_\la {\wt L'}_\la$ for every $\la\in\TT$, 
 \item \label{prop23ii} there exists an $\wh N$-equivariant extension map $\Lambda_{L,\TT}$ \wrt $L\lhd N$ for $\TT$, and
\item for every $\la\in\TT$, $\wt \la\in\Irr(\wt L'_\la\mid \la)$ and $\eta_0\in \Irr(W(\wt\la))$ there exists some $\wh K(\la)_{\eta_0}$-stable $\eta\in\Irr(W(\la)\mid \eta_0)$.
\end{asslist} 
Let $\ov \TT\subseteq \cusp(L)$ be the set of characters that are ${\wt L'}$-conjugate to one in $\TT$. Then there exists some $\wh N$-stable ${\wt L'}$-transversal in $\Irr(N\mid \ov \TT)$. 
\end{prop}
\begin{proof}
%It is sufficient to check that in each ${\wt L'}$-orbit in $\Irr(N\mid \TT )$ some character $\xi$ satisfies $(\wh N \wt L')_{\xi^t}=(\wh N^x)_{\xi^t} \wt L'_{\xi^t}$. 
By the assumptions there exists $\calP_1 \subseteq \calP$ such that 
\begin{itemize}
\item 
if $\ov{(\la,\eta)} \in \calP_1$, then $\la\in \TT$  and $\eta$ is $\wh K(\la)_{\eta_0}$-stable for some $\wt \la\in\Irr(\wt L'_\la\mid \la)$ and $\eta_0\in\Irr(\restr \eta | {W(\wt \la)})$; and
\item for each $\la\in\TT$,  $\wt \la\in\Irr(\wt L'_\la\mid \la)$ and $\eta_0\in\Irr(W(\wt \la))$ there exists some $\eta\in\Irr(W(\la)\mid \eta_0)$ with $\ov{ (\la,\eta)}\in\calP_1$.
\end{itemize}
\Cref{prop23} tells us that the characters $\Upsilon(\calP_1)$ can form part of an $\EL$-stable ${\wt L'}$-transversal. 

According to \Cref{propcuspN}, for every $\la\in\TT$ and $\eta_0\in\Irr(W(\wt\la))$, the group ${\wt L'}_\la$ acts transitively on the set $\Irr(W(\la)\mid \eta_0)$. Since for each $\la\in\TT$ and $\eta_0\in\Irr(W(\wt \la))$ there exists some $\eta\in\Irr(W(\la)\mid \eta_0)$ such that $\ov{ (\la,\eta)}\in\calP_1$, each ${\wt L'}$-orbit has a non-empty intersection with $\Upsilon(\calP_1)$.
This implies that every character in $\Upsilon(\calP_1)$ has the property required, see \Cref{*cond_trans}. 
\end{proof}

%%%%%%%%%%%%%%%%%%%%%%%%%%%%%%%%%%%%%%%%%%%%%%%%%%%%
\subsection{Reminder on cuspidal characters}
The considerations of \Cref{ssec2A} explain how the action of automorphisms on non-cuspidal characters depends on the underlying cuspidal character and a character of the relative Weyl group associated to a cuspidal pair. For the proof of \Cref{thm1} we require some general results on cuspidal characters that we collect here. By a theorem of Malle, stabilizers of cuspidal characters coincide with those of semisimple characters (see \cite[15.A]{Cedric} for a definition of semisimple characters). 

\begin{thmbox}\label{thm_Malle}
Let $\bH$ be a simply-connected simple linear algebraic group with an $\FF_q$-structure given by a Frobenius map $F:\bH\rightarrow \bH$. Let $\bH \lra \wt \bH$ be a regular embedding and $E(\HF)$ be a group of automorphisms of $\HF$ generated by graph and field automorphisms as in \ref{not}. Then there exists some $E(\HF)$-stable $\wt \bH^F$-transversal in $\cusp( \bH^F)$.
\end{thmbox}
\begin{proof} We abbreviate $E(\HF)$ as $E$. Let $\chi\in\cusp(\bH^F)$.
According to \Cref{*cond_trans} it is sufficient to prove that $\chi$ has some $\wt \bH^F$-conjugate $\chi_0$ with 
$(\wt \bH^F E)_{\chi_0}= \wt \bH^F_{\chi_0} E_{\chi_0}$.
 By \cite[Thm.~1]{Ma17} there exists a semisimple character $\rho$ of $\HF$, such that $\rho$ and $\chi$ have the same stabilizer. By \cite[proof of~3.4(c)]{S12} the semisimple character $\rho$ has some $(\wt \bH)^F$-conjugate $\rho_0$ with $(\wt \bH^F E)_{\rho_0}= \wt \bH^F_{\rho_0} E_{\rho_0}$. 
\end{proof}
In our considerations on $\tD_{l,sc}(q)$ we assume the following for all $4\leq l'<l$, which amounts to $A(\infty)$ for cuspidal characters in rank $<l$. This was called $A_{\mathrm{cusp}}$ in our Introduction. In \cite{S21}  we will see that it is actually always satisfied.
\begin{hyp}[Extension of cuspidal characters of $\tD_{l',sc}(q)$]\label{hyp_cuspD_ext}\label{hyp_cuspD}
Let $\bH$ be a simply-connected simple group of type $\tD_{l'}$ ($l'\geq 4$) and $F:\bH\ra \bH$ a standard Frobenius endomorphism. Then there exists some $E(\HF)$-stable $\wt \bH^F$-transversal $\TT$ in $\cusp(\bH^F)$ such that every $\chi\in \TT$ extends to $\bH^FE(\HF)_\chi$.
\end{hyp}
The following facts are well-known (see also \cite[12.1]{Cedric}). 
\begin{lem}\label{lem_cusp}
	Let $\bG$ be a simply-connected simple group with Frobenius endomorphism $F:\bG\ra\bG$, $\bL$ an $F$-stable Levi subgroup of $\bG$, $L:=\bL^F$, $L_0:=[\bL,\bL]^F$, and $\la\in\cusp(L)$. 
	\begin{enumerate}
		\item Then $\Irr(\restr\la|{L_0})\subseteq\cusp(L_0)$
		\item If $[\bL,\bL]$ is a central product of $F$-stable semisimple groups $\bH_1$ and $\bH_2$, then $\Irr(\restr \la|{\bH_1^F}) \subseteq \cusp(\bH_1^F)$.
		\item \label{lem_cuspc} 
		Let $ \wt \bG$ be a reductive group with $\FF_q$-structure given by $F:\wbG\ra\wbG$ extending $F$ already defined on $\bG$ and such that $[\wt \bG , \wt \bG]=\bG $, then every $\wt \la\in\Irr((\Z(\wbG)\bL)^F\mid \la)$ is cuspidal.
	\end{enumerate}
\end{lem}
\begin{proof}
For a finite group $H$ with a split BN-pair of characteristic $p$, a given $\chi\in\Irr (H)$ is cuspidal if and only if the corresponding representation space has no non-zero fixed point under any O$_p(P)$ for  any proper parabolic subgroup $P$ of $H$. It is then clear that for any $H'\lhd H$ with $p'$-index, one has $\chi\in\cusp (H)$ if and only if $\restr{\chi}|{H'}$ has a cuspidal irreducible component (and then all are). This gives (a), and (c). For (b), note that $\bH_1\cap\bH_2$ is a group of semi simple elements, so that the O$_p(P)$'s as above for $H:=\bL_0^F$ are direct products of corresponding subgroups of $\bH_1^F$ and $\bH_2^F$.
\end{proof}

\begin{remark}\label{rem_whG}
\begin{thmlist}
\item Let $\bG$ be a simply-connected simple group and $\wt \bG$ a connected algebraic group with $\wt \bG=\bG \Z(\wt \bG)$. Let $F:\wt \bG\ra\wbG$ be a Frobenius endomorphism stabilizing $\bG$. Then $x\in\wbG^F$ can be written as $x=gz$ with $g\in \bG$ and $z\in \Z(\wbG)$, such that $g^{-1}F(g)=z F(z^{-1})$. If 
$\II{Lcal}@{\protect \calL}:\bG\ra \bG$ is defined by $g\mapsto g^{-1}F(g)$ and
$\II{Ghat}@{\protect \wt G}:=\calL^{-1}(\Z(\bG))$, we see 
$$\wGF \leq \wt G . \wt Z,$$
 where $\wt Z:=\{z\in \Z(\wbG)\mid F(z)\in z\Z(\bG) \}$. Note that $\wt G$ by its construction is independent of the choice of $\wbG$. We also have $\wt G=\norm{\bG}{\bG^F}$ as an easy consequence of \cite[Lem. 6.1]{Cedric}. 
\item 
\label{labeldiag}From now on we assume additionally that $\Z(\wbG)$ is connected. Then the (outer) automorphisms of $\GF$ induced by conjugation by some element $g\in \wt G$ are called {\it diagonal automorphisms} \index{diagonal automorphism}and they are parametrized by $\calL(g) [\Z(\bG),F]\in  \Z(\bG)/ [\Z(\bG),F]$, see also \cite[1.5.12]{GM}. 

Note the difference with the convention used in the introduction where $\wG$ was used to abbreviate $\wbG^F$. We still clearly have $\wG/\Z(\GF)=\wbG^F/\Z(\wbG^F)$.
\index{automorphism! diagonal}
\end{thmlist}
\end{remark}
This allows the following conclusion for the above group $\wt G$.
\begin{theorem}\label{lem_ext_G_whG}
Maximal extendibility holds \wrt $\bG^F \lhd \wt G$. 
\end{theorem}
\begin{proof}
%	If $\wh G_{\chi}/\GF$ is cyclic, the statement follows from \cite[]{Isa}.
Let $\wbG$ be a group with connected centre, such that there exists a regular embedding $\bG\ra \wbG$ that is also an $\FF_q$-morphism as in \ref{not}. Then according to a theorem of Lusztig (see \cite[15.11]{CE04}) maximal extendibility holds \wrt $\GF\lhd \wGF$, and $\chi$ has an extension $\wt \chi$ to $\wGF_{\chi}$.
According to the above, $\wGF\leq \wt G. \wt Z$. Clearly $\wt \chi$ extends to $\wGF_\chi\wt Z$ since $\wt Z$ is abelian and $[\wt Z, \wGF]=1$. Now we see that $\wGF_{\chi}\wt Z=\wt G_\chi \wt Z$ and hence $\chi$ extends to $\wt G_\chi$ as well. 
\end{proof}

\begin{prop}\label{gammastablewhG_extension}
In the situation of \Cref{rem_whG}, let 
$\bK\leq \bG$ be an $F$-stable reductive subgroup with $\bT_0\leq \bK$. Let $\wt \bK:=\bK \Z(\wt \bG)$ and 
$\wt K:=\calL\inv(\Z(\bG))\cap \bK$.
Let $\chi \in\Irr(\KF)$, $\wt \chi\in\Irr(\wt \bK^F \mid \chi)$ and $\nu\in\Irr(\restr \wt \chi|{\Z(\wt \bG^F)})$. As said above, $\chi$ extends to $\wt\bK^F_\chi$.
Let $\gamma\in E(\GF)_{(\chi,\bK)}$ and $\mu\in\Irr(\wt \bK^F/\KF)$ with $\wt \chi^\gamma=\wt \chi\mu$. 
Then the following are equivalent:
\begin{asslist}
\item $\chi$ has a $\gamma$-stable extension to $\wt K _\chi$.
\item for $\wt Z':=\calL(\wt K_\chi)$ there exists some extension $\wt \nu\in \Irr(\wt Z')$ of $\nu$ such that $\mu(tz)= \wt \nu (z)^{-1} (\wt \nu^\gamma (z){}) $ for every $t\in \wt K_\chi$ and $z \in \wt Z'$ with $tz\in \wt \bK^F_\chi$.
\end{asslist} 
\end{prop}
\begin{proof} We prove the statement only in the case where $\bK=\bG$. 
	The results transfer to a general $\bK$ as only the quotient groups are relevant to our considerations. Let $\wt \chi$ be a $\gamma$-stable extension to $\wt G_\chi$, then there exists an extension $\wt \nu\in \Irr(\wt Z')$ of $\nu$ such that $\wt \chi:=(\restr \wt \chi.\wt \nu|{\wGF_\chi})^{\wGF}$. We observe $(\wt \chi.\wt \nu)^\gamma=\wt \chi.\wt \nu^\gamma$. This leads to the given formula for $\mu$ in (ii). 

For the other direction let $\chi_0$ be the extension of $\chi$ to $\wt\bG^F_\chi$ such that $\wt \chi=\chi_0^{\wbG^F}$. Then $\chi_0^\gamma=\chi_0\restr \mu|{\wGF_\chi}$ and $ \chi_0.\wt\nu$ is an extension of $\chi$ to $\wt \bG_\chi^F\wt Z'=\wt G_\chi \wt Z'$.  The character $\wh\chi:=\restr (\chi_0.\wt\nu)|{\wt G_\chi}$ satisfies 
$$ \wt \chi^\gamma.\wt\nu^\gamma=(\wt \gamma. \wt\nu)^\gamma=
(\chi_0.\wt \nu)^\gamma= \chi_0^\gamma.\wt\nu^\gamma= 
\chi_0 \restr \mu|{\wt\bG^F_\chi}.\wt\nu^\gamma.$$ 
There is some $\kappa\in\Irr(\wt G_\chi/\bG^F)$ with $\wt \chi ^\gamma=\wt \chi \kappa$. According \cite[(6.17)]{Isa} the above equality of characters implies $\kappa(t)\wt \nu^\gamma (z)(\wt \nu (z))^{-1}= \mu (tz)$, whenever $t\in \wt G_\chi$ and $z\in\wt Z'$ with $tz\in \wt\bG^F_\chi$. By the assumption on $\mu$ and $\wt \nu$ this leads to $\kappa=1$. Then $\chi$ has a $\gamma$-stable extension to $\wt G_\chi$. 
\end{proof}
For later we restate $A(\infty)$ for groups of type $\tA$ (see \cite{CS17A}).
\begin{prop}\label{prop43}
Let $G=\SL_n(q)$, $\wt G:=\GL_n(q)$ and write $\II ESLnq@{E(\SL_n(q))} $ for the group of field and graph automorphisms of $G$ and $\wt G$ with regard to the usual BN-pair. 
\begin{thmlist}\item 
Then there exists an $E(\SL_n(q))$-stable $\GL_n(q)$-transversal $\TT$ in $\Irr(\SL_n(q))$, such that every $\chi\in \TT$ extends to $\SL_n(q)E(\SL_n(q))_\chi$.
\item Let $\gamma'$ be the automorphism of $\SL_n(q)$ given by transpose-inverse and 
$E'(\SL_n(q))\leq\Aut(\SL_n(q))$ be the subgroup generated by $\gamma'$ and the field automorphisms described above. Then $E'(\SL_n(q))$ is abelian and  there exists an $E'(\SL_n(q))$-stable $\GL_n(q)$-transversal $\TT$ in $\Irr(\SL_n(q))$, such that every $\chi\in \TT$ extends to $\SL_n(q)E'(\SL_n(q))_\chi$.
\end{thmlist}
\end{prop}
\begin{proof}
Part (a) follows from {\cite[Thm.~4.1]{CS17A}} using \Cref{*cond_trans}. 

Let $\gamma\in E(\SL_n(q))$ be the graph automorphism. Following the considerations in \cite[3.2]{CS17A} we see that $\gamma'$ and $v_0\gamma$ induce the same automorphism of $\SL_n(q)$, where $v_0\in\SL_n(p)$ is defined as in \cite[3.2]{CS17A} and $p$ is the prime dividing $q$.
This proves that $\TT$ is also $E'(\SL_n(q))$-stable. 
For part (b) we have to prove that every $\chi\in\TT$ extends to its inertia group in $\SL_n(q)E'(\SL_n(q))$. This statement is clear whenever $E'(\SL_n(q))_\chi$ is cyclic, see \cite[(9.12)]{Isa}.
If for $\chi\in\TT$ the group $E'(\SL_n(q))_\chi$ is non-cyclic, we see $\gamma'\in E'(\SL_n(q))_\chi$. Let $F_{q'}\in E'(\SL_n(q))$ be a field automorphism such that $E'(\SL_n(q))_\chi =\spann<F_{q'},\gamma'>$. 
By (a) there exists some $\gamma$-stable extension of $\chi$ to $G\spann<F_{q'}>$. This extension is then also $\gamma'$ and hence $\gamma v_0$-stable as $[v_0,F_{q'}]=1$. From this we deduce that $\chi$ extends to $\SL_n(q)E'(\SL_n(q))_\chi$. 
\end{proof}
%%%%%%%%%%%%%%%%%%%%%%%%%%%%%%%%%%%%%%%%%%%%%%%%%%%%%%%%%%%%%%%%%%%%%%%%%%%%%%%%%%%%%%%%%%%%%
\section{The Levi subgroup and its normalizer}\label{secext}
%%%%%%%%%%%%%%%%%%%%%%%%%%%%%%%%%%%%%%%%%%%%%%%%%%%%%%%%%%%%%%%%%%%%%%%%%%%%%%%%%%%%%%%%%%%%%
In this and the following section we reprove with quite different methods that 
for every standard Levi subgroup $L$ of $\tDlsc(q)$, every $\la\in\cusp(L)$ extends to its stabilizer inside $\NNN_\GF(\bL)$, which follows from the mentioned results by Geck and Lusztig.  For $E(\GF)\leq \Aut(\GF)$ from \Cref{ssec2A}, we construct a $\wt T$-transversal $\TT$ of $\cusp(L)$ and an $N\Stab_{E(\GF)}(L)$-equivariant extension map \wrt $L\lhd N$ for $\TT$. 

\begin{thmbox}\label{thm_loc} 
Let $L$ be a standard Levi subgroup of $\GF=\tD_{l,sc}(q)$. Let $\EL:=\Stab_{E(\GF)}(L)$, $N$, $\wh N:=N\EL$ and ${\wt L'}:=\wt T_0 L$ be associated to $L$ as in \ref{not_L_N_E1}.
If $\tD_{l',sc}(q)$ is a direct factor of $[L,L]$, then assume \Cref{hyp_cuspD_ext} holds for $\tD_{l',sc}(q)$. Then:
\begin{thmlist}	
\item There exists an $\wh N$-stable ${\wt L'}$-transversal $\TT\subseteq \cusp(L)$.
\item \label{stabcuspN} There exists an $\wh N$-equivariant extension map $\II LambdaLN@{\Lambda_{L\lhd N}}$ \wrt $L\lhd N$ for $\TT$.
\end{thmlist}
\end{thmbox}
This implies \Cref{theoC} and ensures Assumptions (i) and (ii) of \Cref{prop23_neu}. In \cite[Thm.~4.3]{BS} and \cite[Prop.~4.13]{CSS} the analogous result was shown in the case where $\bG$ is of type $\tA_l$ or $\tC_l$. The interested reader may notice that without assuming \Cref{hyp_cuspD_ext} for smaller ranks, the proof we give implies a version of the theorem without the equivariance statement.  

Like in the proofs given in \cite{BS} and \cite{CSS}, we essentially apply the following statement providing an extension map for non-linear characters.
\begin{prop}[{\cite[Prop.~4.1]{CSS}}] \label{cor_tool} 
Let $K\lhd M$ be finite groups, let the group $D$ act on $M$, stabilizing $K$ and let $\KK\subseteq \Irr(K)$ be $MD$-stable. Assume there exist $D$-stable subgroups $\Kcirc$ and $V$ of $M$ such that 
\begin{asslist}
	\item \label{cor32i} the groups satisfy: 
	\begin{enumerate}
		\item[(i.1)] \label{cor23i1} $K=\Kcirc (K\cap V) $ and $H:= K\cap V \leq \Z(K)$,
		\item[(i.2)] \label{prop23i2} $M=KV$; 
	\end{enumerate}
	\item \label{cor23ii} for $\KK_0:=\bigcup_{\la\in\KK}\Irr(\restr\la|{\Kcirc})$ there exist
	\begin{enumerate}
		\item[(ii.1)] a $VD$-equivariant extension map $\Lambda_0$ \wrt $H\lhd V$; and 
		\item [(ii.2)] \label{cor_toolext} an $\epsilon(V)D$-equivariant extension map $\Lambda_\epsilon$ with respect to $\Kcirc\lhd \Kcirc\rtimes \epsilon(V)$ for $\KK_0$, where $\eps\colon V\to V/H$ denotes the canonical epimorphism. 
	\end{enumerate}
	\end{asslist}
Then there exists an $MD$-equivariant extension map with respect to $K\lhd M$ for $\KK$.
\end{prop}
In this section we construct the set $\TT$ for part (a) of \Cref{thm_loc} and introduce groups $H$, $K$, $\Kcirc$ (see \Cref{lem34}), $M$, $D$ and $V$ (in \Cref{cor322}) for a later application of \Cref{cor_tool} in the proof of Theorem~\ref{stabcuspN}. Here we show that the groups introduced satisfy the group-theoretic assumptions made in~\ref{cor32i}. 
Afterwards, in \Cref{sec4_neu} we ensure the character-theoretic assumptions, namely~\ref{cor23ii} in order to prove Theorem \ref{stabcuspN}. 

\subsection{Subgroups of the Levi subgroup $L$}\label{subsec_L}
As a first step we dissect the root system of $\bL$ and introduce subgroups of $L$ with those new root systems. For a non-negative integer $i$ let $\II underline i@{\underline{i}}:=\{1,\ldots, i\}$. For computations with elements of $\bG$ we use the Steinberg generators satisfying the Chevalley relations together with an explicit embedding of $\tD_{l,sc}(\FF) $ into $\tB_{l,sc}(\FF) $.
\begin{notation}[The groups $\bG$ and $\obG$, roots  and generators]\label{not_32}
In this and the following section we assume that the simply-connected simple group $\bG$ from \ref{not} is of type $\tD_l$ ($l\geq 4$) over $\FF$ the algebraic closure of $\FF_p$ for $\III p$ some odd prime. Hence $\II G@\bG\cong\tDlsc(\FF)$. Denote $\II lunder@\ul:=\{1,\dots,l\}$. 
Let $\II Phi@ \Phi:=\{\pm e_i\pm e_j\mid i,j\in\ul,\, i\neq j \}$ be the root system of $\bG$ with simple roots $\III{\al_2}:=e_2+e_1$, $\III{\al_1}=e_2-e_1$ and $\III{\al_i}:=e_i-e_{i-1}$($i\geq 3$),  
\begin{align*}
\II Delta@\Delta:=\{ \al_{i}\mid i\in \underline {l} \},
\end{align*} 
see \cite[Rem.~1.8.8]{GLS3}, where the set $\{\II ei@{e_i}\}_{i\in\ul}$ is an orthonormal basis of $\mathbb R^l$ whose scalar product is denoted by $\II (xy)@{(x,y)}$. The Chevalley generators $\II xalphat@{\xx_\al(t)}$, $\II nalphat@{\n_\al(t')}$ and $\II{halphat}@{\hh_\al(t')}$ ($\al\in\Phi$, $t,t'\in\FF$ with $t'\neq 0$) together with the Chevalley relations describe the group structure of $\bG$, see \cite[Thm.~1.12.1]{GLS3}. 

Let $\II Phioverline @{\ov \Phi }:=\{\pm e_i,\,\pm e_i\pm e_j\mid i,j\in\ul,\, i\neq j \}$, $\II G overline @{\obG}:=\tB_{l,sc}(\FF)$ with Chevalley generators $\overline{\xx}_\al(t)$, $\ov\n_\al(t')$ and $\ov\hh_\al(t')$ ($\al\in\ov \Phi$, $t,t'\in\FF$ with $t'\neq 0$). Assume that the structure constants of $\bG$ and $\ov\bG$ are chosen such that $\xx_\al(t)\mapsto\ov\xx_\al(t)$ ($\al\in \Phi$, $t\in\FF$) defines an embedding $\iota_\tD:\bG\ra\ov\bG$. For simplicity of notation we write $ \xx_\al(t)$, $ \n_\al(t')=\xx_\al(t')\xx_{-\al}(-t'{}^{-1})\xx_\al(t')$ and $ \hh_\al(t')=\nn_\al(t')\nn_\al(1)^{-1}$ for the generators of $\obG$ and thus identify $\bG$ with the corresponding subgroup of $\obG$. This is possible according to \cite[10.1]{S10a}, see also \cite[2.C]{MS16}. Among the relations between Chevalley generators, the following will be the most useful to us. For $a,b\in{\mathbb R}^l\setminus \{0\}$ recall ${\langle a,b\rangle}= 2(a,b)/(b,b)$. Let $\al,\beta\in\ov\Phi$, $t\in\FF$, $t'\in\FFtimes$. Then
\begin{align*}
\hh_\al(t')\hh_\beta(t')&=\hh_{\al+\beta}(t')\ \ \text{whenever $\al+\beta\in\ov\Phi$,}\\
\nn_\al(t)^{\hh_\beta(t')}&=\nn_\al(t'{}^{\spann<\al,\beta>}t)  \\
\hh_\al(t)^{\nn_\beta(1)}&=\hh_{\al-\spann<\al,\beta>\beta}(c_{\al,\beta}t) 
\end{align*} 
where the first line is from \cite[1.12.1(e)]{GLS3}, the second is easy from \cite[1.12.1(g)]{GLS3}, and the third, along with the definition of $c_{\al,\beta}\in\{\pm 1\}$, is from \cite[1.12.1(i)]{GLS3}.
\end{notation}
 
\begin{defi} \label{gammane1}
Let $\II Xalpha@{\bX_{\al}}:=\spann<\xx_{\al}(t)\mid t \in \FF>$ for $\al\in\ov \Phi$,
$\II T@\bT:=\spann<\h_\al(t')\mid \al\in \Phi, t'\in \FF^\times> $ and $\II T overline @{\ov\bT}:=\spann<\h_\al(t')\mid \al\in \ov\Phi, t'\in \FF^\times>$. Note $\bT=\ov\bT$ is the image of the map $$(\FFtimes){}^l\ni (t'_1,\dots ,t'_l)\mapsto \hh_{e_1}(t'_1) \dots \hh_{e_l}(t'_l) $$ with kernel $\{  (t'_1,\dots ,t'_l)\in\{ \pm 1\} ^l \mid t'_1\dots t'_l=1   \}$, see also \cite[10.1]{S10a}. 
The group $\bT$ can be chosen as the group $\bT_0$ from \ref{not} and $\bT\spann<\bX_\al\mid \al\in\Delta>$ as the group $\bB$. 

Denoting $\III{h_0}=\hh_{e_1}(-1)$, one has $\Z(\ov\bG)=\spann<h_0>$ of order 2, see \cite[1.12.6]{GLS3}, with $\ov\bG/\spannh =\II SOodd@{\SO_{2l+1}(\FF)}\geq \II SOeven@{\SO_{2l}(\FF)}=\bG/\spannh$.

For every positive integer $i$ let $\III {F_{p^i}}:\obG \ra \obG$ be the Frobenius endomorphism given by $\xx_{\al}(t)\mapsto \xx_\al(t^{p^i})$ for $t\in\FF$ and $\al\in \ov\Phi$. We write $\II gamma@{\gamma}$ for the graph automorphism of $\bG$ given by $\xx_{\eps\al}(t)\mapsto \xx_{\eps\gamma_0(\al)}(t)$ for $t\in\FF$, $\eps\in \{\pm 1\}$ and $\al\in\Delta$, where $\II gamma0@{\gamma_0}$ denotes the symmetry of the Dynkin diagram of $\Delta$ of order $2$ with $\al_2 \mapsto \al_1$. If $l=4$ we denote by $\II gamma3@{\gamma_3}$ the graph automorphism of $\bG$ induced by the symmetry of the Dynkin diagram of $\Delta$ with order $3$ sending $\al_2 \mapsto \al_1$ and $\al_1\mapsto\al_4$. We assume that $\III F=F_{q}$ for $\III q:=p^f$, where $\III f$ is a positive integer. Note that the group $E(\GF)$ from \ref{not} satisfies accordingly  $\III{E(\GF)}=\spann<\restr F_p|{\GF},\restr\gamma|{\GF}>$ whenever $l\geq 5$, otherwise $l=4$ and $E(\GF)=\spann<\restr F_p|{\GF},\restr\gamma|{\GF},\restr\gamma_3|{\GF}>$. 
 
We recall that the graph automorphism $\gamma$ of $\bG$ is induced by an element of $\obG$ (see \cite[2.7]{GLS3} for the corresponding statement in classical groups).  %\begin{defi} 
Let $\II varpi@\varpi\in\FF^\times$ such that $\varpi^2=-1$. By \cite[10.1]{S10a}, see also \cite[2.C]{MS16}, the automorphism $\gamma$ of $\bG$ is induced by conjugating with $\n_{e_1}(\varpi)\in\obG$. 
\end{defi}
\begin{notation}\label{notationL}
Let $\bL$ be a Levi subgroup of $\bG$ such that $\bB \bL$ is a parabolic subgroup of $\bG$ and $\bT\subseteq \bL$. Let $\III L:=\bL^F$ and $\II Phi'@{\Phi'}$ be the root system of $\bL$, i.e., $\bL=\bT\Spann<\bX_ \al| \al\in \Phi'>$. As $\Phi'$ is a parabolic root subsystem of $\Phi$ it has as basis $\II Delta'@{\Delta'}= \Delta\cap\Phi'$. We assume that one of the following holds: 
\begin{asslist}
	\item \label{cases35i} $\Delta'\subseteq \{ e_2-e_1, e_3-e_2, \ldots, e_l-e_{l-1}\}$, or 
	\item \label{cases35ii} $\{ e_2-e_1, e_2+e_1\}\subseteq \Delta'$.
\end{asslist}
\end{notation}
Recall that a split Levi subgroup of $\bG$ containing $\bT$ is called standard if it is generated by $\bT$ and the $\bX_{\al}$'s such that $\al\in\pm\Delta '$ for some subset $\Delta '\subseteq \Delta$.  Recall that $\gamma$ swaps $e_2-e_1$ and $ e_2+e_1$ while fixing the other elements of $\Delta$. Then any subset $\Delta '\subseteq \Delta$ is such that $\Delta '$ or $\gamma(\Delta ')$ satisfies \ref{cases35i} or \ref{cases35ii}. We then get that $\bL$ can be assumed to satisfy \Cref{notationL}.
\begin{lem}
Every standard Levi subgroup of $\bG$ containing $\bT$ is $\spann<\gamma>$-conjugate to a standard Levi subgroup whose root system has a basis $\Delta'\subseteq \Delta$ satisfying \ref{cases35i} or \ref{cases35ii}.
\end{lem}
%\begin{proof}
%If $\Delta'_0$ is the basis of our standard Levi subgroup, then one of the following holds: 
%\begin{asslist}
%	\item $e_2 +e_1 \notin \Delta'_0$, and equivalently $\Delta'_0\subseteq \{ e_2-e_1, e_3-e_2, \ldots, e_l-e_{l-1}\}$,
%	\item $\{ e_1-e_2, e_1+e_2\}\subseteq \Delta'_0$, or 
%	\item $e_2 -e_1 \notin \Delta'_0$, and equivalently $\Delta'_0\subseteq \{ e_2+e_1, e_3-e_2, \ldots, e_l-e_{l-1}\}$.
%\end{asslist}
%In the last case $\gamma(\Delta'_0)$ satisfies $e_2 +e_1 \notin \gamma(\Delta'_0)$ and then $\gamma(\Delta'_0)$ satisfies \ref{cases35i}.
%\end{proof}

\begin{num}[{\color{ocre}{\bf Decomposing $\Phi'$}}]
In the following we decompose $\Phi'$ into smaller root systems, which are the disjoint union of irreducible root systems of the same type. By $\II type @{\type(\Psi)}$ we denote the type of the root system $\Psi$. Whenever $\Psi$ is a subset of $\ov\Phi$ we also denote by $\II WPsi@{W_\Psi}$ the subgroup of N$_\obG (\bT)/\bT$ generated by reflections defined by elements of $\Psi$.

Since $\Phi'$ is a parabolic root subsystem of $\Phi$, $\Phi'$ decomposes as a disjoint union of indecomposable root systems of type $\tD$ and type $\tA$, that are called components of $\Phi'$.

If $\Delta'$ satisfies Assumption \ref{cases35i}, let {$\II Phid@{\Phi_{d}}$} be the union of the components of $\Phi'$ of type $\tA_{d-1}$ ($d\geq 2$). If $\Delta'$ satisfies Assumption \ref{cases35ii}, let $\II Phi-1@{\Phi_{-1}}$ be the union of components of $\Phi'$ that have a non-trivial intersection with $\{ e_2-e_1, e_1+e_2\}$ and let {$\Phi_{d}$} be the union of components of $\Phi' \setminus \Phi_{-1}$ of type $\tA_{d-1}$ ($d\geq 2$).
If $\Delta'$ satisfies Assumption \ref{cases35ii}, $\type(\Phi_{{-1}})\in \{\tA_3,\tA_1\times \tA_1, \tD_m \mid m\geq 4\}$.

Let $\II DD'@{\DD '}$ be the set of integers $d$ such that $\Phi_d$ is defined and non-empty, that is $\SL_d(\FF)$ is a summand of $[\bL,\bL]$. Then $\Phi'=\bigsqcup_{d\in \DD'} \Phi_d$, a disjoint union. 
\end{num}
Recall that $W_{\ov \Phi}$, the group generated by the reflections along the roots of $\ov \Phi$ coincides with $W_0:=\ov N_0/T_0$, can be identified with the permutations of $\ul\cup-\ul$ that commute with the sign change, and hence acts on $\ul$, see \cite[Rem.~1.8.8]{GLS3}. 
\begin{num}[{\color{ocre}{\bf Orbits of $W_{\Phi'}$ on $\ul$}}] \label{defcO}
Let $\II Ocal@\cO$ be the set of orbits of $W_{\Phi'}$ on $\ul$, $\II O1@{\cO_1}\subseteq \cO$ the set of singletons in $\cO$ and $\II Ocal d@{\cO_d}$ be the set of orbits of $W_{\Phi_d}$ on $\ul$ contained in $\cO\setminus \cO_1$, whenever $d\in\DD'$. We define
\begin{align*}
\II DDL@{\DD(L)}=\II DD@\DD&=\begin{cases} \DD'\cup \{1\}&, \text{ if }\cO_1\neq \emptyset,\\
\DD'&,\, \otw. \end{cases} \end{align*}
For $d\in\DD \setminus\{-1\}$ let $\II ad@{a_d}:=|\cO_d|$ and note that $|I|=d$ for any $I\in\cO_d$.

For $I\subseteq \ul$ let $\II PhiI@{\Phi_I}:=\Phi'\cap \Spann<e_k|k\in I>$ and $\II Phioverline I@{\ov\Phi_I}:=\ov \Phi\cap \Spann<e_k|k\in I>$. For $d\in\DD$ let $\II Jd@{J_d}:=\bigcup_{o\in\cO_d}o$, 
 and $\II Phioverline d@{\ov\Phi_d}:= \ov\Phi_{J_d}$.
\end{num}
\noindent Next we introduce groups $K$, $\Kcirc$ and $H$ that will later be proven to  satisfy Assumption \ref{cor32i} with a group $M$.  
\begin{notation}[Subgroups of $\bL$ and $L$]
\label{defhI}
Let $\varpi\in \FF^\times$ and $h_0$ as in \ref{gammane1}. Define $\II hIt@{\h_I(t)}:=\prod_{i\in I}\hh_{e_i}(t)$ for $I\subseteq \ul$ and $t\in \FF^\times$. 
For $I\in \cO$ let $\II GIbold@{\bGI}=\Spann<\bX_\al| \al \in \Phi_I>$ and $ \II GI@{G_I}:=\bG_I^F$. Note that for $I\in \cO_1$ the group $\bGI$ is trivial. 
Let $\II H0@{H_0}:=\spann<h_0,\h_{e_i}(\varpi)\h_{e_{i'}}(-\varpi)\mid i,i'\in \ul>=\spann<h_\alpha(-1)\mid\alpha\in\Phi >$. 
For $d\in \DD$ let
 $\II Hdtilde@{{\pwt H}_d}:=\spann<h_0, \h_I(\varpi) \mid I\in\cO_d>$, 
 $\II Hd@{H_d}:=\spann<h_0, \h_I(\varpi) \h_{I'}(-\varpi)\mid I,I'\in\cO_d>$ and $$\II H@ H:=\spann<\wt H_d\mid d\in \DD>\cap H_0.$$
\end{notation}

\begin{lem} \label{lem3_9}
Let $\II Deven@{\DD_\even}:=  \DD \cap 2\ZZ$ and
$\II Dodd@{\DD_\odd}:=  \DD\setminus \DD_\even$. If $\II Heps@{ H_{\epsilon}}:=\spann<\wt H_d\mid d\in \DD_\eps>\cap H_0$ for $\eps\in \{ \odd,\even\}$, then $H=H_\even . H_\odd$. 
\end{lem}
\begin{proof} 
An element $t\in\bT$ can be written as $\prod_{i=1}^l\h_{e_i}(t_i)$ ($t_i\in\FF^\times$ ). We have $t\in H_0$ if $t_i\in\spann<\varpi>$ and $\prod_{i=1}^l t_i^2=1$. 
In particular $\h_I(\varpi)\in H_0$ if and only if $|I|$ even. This implies $H_d\leq H_0$ whenever $2\mid d$. On the other hand $\wt H_d\not \leq H_0$ for every $d\in \DD_\odd$. 
\end{proof}

With this notation $\Z(\bG)=\spann<h_0, \h_{\ul} (\varpi)>$, see \cite[Table 1.12.6]{GLS3}. %Next we analyse the structure of $L$ along its normal subgroups. 
\begin{lem}\label{lem34}
 $H\leq \Z(\bL)$.
\end{lem}
\begin{proof}We see that $[\h_I(\varpi), \bG_I]=1$ by the Chevalley relations and this implies the statement by the definition of $H$.\end{proof}

The groups $\II K0@{\Kcirc}:=\spann<G_I\mid I \in \cO>$ and $\III K:=\Kcirc H$ then satisfy Assumption \ref{cor_tool}(i.1) for $H$. 

To understand later the action of $\NNN_\GF(\bL)$ on $\Irr(K)$ we analyse the structure of $L$ by introducing several subgroups. 

\begin{num}[{\color{ocre}{\bf Structure of $\bL$}}] \label{structL}
We note that the Levi subgroup $\bL$ satisfies $\bL=\bT \spann<\bG_I\mid I\in \cO>$. 
Let $\II TI@{\bT_I}:=\Spann<\h_{e_i}(t)|i\in I, t \in \FF^\times >$ for $I\in \cO$.
For $I,I'\in\cO$ with $I\neq I'$ we see that no non-trivial linear combination of a root in $\Phi_I$ and one in $\Phi_{I'}$ is a root in $\Phi$ as well. Therefore  $[\bG_I,\bG_{I'}]=1$
according to Chevalley's commutator formula. By the Steinberg relations we see $[\bG_I,\bT_{I'}]=1$. The group $\bG_I$ is either trivial or a  simply-connected  simple group unless $I= \cO_{-1}$ and $\type(\Phi_{-1})= \tA_1\times \tA_1$. Accordingly $[\bL,\bL]=\spann<\bG_I\mid I\in\cO>$.

We observe that $\bG_I\cap \bT\leq \bT_I$ and computations with the coroot lattices prove that $\bT$ is the central product of the groups $\bT_I$ ($I\in \cO$) over $\spannh$. This implies that $\bL$ is the central product of the groups $\bL_I$ ($I\in \cO$) where $\II LIbold@{\bL_I}:=\bT_I \bG_I$.

Analogously we see that $\bL$ is the central product of the groups $\bL_d$ ($d\in\DD$) over the group $\spann<h_0>$, where $\II Lboldd@{\bL_d}:=\spann<\bL_I\mid I\in \cO_d>$.
\end{num}
The structure of $\bL$ studied above implies the following results on $L$. Recall $\Kcirc:=\spann<G_I\mid I \in \cO>$ from \Cref{lem34}.

\begin{lem} \label{lem36}
Recall $\II Lcal@\calL:\bG\ra\bG$ the Lang map defined by $g\mapsto g\inv F(g)$, let
$\II Lhat@{\pwh L}:=\bL\cap \calL^{-1}(\spann<h_0>)$ and
$\II Ltilde @{\pwt L}:=\bL\cap \calL^{-1}(\Z(\bG))$. 
\begin{thmlist}  
\item If $\II LI@{L_I}:=\bL_I^F$ for every $I\in \cO$ and 
$\II Lcirc @{L_0}:=\Spann<L_I| I \in \cO>$, then
$L_0\lhd L$.
\item Let $I\in \cO$ and $\II LIhat@{\pwh L_I}:=\bL_I\cap \wh L$. Then $\wh L_I=\spann<L_I, t_I>$ for every $t_I\in \bT_I\cap \calL^{-1}(h_0)$. We assume chosen such a $t_I$ for each $I\in\cO$.

The group $\II Lhat@{\pwh L}$ is the central product of $\wh L_I$ ($I\in \cO$) and for $d\in \DD$, $\II Lhatd@{\pwh L_d}:=\wh L\cap \bL_d$ is the central product of $\wh L_I$ ($I\in \cO_d$).
\item $L= \spann<\bL_I^F, t_I t_{I'}\mid I, I'\in \cO>$.
\item $\Kcirc$ is the direct product of all $G_I$, $\Kcirc\lhd \wt L$ and $\wt L/\Kcirc$ is abelian.
\item If $\II zeta@{\zeta}\in \FF^\times $ with $\zeta^{(q-1)_2}=\varpi$ and $t_{Q,2}:=\h_Q(\zeta)$ for every $Q\subseteq \ul$, then $\wt L=\wh L \spann< t_{\ul,2}>$.
\end{thmlist}
\end{lem}
The arguments of \Cref{rem_whG} show that $\wt L'$ from \ref{not_L_N_E1} and $\wt L$ induce the same automorphisms on $\bG$. 
\begin{proof}
Recall that $\bL$ is the central product of the groups $\bL_I$, where each $\bL_I$ is $F$-stable. Every $x\in \bL$ can be written as $\prod_ {I \in\cO} x_I$ with $x_I \in\bL_I$. Clearly $x\in L$ if and only if $\calL(x)=1$. We see that $\calL(x)=\prod_{I \in \cO}\calL(x_I)$ by the structure of $\calL$ and hence $x\in L$ implies $\calL(x_I)\in \spann<h_0>$. The group $L_0$ is the group of elements $\prod_ {I \in\cO} x_I$ with $x_I\in L_I:=\bL_I^F$. The group $\wh L:= \calL(\spann<h_0> ) \cap \bL$ is the group of elements $\prod x_I$ with $x_I\in \bL_I$ and $\calL(x_I )\in \spannh$. Hence $\wh L $ is the central product of $\wh L_I$ ($I\in \cO$) over $\spannh$. Clearly $ L_0 \lhd \wh L$, $\wh L_I=\spann<L_I,t_I>$ and $L=L_0 \spann<t_I t_{I'}\mid I, I' \in \cO>$. This ensures the parts (a), (b) and (c). 

Part (d) follows from the fact that $\bL/\spann<\bG_I\mid I \in \cO>$ is isomorphic to a quotient of $\bT$ and hence abelian. 
For part (e) we observe $\calL(\h_Q(\zeta))= \h_Q(\varpi)$ for every $Q\subseteq \ul$ and recall that $\Z(\bG)=\spann<h_0,\h_{\ul}(\varpi)>$. 
\end{proof}

\subsection{The structure of $N/L$}\label{ssec2B }
We analyse $N:=\NNN_\GF(\bL)$ and $\ov N:=\norm{\oGF}\bL$. In the following we identify $W_{\ov \Phi}$ with certain permutation groups $\Sym_{\pm \underline l}$ via the action on $\{\pm e_i\mid i \in \ul\}$ and $W_{\Phi}$ with $\Sym_{\pm \ul}^{\tD}$. We generalize the notation of those permutation groups in order to describe  $N/L$.

\begin{notation}[Young-like subgroups, $\Sym_M$ and $\Young_J$] 
\label{not_Sym}
Let $M$ be a set. 
Given a map $\| . \|: M\ra \ZZ$ with $m\mapsto \| m\|$ we define $\Sym_M$ to be the group of bijections $\pi: M\ra M $ with $\|\pi(m)\|=\|m\|$ for every $m\in M$ and we write $\II SympmM@{\Sym_{\pm M}}$ for the bijections $\pi: \{ \pm 1\}\times M \ra \{ \pm 1\}\times M$ satisfying $\pi(-1,m)= (-\epsilon, m')$ and $\|m\|=\|m'\|$, whenever $m, m'\in M$ with $\pi(1,m)=(\eps,m')$.  When no map $\|.\|$ is specified we assume it is a constant map.

In order to denote the elements of $\Sym_M$ and $\Sym_{\pm M}$ we fix a bijection $f:M \ra \{1,\ldots, |M|\}$. This induces a canonical embedding $\iota: \Sym_M \ra \Sym_{\underline{{|M|}}}$ and an embedding $\iota_{\pm}: \Sym_{\pm M}\ra \Sym_{\pm \underline{|M|}}$. 
For $r$ pairwise distinct elements $m_1, m_2,\ldots, m_r\in M$ we write $(m_1,m_2,\ldots, m_r)\in \Sym_M$ for the element $\iota^{-1}(f(m_1), f(m_2),\ldots, f(m_r))$. Via $\iota_{\pm}$ we obtain also a cycle notation for elements of $\Sym_{\pm M}$.

If  $J$ is a partition of $M$ we write $J\vdash M$ for short. For $J\vdash M$ we set 
\begin{align*}
\II YoungJ@{\Young_J}&:=\{ \pi \in\Sym_M\mid \pi(J')=J' \text{ for every }J'\in J \}, \und \\ 
\II YoungpmJ@{\Young_{\pm J}}&:=\{ \pi \in\Sym_{\pm M}\mid \pi(\{\pm 1\}\times J')= \{\pm 1\}\times J' \text{ for every }J'\in J \}.
\end{align*}
Let $M_{odd}:=\{ m\in M \mid \| m\| \text{ odd }\}$ and
\begin{align*}
\II SympmMD@{\Sym_{\pm M}^\tD} = \left \{ \pi\in \Sym_{\pm M} \mid 
 |(\{1\}\times M_{\odd})\cap \pi^{-1}(\{-1\}\times M_\odd)| \text{ is even } \right \}. \end{align*}
\end{notation}
We use the above notation for permutation groups on the set $\cO$ from \ref{defcO}. 
\begin{defi} 
Let $\II norm@{\protect \Abstand . \protect \Abstand}: \cO \lra \ZZ$ be given by $\|I\|=d$ for every $I\in \cO_d$ and let $\Sym_{\pm \cO}$, $\Sym^\tD_{\pm \cO}$ and $\Sym_{ \cO}$ be the permutation groups on $\cO$ defined as in \ref{not_Sym} with respect to $\|.\|$. 
\end{defi} 

Recall that we have chosen a maximal torus $\bT$ of $\bG$ and that $\bL$ is a standard Levi subgroup of $\bG$ with $\bT\subseteq \bL$, see \ref{not_32} and \ref{notationL}.  
For $\II N0ov@{\ov\bN_0}:=\NNN_\obG(\bT)$ we identify the Weyl group $\ov \bN_0/\bT$ with $\Sym_{\pm \ul}$, the epimorphism $\II rhoT0@{\rho_{\bT}}: \ov\bN_0\lra \Sym_{\pm \ul}$ is given by 
\[ \rho_{\bT}(\nn_{e_i}(-1))= (i,-i) 
\und 
\rho_{\bT}(\nn_{e_i-e_j}(-1))=(i,j) (-i,-j).\] 
With this notation we can compute the relative Weyl group of $L$ in $G$. Recall $N:=\NNN_\bG(\bL)^F$.
\begin{prop}\label{prop64}
Let $\III{N_0}:=\NNN_{\GF}(\bT)$, $\II N0overline @{\oN_0}:=\NNN_{\oGF}(\bT)$, and $\II N overline@\oN:=\NNN_{\oGF}(\bL)$. Then
\[\rho_{\bT}(\oN\cap N_0)/\rho_{\bT}(L\cap N_0) \cong \Sym_{\pm \cO} \und \rho_{\bT}(N\cap N_0)/\rho_{\bT}(L\cap N_0) \cong \Sym_{\pm \cO}^\tD. \] 
\end{prop}
\begin{proof}According to the considerations in \cite[9.2]{Ca85}, $\rho_{\bT}(N\cap N_0)/\rho_{\bT}(L\cap N_0)\cong \Norm_{\ov W_0} (W_{\Phi'})/W_{\Phi'}$, where $\II Woverline @{\overline W_0}:=\ov N_0/\ov T_0$. We then make routine considerations inside $\overline W_0$, see for instance \cite{H80}. Note that $\Norm_{\ov W_0} (W_{\Phi'})=\Stab_{\ov W_0}(\Phi')=W_{\Phi'}\Stab_{\ov W_0}(\Delta')$.

From the definition of $\Phi_{-1}$ one can check that $\Stab_{\ov W_0}(\Delta')$ stabilizes $\Phi_{-1}\cap \Delta'$. This implies that $\Stab_{\ov W_0}(\Delta')$ stabilizes $\Phi_d\cap \Delta'$ for every $d\in\DD$, and 
\[\Stab_{W_{\o \Phi_d}}(\Phi_d\cap \Delta')= \Sym_{\pm \cO_d}.\]

\noindent We have $\Norm_{\ov W_0} (W_{\Phi'})=W_{\ov\Phi_1}\times \prod_{d\in\DD}\Stab_{W_{\o\Phi_{d}}}(\Phi_{d})=
\Stab_{W_{\o\Phi_{-1}}}(\Phi_{-1})\times 
	W_{\o\Phi_1}\times \prod_{\stackrel{d\in \DD}{d>1}}
	\Stab_{W_{\o\Phi_d}}(\Phi_d)$ with 
	\begin{align*}
	\Stab_{W_{\ov \Phi_{-1}}}(\Phi_{-1})= W_{\Phi_{-1}} \spann< (1,-1)>=W_{\o\Phi_{{-1}}},
	\end{align*}
	and 
	$$\Stab_{W_{\o \Phi_d}}(\Phi_d)= W_{ \Phi_d} \rtimes \Sym_{\pm \cO_d} $$
	for $d\in \DD$ with $d>1$. 
	Hence $\rho_{\bT}( N \cap N_0)/\rho_{\bT}(L\cap N_0)\cong 
\Sym_{\pm \cO}^\tD$. 
\end{proof}
By the proof we see that $\Sym_{\pm \cO}$ corresponds to $\Stab_{\ov W_0}(\Phi'\cap \Delta)$ and hence there exists some embedding of $\Sym_{\pm \cO}$ into $\Sym_{\pm \ul}$. We fix some more notation to describe explicitly the permutations in $\Sym_{\pm \ul}$ corresponding to $\Stab_{\ov W_0}(\Delta')$. 
\begin{notation}\label{not2_17}
For $d\in \DD\setminus\{-1\}$ we fix orderings on $\cO_d$ and the sets $I\in\cO_d$: We write $\II Idj@{I_{d,j}}$ ($ j \in \underline{ a_d}$) for the sets in $\cO_d $ and $\II Idjk@{I_{d,j}(k)}\in I_{d,j}$ ($ j \in \underline{ a_d}$, $k \in \underline d$) for the elements of $I_{d,j}$.

For each $ k\in \underline d$ let %$\II fkd@{f_k^{\circ,(d)}}: \underline {a_d} \ra J_d$ be given by $ j\mapsto I_{d,j}(k)$. Let 
$f_k^{(d)}: \ul\lra \ul$ be a bijection such that $f_k^{(d)}(j)=I_{d,j}(k)$ 
for every $j\in \underline {a_d}$ and 
$f_k^{(d)}$ has the maximal number of fixed points. Then $\II fkd@{f_k^{(d)}}$ defines an element of $\Sym_{\pm \ul}$ without sign changes, that we also denote by $f_k^{(d)}$ by abuse of notation. 
\end{notation}
In the following we use that for every $Q\subseteq \ul$, $\Sym_{\pm Q}$ can be seen naturally as a subgroup of $\Sym_{\pm \ul}$. 
\begin{lem} \label{lem3_16}
\begin{enumerate}
\item Let $d\in \DD\setminus \{-1\}$ and $\II kappadoverline@{\ov\kappa_d}:\Sym_{\pm \underline{a_d}}\ra \Sym_{\pm J_d}$ be given by $\pi \mapsto \prod_{k \in \underline d } \pi ^{f_k^{(d)}}$ the latter a product of conjugates of $\pi$ in $\Sym_{\pm\ul}$. 
Then $\ov\kappa_d$ is injective and $\Stab_{\Sym_{\pm J_d}}(\Phi_d) =W_{\Phi_d} \rtimes \ov\kappa_d(\Sym_{\pm \underline{a_d}})$
\item If $-1\in \DD$ let $\II kappa overline-1@{\ov \kappa_{-1}}: \Sym_{\pm \underline{1}}\ra \Sym_{\pm J_{-1}}$ be the morphism with 
$\ov \kappa_{-1}(\Sym_{\pm \underline{1}})= \spann<(1,-1)>$. Let $\II Woverline d@{\ov W_d}:=\ov \kappa_d(\Sym_{\pm \underline{a_d}})$ and $\II W0L@{W^\circ(L)}:= \prod_{d\in \DD} \ov W_d$. Then $\Stab_{\ov W_0}(\Phi')= W_{\Phi'} W^\circ(L)$. 
\end{enumerate}
\end{lem}
\begin{proof}
For (a) we observe that
the sets $\bigcup_{j\in \uad }I_{d,j}(k)$ ($k\in \underline d)$ form a partition of $J_d$. This implies that the groups $\Sym_{\pm \underline{a_d}}^{f_k^{(d)}}$ and 
$\Sym_{\pm \underline{a_d}}^{f_{k'}^{(d)}}$ commute and are disjoint.
We see that  $\ov\kappa_d(\Sym_{\pm \underline{a_d}})$ stabilizes $\cO_d$. This proves (a). Part (b) is clear from the definitions. 
\end{proof}
We can choose $I_{d,j}(k)$ ($d\in\DD\setminus\{\pm 1\}$, $j\in \uad$ and $k\in \underline d$) such that $e_{I_{d,j}(k+1)}-e_{I_{d,j}(k)}\in \Delta'$ for every $j\in \underline {a_d}$ and $k\in \underline{d-1}$. With this choice, $\ov\kappa_d(\Sym_{\pm \underline{a_d}})$ stabilizes $\Delta'$ and hence coincides with $\Stab_{\ov W_{\ov \Phi_d}}(\Delta')$. 

\subsection{A supplement of $L$ in $\ov N$} 
In the following we determine a supplement $\ov V\leq \ov N_0$  with $\ov N=L \ov V$ and $\rho_{\bT}(\ov V)= W^\circ(L)$ where $W^\circ(L)$ is the group from  \Cref{lem3_16}. We construct the group $\ov V$ using extended Weyl groups $\ov V'_{Weyl}$, see \ref{not38}. Extended Weyl groups are known to be supplements of $ T_0$ in $\ov N_0$.

In a first step we define for every $d\in\DD$ a subgroup $\ov V_d\leq \oN_0$ with $\rho_{\bT}(\ov V_d)=\ov\kappa_d(\Sym_{\pm a_d})$. 
%The group $\ov \kappa_d(\Sym_{\pm a_d})$ was obtained using the morphism $\ov \kappa_d$ from $\Sym_{\pm a_d}$ to $\Sym_{\pm \ul}$. 
We construct $\kappa_d$, a lifting of $\ov \kappa_d$ via $\rho_{\bT}$. 
%Hopefully the construction of $\ov V_d$ in this way makes the underlying ideas accessible for the reader. 
This construction will later simplify some arguments by providing a tool to transfer results from \cite{MS16}. 

By definition the group $\ov \bN_0$ is an extension of $\ov W_0$ by $\bT$. It has proven to be more convenient to work with an extension of $\ov W_0$ by an elementary abelian $2$-group, the extended Weyl group (introduced first by Tits), here denoted by $\ov V'_{Weyl}$. 
(Note that if $2\mid q$ the group $\ov \bN_0$ is the semi-direct product of $\bT$ and a group isomorphic to the Weyl group.) In consideration of \Cref{gammane1} we work here with the group $\ov V_0$, a $\bT$-conjugate of $V'_{Weyl}$. Then the graph automorphism of $ G$ is induced by an element of $\ov V_0$ (see \Cref{gammane1}). 

\begin{notation}[The groups $\ov V_0$, $\ov V_I$ and $V_I$]\label{not38}
The group $\II VWeyl@{V'_{Weyl}}:=\Spann<\ov \nn_i'| i\in \ul>$ with $\ov \nn_1':=\nn_{e_1}(1)$ and $\ov \nn_i':=\nn_{\al_i}(-1)$, where $\al_i= e_i-e_{i-1}$ ($2\leq i\leq l$) is known as the extended Weyl group of type $\tB_l$.
%, see \cite{Tits}. %Hence the elements $\ov \nn_i'$ ($1\leq i \leq i$) satisfy the braid relations of type $\tB_l$. 

Let $\zeta_8\in \FF$ with $\zeta_8^2=\varpi$. The group $\II Voverline0@{\ov V_0}:=(V'_{Weyl})^{\h_{\underline l}(\zeta_8)} $ is accordingly generated by $\II n1ov@{{\ov{\mathbf n}}_1}:=(\ov\nn_1')^{\h_{\underline l}(\zeta_8)}=\nn_{e_1}(\varpi)$ and $\II niov@{\ov{\mathbf n}_i}:=(\ov\nn_i')^{\h_{\underline l}(\zeta_8)}=\nn_{\al_i}(-1)$. The group $\ov V_0$ satisfies $\ov V_0 \cap \ov T_0=H_0$ where $H_0$ is defined as $\spann<\hh_{\al}(-1)\mid \al\in\Phi >$ in \ref{not38}.
According to \ref{gammane1},  $\ov \nn_1\in \ov V_0$ and $\gamma$ induce the same automorphism of $\bG$.

For $I\subseteq \underline l$ we set
\begin{align}
\II VI@{V_I}:= \spann<h_0,\nn_{\pm e_i\pm e_{i'}}(1)\mid i,i'\in I \text{ with } i\neq i'> \und
\II VIov@{\ov V_I}:=V_I\spann<\nn_{e_i}(\varpi)\mid i\in I>.
\end{align}
Let $\II HItilde@{\pwt H_I}:=\spann<\h_{e_i}(\varpi)\mid i \in I>$ and $\II H0tilde@{\pwt H_0}:=\wt H_\ul$.
\end{notation}

\begin{num}[Facts around $H_I\lhd \ov V_I$]
Maximal extendibility holds \wrt $H_I\lhd \ov V_I$, since those groups are conjugate to those considered in \cite[Prop.~3.8]{MS16} for the case where the underlying root system is of type $\tB_{|I|}$. For $\II HI@{H_I}:=\spann<h_0,\hh_{\pm e_i\pm e_{i'}}(-1)\mid i,i'\in I>$ we obtain $\ov V_I\cap \bT=H_I$.

For disjoint sets $I,I'\subseteq \ul$ the Steinberg relations imply 
\begin{align}\label{VIcommutator}
[\ov V_I, V_{I'}]=1 \und [\ov V_I, \ov V_{I'}]=\spannh.
\end{align} We introduce maps $\kappa_d: \wt H_{\uad}V_{\underline {a_d}}\lra \wt H_0\ov V_0$ with $\rho_{\bT}\circ \kappa_d=\ov \kappa_d\circ \rho_{\uad}$ for the canonical epimorphism $\II rhoI@ {\rho_{\uad}: \ov V_{\uad}\lra \Sym_{\pm \uad}}$. 
\end{num}
The following defines a lift of  $\ov W_d:=\ov \kappa_d(\Sym_{\pm a_d})$ that is a subgroup of $\ov V_0$. In \ref{not2_17} we introduced  the elements $f_k^{(d)}\in \Sym_{\pm \ul}$ ($d\in \DD\setminus\{-1\}$, $k \in \uad$) without sign changes.
\begin{lem}\label{lem3_20}
Let $d\in \DD\setminus\{-1\}$, $m_k^{(d)}\in \ov V\cap \rho_\bT^{-1}(f_k^{(d)})$ ($k\in \underline d$) and 
$$\II kappad@{\kappa_d} : \wt H_{\uad}\ov V_{\uad} \lra \wt H_{0} \ov V_{0} 
\text{ with }
x\mapsto \prod_{k=1}^d x ^{m_k^{(d)}}$$
for a fixed order in $\underline d$. 
Set $\II Voverlined@{\ov V_d}:= \spann<\kappa_d(\ov V_{\uad})>$.
Then: 
\begin{thmlist}
 \item $\restr\kappa_d|{V_\uad}$ is a morphism of groups;
 \item $\kappa_d(v^x)=\kappa_d(v)^{\kappa_d(x)}$ for every $x\in \ov V_{\uad}$ and $v\in V_{\uad}$;
  \item $\kappa_d(H_\uad)=\spann<h_0^d,\hh_{I}(\varpi) \hh_{I'}(-\varpi)\mid I,I'\in \cO_d>\leq H_d$;
  %hence $\restr \kappa_d|{H_\uad} $ induces an isomorphism of $H_\uad$ and $H_d$ for odd $d$, and of $H_\uad/\spannh$ and $\kappa_d(H_\uad)=\spann<\hh_{I}(\varpi) \hh_{I'}(-\varpi)\mid I,I'\in \cO_d>$ otherwise;
  \item $\kappa_d(\wt H_{\uad})= \spann<h_0^d,\hh_{I}(\varpi) \mid I\in \cO_d>$,
\item $\rho_{\bT}\circ \kappa_d=\ov \kappa_d\circ \rho_{\uad}$ for the canonical epimorphism $\rho_\uad:\ov V_\uad\lra \Sym_{\pm \uad}$, in particular $\rho_{\bT}(\ov V_d)=\ov W_d= \kappa_d(\Sym_{\pm a_d})$.
\end{thmlist}
\end{lem}
\begin{proof}
The sets $J_d(k):=f_k^{(d)}(\uad)$ form a partition of $J_d$. For $x\in V_{\uad}$ we see $x^{m_k^{(d)}}\in V_{J_d(k)}$ and hence $\restr\kappa_d|{V_\uad}$ is independent of the order chosen in $\underline d$. Then $\restr\kappa_d|{V_\uad}$ is a diagonal embedding of $V_{\uad}$ into the central product of the groups $V_{J_d(k)}$ ($k\in \underline d$) over $\spannh$. This implies that $\restr\kappa_d|{V_\uad}$ is a morphism of groups. This proves (a). 

By part (a) it is enough to prove part (b) for $x=\ov\nn_1$ and $v\in \{\ov \nn_2^{\ov \nn_1}, \ov \nn_2, \ov \nn_3,\ldots, \ov \nn_{a_d}\} $, since $V_{\uad }$ is generated by $\{\ov \nn_2^{\ov \nn_1}, \ov \nn_2, \ov \nn_3,\ldots, \ov \nn_{a_d}\}$. The equation $\kappa_d(\ov \nn_i^{\ov\nn_1})=\kappa_d(\ov \nn_i)^{\kappa_d(\ov\nn_1)}$ for $i\geq 3$ is clear since no non-trivial linear combination of those roots is a root. Computations show $\kappa_d(\ov \nn_2^{\ov\nn_1})=\kappa_d(\ov \nn_2)^{\kappa_d(\ov\nn_1)}$ and hence part (b).

For part (c) we note that $\ker (\restr \kappa_d|{H_\uad})= \spann<h_0^{d-1}>$. The equation $\rho_{\bT}\circ \kappa_d=\ov \kappa_d\circ \rho_{\uad}$ in (e) follows from $\rho_{\bT}(m_k^{(d)})=f_k^{(d)}$. 
\end{proof}
Recall that the group $H$ from Notation~\ref{defhI} is a subgroup associated to $\bL$. To understand the above construction we consider the following statement. 
\begin{theorem}\label{prop_lift}
If $-1\in \DD$, set $\II V-1ov@{\ov V_{-1}}:= \spann<H_{-1},\ov \nn_1>$. Let $\II Vov@{\ov V}:= H\Spann<\ov V_d| d\in\DD>$ and $\II VD@{V_\tD}:=\ov V\cap \bG$.
\begin{thmlist}
\item  $N=L V_\tD$. 
\item If $\gamma\in \EL$, then $\ov \nn_1\in \ov V$.
 \end{thmlist}
\end{theorem} 
\begin{proof}
Because of $\rho_{\bT}(\ov V_d)=\ov W_d$ we see $\rho_{\bT}(\ov V)= W^\circ(L)$. Clearly $V_\tD\leq N$. Additionally we see that $\ov V$ normalizes $L$ and $\bL$ by definition. If $\bL$ is $\gamma$-stable, then $\ov \n_1\in \ov V$. According to \Cref{gammane1}, $\ov \n_1$ and $\gamma$ induce the same automorphism of $\bG$. 
\end{proof}

\begin{cor}\label{cor322}
The groups $K$, $\Kcirc$ and $H$ from \ref{defhI} and \Cref{lem34} together with $\III{V}:=V_{\tD}$, $\III{M}:=KV$ and $\III{D}:=\EL$  satisfy the Assumption \ref{cor23i1}. 
\end{cor}
\begin{proof}
According to \Cref{lem34}, $K=HK_0$ and $H\leq \Z(K)$. This is Assumption~\ref{cor_tool}.(i.1).
The equality $M=KV$ follows from the definition of $M$. In order to prove $H=V\cap K$ we show $\ov V_d\cap \bL\leq  H_d$ for every $d\in\DD$. Since $\ov V_d \leq \ov V_0$ by construction and $\ov V_d:=\spann<\kappa_d(\ov V_{\uad})>=\ov W_d$, we observe that $\kappa_d(H_{\uad})\leq H_d$ according to \Cref{lem3_20}(c). 

By the construction $V_\tD$ is $\ov V\spann<F_p>$-stable and hence $\EL$-stable. By the construction we also see that $K_0$ and $V_\tD$ are $D$-stable. 
\end{proof}

\section{Extending cuspidal characters of Levi subgroups} \label{sec4_neu}
This section now focuses on the character theory of our groups. We ensure the character-theoretic Assumptions \ref{cor23ii} and apply \Cref{cor_tool} in the proof of Theorem~\ref{stabcuspN}. We analyse the action of $\ov V$ on $\Kcirc$ and consider subgroups of $\ov N$ and $L$ associated to each $d\in \DD$. For every $d\in \DD$ we define subgroups $K_{0,d}$ and $K_d$ and study them separately for $d=1$, $d\geq 2$ and $d=-1$. 

%%%%%%%%%%%%%%%%%%%%%%%%%%%%%%%%%%%%%%%%%%%%%%%%%%%%%%%%%%%%%%%%%5
\subsection{The inclusion $H_1 \lhd \ov V_{1}$}
We recall here some results on the extended Weyl groups. 
If $1\in \DD$, then $H_1\cong H_{J_1}$ and $\ov V_1\cong \ov V_{J_1}$ for the group $\ov V_{J_1}$ from \ref{not38}. (Recall $J_d=\bigcup _{I\in \cO_d} I$ for $d\in\DD$, see \ref{defhI}.) 
We set $K_{0,1}:=1$ and $K_1:=H_1$. In order to apply \ref{cor_tool} we investigate the Clifford theory for $H_1\lhd \ov V_1$. 
The results are also relevant for studying $H_d\lhd \ov V_d$ for $d\geq 1$.
\begin{prop}\label{prop4_1}
Let $l'\leq l$ be some positive integer, $\wt H':=\wt H_{\underline{l'}}$, $H':=\wt H'\cap H_0$, $\ov V':=\ov V_{\underline {l'}}$ and $\rho':\ov V'\ra \Sym_{\pm \ul'}$ the canonical epimorphism. 
\begin{thmlist}
\item Maximal extendibility holds with respect to $H'\lhd \ov V'$. 

\item Let $\la\in\Irr(H')$ with $h_0\notin\ker(\la)$. Then some $\ov V'$-conjugate $\la'$ of $\la$ has an extension $\wt \la'$ to $\wt H'$ such that $\rho '(\ov V'_{\wt \la'})=\Sym_{l'}$ and $\ov V'_\la=\ov V'_{\wt \la}\spann<c>$ for some $c\in \ov V'$ with $\rho_{\bT}(c)=\prod_{i\in\underline {l'}}(i,-i)$. 
\end{thmlist}
\end{prop}
\begin{proof}
By \cite[Prop.~3.10]{MS16} maximal extendibility holds with respect to $ H'\lhd \ov V'$.
Note that $\ov V'$ coincides with the group $\ov V^t$ considered in \cite{MS16}. This proves part (a).

Let $\la\in\Irr(H')$ with $\la(h_0)=-1$ and $\wt \la\in\Irr(\wt H'\mid \la)$. Note that $\wt H'$ is the $l'$-fold central product of the cyclic groups $\spann<\hh_{e_i}(\varpi)>$ ($i\in \underline {l'}$) over $\spannh$. The group $\ov V$ acts by permutation and inversion on the factors. It is then easy to see that some $\ov V'$-conjugate $\la'$ of $\la$ has an extension $\wt \la'\in\Irr(\wt H')$ such that 
$$\wt \la'(\hh_{e_i}(\varpi))=\wt \la'(\hh_{e_{i'}}(\varpi)) \text{ for every } i,i'\in \underline {l'}.$$ 
The other extension of $\la'$ to $\wt H$ is $(\wt \la')^{-1}$. Observe that $(\wt \la')^2$ is the character with kernel $H'$. (Recall $\wt H'/H'\cong \Cy_2$ and hence there is exactly one character with this property.) The element $c\in \ov V'$ with $\rho'(c)=\prod_{i=1}^{l'} (i,-i)$ satisfies $(\wt \la')^{c}=(\wt \la')^{-1}$ and hence $\ov V'_{\la'}=\ov V'_{\wt \la'} \spann<c>$.
According to (a) there exists some extension $\phi_0$ of $\la'$ to $\ov V_{\la'}$. Then $\restr \phi_0|{V_{\wt \la'}}$ and $\wt \la'$ determine a common extension $\phi$ to $\wt H'\ov V'_{\wt \la'}$, see \cite[4.1(a)]{S10b}. By this construction  $\restr \phi |{\ov V'_{\wt \la}}$ is $c$-stable. 
\end{proof}
%%%%%%%%%%%%%%%%%%%%%%%%%%%%%%%%%%%%%%%%%%%%%%%%%%%%%%
\subsection{The inclusion $K_{d}\lhd K_{d} \ov V_d$ for $d\geq 2$}
In the following we investigate the groups 
$\II Kd@{K_d}:=H_d\spann<G_I\mid I\in \cO_d>$ and $K_d \ov V_d$ for $d\in\DD\setminus\{\pm 1\}$, 
where $\bG_I=\spann<\bX_\al\mid \al\in\Phi_I>$ and $G_I=\bGI^F$, see \Cref{defhI} and \Cref{lem36}.

\begin{lem}\label{lem33_loc} Let $I\in \cO\setminus (\{J_{-1}\} \cup \cO_1)$  and $\II ZI@{\bZ_I}:= \h_I(\FF^\times)$. Then:
\begin{thmlist}
\item $\bGI\cong\SL_{|I|}(\FF)$ and $G_I\cong\SL_{|I|}( q)$;
\item \label{alggroup_eq_loc} $\bL_I= \bGI.\bZ_I$, $\bL_I/\spann<h_0>\cong \GL_{|I|}(\FF)$ and $|\bGI\cap \bZ_I|=\frac{{|I|_{p'}}}{\gcd(2,|I|)}$; and 
\item \label{45d_loc} $L_I\cong\GL_{|I|}(q)$ if $2\nmid |I|$.
\end{thmlist}
\end{lem}
\begin{proof} By the  assumptions $d:=|I|\geq 1$ and $\Phi_I$ is a root system of type $\tA_{d}$. One has $\bGI=[\bT\bG_I,\bT\bG_I]$ where $\bT\bG_I$ is a Levi subgroup so $\bG_I$ is simply-connected $\cong \SL_{|I|}(\FF)$ by \cite[12.14]{MT}. Note $I\neq J_{-1}$. This gives (a).

Any element of $\bT_I$ can be written as $\prod_{i\in I}\hh_{e_i}(t_i)$ for some $t_i\in \FF^\times$. Let $\kappa\in\FF$ with $\kappa^{|I|}=\prod_{i\in I} t_i$ and fix $j\in I$. Then by the Chevalley relations in $\ov\bG$
\begin{align*}
\prod_{i\in I}\hh_{e_i}(t_i)&= 
\h_{e_j}(t_j\kappa^{-1}) \h_{e_j}(\kappa) \, 
\prod_{\substack{i\in I\\ i\neq j }}\left ( \hh_{e_j}(t_i^{-1}\kappa)^{-1} \hh_{e_j}(t_i^{-1}\kappa) \hh_{e_i}(t_i\kappa^{-1}) \hh_{e_i}(\kappa) \right )= \\
&=\left( \hh_{e_j}( \kappa^{-|I|} \prod_{i\in I}t_i)\right) \, \h_{e_j}(\kappa) \, 
\prod_{\substack{i\in I\\i\neq j }}\left ( \hh_{e_i-e_j}((t_i\kappa^{-1})^2) \hh_{e_i}(\kappa) \right )=\\
&= \left (\prod_{\substack{i\in I\\ i\neq j } 	}\hh_{e_i-e_j}((t_i\kappa^{-1})^2) \right )
\h_{I}(\kappa).
\end{align*}
Accordingly $\bT_I=(\bT_I\cap\bGI) \bZ_I$.
We note that $\bGI\cong \SL_{|I|}( \FF)$ and $G_I\cong \SL_{|I|}( q)$ as $F$ acts on $\bGI$ as standard Frobenius endomorphism. By the Chevalley relations $\bZ_I\leq \cent\bL\bGI$ and $\bL_I=\bT_I \bGI= \bZ_I \bGI$. 

The calculations above show that  an element of $\bZ_I\cap \bGI$ can be written as $\prod_{i\in I}\h_{e_i}(t)$ with $t^{|I|}=1$. For $d\in\DD_{\even}$, the element $\prod_{i=1}^l\h_{e_i}(-1)$ is trivial and hence $|\bZ_I\cap \bGI|=\frac {|I|_{p'}} {\gcd(2,|I|_{p'})}$. 

If $2\mid d $, then with similar considerations as above we see 
\begin{align}\label{eq41h0}
h_0&= \h_I(\zeta) \prod_{\substack{i\in I\\i\neq j }} \h_{e_i-e_j}(\zeta^{-2}),\end{align}
where $\zeta\in \FF^\times$ has order $2|I|_2$. 
This implies that $\bL_I/\spannh$ is the central product of the one-dimensional torus $\bZ_I/\spannh$ with $\bG_I/\spannh$ over $\Z(\bG_I)$. Accordingly $\bL_I/\spannh \cong \GL_{|I|}(\FF)$. 

For odd $d$ this implies analogously $ \bL_I\cong\GL_d(\FF)$ and $L_I\cong\GL_d(q)$. This is the statement in (b) and (c). We could also have argued on Levi subgroups of $\bG/\spann<h_0>=\text{SO}_{2l}(\FF)$.
\end{proof}

Next we study how $\wt L$ acts on $\Kcircd$, which includes the action induced by  $t_I$ ($I\in \cO$) and $t_{\ul,2}$ from \Cref{lem36}. Recall that $\wh L_I:=\bL_I\cap \calL^{-1}(\spann<h_0>)$ satisfies $\wh L_I=\spann<L_I,t_I>$ for some $t_I\in \bT_I\cap \calL^{-1}(h_0)$, and $\calL^{-1}(\h_\ul(\varpi))\cap \bL = \spann<\wh L,t_{\ul,2}>$ with $t_{\ul,2}=\h_\ul(\zeta)$.
According to Remark~\ref{labeldiag}, diagonal automorphisms of $G_I$ are parametrized by $\Z(\bGI) /[\Z(\bGI) ,F]$.
\begin{lem}\label{lem_act_onGI}
Let $I\in\cO\setminus(\cO_1\cup\{J_{-1}\})$. 
\begin{thmlist}
\item If $2\nmid |I|$, then $\wh L_I= L_I\cent{\bT_I}{L_I} $, in particular $t_{I}$ from \Cref{lem36}(b) can be chosen such that  $t_I\in \cent{\bT_I}{L_I}$. 
\item For  $2\mid |I|$, the element $\II tI@{t_{I}}$ induces on $G_I$ a diagonal automorphism corresponding to $g[\Z(\bGI),F]$ with $g\in \Z(\bGI)$ of order $|\Z(\bGI)|_2$.
\item $\II tunderlinel2@{t_{\ul,2}} \in \cent{\wt L}{L_I}$.
\end{thmlist}
\end{lem}
\begin{proof} Keep $d:=|I|$.
According to the theorem of Lang we can choose $t_I\in \bT_I$ such that 
${t_I}^{-1}F(t_I)=h_0$ as $\bT_I$ is connected. 

If $2\nmid |I|$ we see that $h_0=\h_I(-1)$ and hence $h_0\in\bZ_I$. Since $\bZ_I$ is again connected $t_I$ can be chosen in $\Z(\bL_I)$, whence (a).

Following \eqref{eq41h0}, $h_0=z_1z_2$ for some $z_1\in\bZ_I$ and $z_2\in\Z(\bGI)$. Here $z_2$ is an element of order $|I|_2=d_2$. 
Then the element $t_I$ can be analogously written as $zg$ with $z\in\Z(\bG_I)$ and $g\in \bG$ such that $\calL(z)=z_1$ and $\calL(g)=z_2$. As $g$ induces on $\bG$ a diagonal automorphism associated to $z_2 [\Z(\bGI),F]$, the element $t_I\in\wt\bGI$ with $x^{-1}F(x)=h_0$ induces the same diagonal automorphism. This gives (b).

The element $t_{\ul,2}=\h_{\ul}(\zeta)$ from \Cref{lem36}(e) centralizes $\bGI$ since the Weyl group of $\bG_I$ centralizes $t_{\ul,2}$.
\end{proof}
Recall the groups $\II Hd tilde @{\pwt H_d}=\Spann<h_0,\h_I(\varpi)| I\in\cO_d>$, 
$\II H0@{H_0}=\spann<\h_\al(-1)\mid \al \in \Phi>$ and 
$\II Hd@{{H_d}}=\wt H_d \cap H_{0}$ defined in \ref{defhI} for every $d\in \DD$.
Using the groups $G_I$ from \Cref{lem36} let $\II Kdcirc@{K_{0,d}}:= \Spann<G_I | I \in \cO_d>$ and $\II Kd@{K_d}:=H_dK_{0,d}$. If $\DD=\{d\}$, then $\Kcirc=K_{0,d} $ and $K=K_d$. As $\ov V_d\cap K_d\leq \Cent_{\bL}(\bG_I)$ as a consequence of \Cref{lem34}, there is a well-defined action of $\ov V_d/H_d$ on $K_d$. 

\begin{lem}[The action of $V_d$ on $K_{0,d}$]\label{lemVactK} Let $d\in \DD$.
Let $\II epsd@{\epsilon_d}: \ov V_d\ra \ov V_d/H_d$ be the canonical morphism and $\II n1doverline@{\overline {\n}_1^{(d)}} :=\kappa_d(\ov \nn_1)$.
Then:
\begin{thmlist}
\item $K_{0,d}\rtimes \eps(\ov V_d) \cong (G_{I_{d,1}} \rtimes \spann<\eps(\ov \nn_1^{(d)})>)\wr \Sym_{a_d}$.
	\item \label{lemVactKc} Then $\overline \nn_1^{(d)}$ induces the graph automophism transpose-inverse on $G_{I_{d,1}}$.
	\item If $2 \nmid d$, $\overline \nn_1^{(d)}$ induces on $L_{I_{d,1}}$ a product of a non-trivial graph and an inner automorphism via the isomorphism  $L_{I_{d,1}}\cong\GL_d(q)$ from \Cref{lem33_loc}(c).
\end{thmlist}
\end{lem}
\begin{proof}
Part (a) follows from the Steinberg presentation. 
	
For part (b) we see that $G_{I_{d,1}}^{\overline \n_1^{(d)}}=G_{I_{d,1}}$ and $G_{I_{d,1}}\cap \spannh=\{1\}$ from the Chevalley relations. We compute the action of $\overline \n_1^{(d)}$ on $G_{I_{d,1}}$ in the quotient $G/\spannh$ or $G_{I_{d,1}} \times \spann<h_0>/\spann<{h_0}>$, respectively. 
In \cite[2.7]{GLS3} the group $\bG/\spannh$ and its Steinberg generators are given explicitly as subgroup and elements of  $\operatorname{SO}_{2l}(q)$. The element  $\overline \n_1^{(d)}$ acts on $\bG_{I_{d,1}}\spannh/\spannh$ by transpose-inverse via $G_{I_{d,1}}\cong\SL_d(q)$. Computations in that group show part (b).

The element $\ov\nn_1^{(d)}$ acts by inversion on $\bZ_{I_{d,1}}$ and hence $\ov\nn_1^{(d)}$ satisfies the statement in part (c) as $\bL_{I_{d,1}}=\bG_{I_{d,1}}\bZ_ {I_{d,1}}$.
\end{proof}
Next we study an analogue of $\wt L$ from \Cref{lem36} associated to $d\in\DD$, that is defined using the Lang map $\calL$ from there. Note that $\h_\ul(\varpi)\notin \bT_{J_d}$ whenever $\DD\neq \{d\}$, but $\h_\ul(\varpi)= \prod_{d\in\DD}\h_{J_d}(\varpi)$. 

\begin{prop} \label{prop3C}
Let $d\in \DD\setminus\{\pm 1\}$, $\eps_{d}: \ov V_{d}\lra \ov V_{d}/H_{d}$ be as in \Cref{lemVactK}, $\II Ttilded@{\pwt T_d}:=\bT_{J_d}\cap \calL\inv(\spann<h_0, \h_{J_d}(\varpi)>)$ and $\wt L_{d}:= \pwt T_d K_{0,d} $. Then:
\begin{enumerate}
\item There exists some $\ov V_{d}\spann<F_p>$-stable $\wt L_{d}$-transversal $\II Tdcirc@{\TT^\circ_{d}}$ in $\cusp(K_{0,d} )$.
\item There exists an $\eps_{d}(\ov V_{d})\times \spann<F_p>$-equivariant extension map $\II Lambdaepsilond@{\Lambda_{\epsilon_d}}$ \wrt $K_{0,d} \lhd K_{0,d}\rtimes \eps_{d}(V_{d}) $ for $\TT^\circ_{d}$.
\item Maximal extendibility holds with respect to $\Kcircd\lhd \wt L_{d}$ and $K_{d}\lhd \wt L_{d}$.
\end{enumerate}
\end{prop}
For the proof of part (b) we require a strengthening of a result on wreath products that can for example be found in \cite[Thm.~2.10]{Klupsch}. 
\begin{lem}\label{lem_wr} Let $X\rtimes Y$ be a finite group and $A$ be a group of automorphisms of $X\rtimes Y$, stabilizing $X$, $Y$ and some $\KK\subseteq \Irr(X)$. Let $a$ be a positive integer. Note that $A$ acts on $X^a\lhd (X\rtimes Y)\wr \Sym_a$ by diagonally acting on $(X\rtimes Y)^a$ and trivially on $\Sym_a$. In this context we write then $\Delta A$ for that group. If there exists an $(X\rtimes Y)\rtimes A$-equivariant extension map \wrt $X\lhd X\rtimes Y$ for $\KK$, then there exists an $((X\rtimes Y)\wr \Sym_a)\rtimes \Delta A$-equivariant extension map \wrt $X^a\lhd (X\rtimes Y)\wr \Sym_a$ for $\KK^a:=\{\chi_1\otimes\cdots\otimes\chi_a\mid \chi_i\in\KK  \}$.\end{lem}
\begin{proof}This follows by the considerations in the proof of \cite[Thm.~2.10]{Klupsch} using the construction of representations of wreath products given in \cite[10.1]{Navarro_book}.  \end{proof}

\begin{proof}[Proof of \Cref{prop3C}]
Let $I_1:=I_{d,1}$. Via the isomorphism $G_{I_1}\cong \SL_d(q)$ from \Cref{lem33_loc} the $E(\SL_d(q))$-stable $\GL_d(q)$-transversal in $\Irr(\SL_d(q))$ from \Cref{prop43}(b) determines a subset $\TT_{I_1}\subseteq \cusp(G_{I_1})$. According to Lemma \ref{lemVactKc} this set is $\ov \nn_1^{(d)}$-stable. 
The $E(\SL_d(q))$-stable $\GL_d(q)$-transversal in $\Irr(\SL_d(q))$ can even be chosen such that each character extends to its inertia group in $\SL_d(q)\rtimes E(\SL_d(q))$. Accordingly maximal extendibility holds with respect to $G_{I_1}\lhd G_{I_1}\rtimes \spann<F_p,\epsilon_d(\ov\nn_1^{(d)})>$ for $\TT_{I_1}$. 

Note that $\TT_{I_1}$ is $\NNN_{\ov V_d}(G_{I_1})$-stable, as $\NNN_{\ov V_d}(G_{I_1})$ acts as $\spann<\ov \n_1^{d}>$. Accordingly via conjugation with elements of $\ov V_d$ the set $\TT_{I_1}$ determines characters $\TT_{I}\subseteq \cusp(G_I)$ for every $I\in \cO_d$. Recall that by \Cref{lem36}(d) the group $\Kcirc$ is the direct product of the groups $G_I$ ($I\in \cO$). Analogously $K_{0,d}$ is the direct product of the groups $G_I$ ($I\in \cO_d$). 

The product $\TT_d^\circ$ of these characters $\prod_{I\in\cO_d}\TT_I$ defines a $\ov V_d \spann<F_p>$-stable set in $\cusp(K_{0,d})$. By this construction $\TT_d^\circ$ is $\ov V_d\spann<F_p>$-stable. Following the description of the action of $\wt L$ on $\Kcircd$ given in \Cref{lem_act_onGI} we see that $\TT^\circ_d$ is an $\wt L_d$-transversal in $\cusp(K_{0,d})$. This proves part (a). 

Recall $K_{0,d} \rtimes \eps_d(\ov V_d) \cong \left (G_{I_1}\rtimes \spann<\eps(\ov\nn_1^{(d)})>\right )\wr \Sym_{a_d}$ from \Cref{lemVactK}. As stated above,  maximal extendibility holds with respect to $G_{I_1}\lhd G_{I_1}\rtimes \spann<F_p,\epsilon(\ov\nn_1^{(d)})>$ for $\TT_{I_1}$. According to \Cref{lem_wr} this implies by the choice of $\TT^\circ_d$ that there is an $\eps_{d}(\ov V_{d})\spann<F_p>$-equivariant extension map $\Lambda_{\eps_d}$ \wrt $K_{0,d} \lhd K_{0,d} \rtimes \eps_{d}(V_{d}) $ for $\TT^\circ_{d}$.

According to \Cref{lem_ext_G_whG} maximal extendibility holds with respect to $G_I\lhd \wt G_I$ where $\wt G_I:=\bG_I \cap \calL\inv(\Z(\bGI))$. Additionally $[\bG_I,\bZ_I]=1$ for $\bZ_I:=\h_I(\FF^\times)$ from \Cref{lem33_loc}. We observe that $\wt L_d\leq \spann<\wt G_I\mid I\in \cO_d> \spann<\bZ_I\mid I\in \cO_d>$, even more precisely 
$$\wt L_d\leq \spann<\wt G_I\mid I\in \cO_d> \spann<\wh Z_I\mid I\in \cO_d>,$$ 
where $\wh Z_I:=\calL^{-1}(\Z(\bG_I)\cap \bZ_I)\cap \bZ_I$. We see that maximal extendibility holds \wrt $K_{0,d} \lhd \spann<\wt G_I\mid I\in \cO_d> \spann<\wh Z_I\mid I\in \cO_d>$. Hence maximal extendibility holds \wrt $K_{0,d} \lhd \wt L_d $ and $K_d\lhd \wt L_d$, as $\wt L_d/\Kcircd$ is abelian. 
\end{proof}

\begin{lem}\label{prop320} Let $d\in \DD \setminus\{\pm 1\}$.
\begin{thmlist}
\item Maximal extendibility holds \wrt $H_d\lhd \ov V_d$.
\item If $2\nmid d$, $\la\in\Irr(H_d)$ with $\la(h_0)=-1$ and $\wt \la\in\Irr(\wt H_d|\la)$, then $(\ov V_d)_{\wt \la}\leq V_\tD$ and $(\ov V_d)_\la=(\ov V_d)_{\wt \la}\spann<c_d>$ for some $c_d\in\ov V_d$. 
\end{thmlist}
\end{lem}
\begin{proof} Recall that by \cite[Prop.~3.8]{MS16} maximal extendibility holds with respect to $ H_{\uad}\lhd \ov V_{\uad}$. Via the map $\kappa_d: \ov V_{\uad}\lra \ov V_d$ from \Cref{lem3_20} the maximal extendibility \wrt $H_{\uad}\lhd \ov V_{\uad}$ gives a $\ov V_d$-equivariant extension map for $\kappa_d(H_\uad)\lhd \kappa_d(V_\uad)$. This implies part (a) according to \cite[4.1(a)]{S10a}.

In part (b) we assume $2\nmid d$ and hence $\kappa_d(H_\uad)=H_d$. The character $\la\in\Irr(H_d)$ with $\la(h_0)=-1$ corresponds via $\kappa_d$ to some $\la_0\in\Irr(H_{\uad})$ with $\la_0(h_0)=-1$. \Cref{prop4_1}(b) implies that via  $\kappa_d$ there is some $\ov V_d$-conjugate $\la'$ of $\la$ with $\rho_{\bT}(\ov V_{\wt \la'})=\Sym_{\cO_d}$ for any $\wt \la'\in\Irr(\wt H_d\mid \la')$ and $(\ov V_{d})_{\la'}=\ov V_{\wt \la'}\spann<c'_d>$ for some $c'_d\in\Irr(\ov V_d)$ with $\rho_{\bT}(c'_d)=\prod_{i\in J_d}(i,-i)$. We observe that $(\ov V_d)_{\wt \la'}\leq V_\tD$. Because of $V_\tD\lhd \ov V$ this implies $(\ov V_d)_{\wt \la}\leq V_\tD$ and $(\ov V_d)_\la=(\ov V_d)_{\wt \la} \spann<c_d>$ for some $c_d\in\ov V_d$. This proves part (b). 
\end{proof}

%%%%%%%%%%%%%%%%%%%%%%%%%%%%%%%%%%%%%%%%%%%%%%%%%%%%%%
\subsection{Consideration of $K_{-1}\lhd K_{-1}\ov V_{-1}$}
The group structure of $G_{J_{-1}}$ depends on $\type(\Phi_{-1})$. By its definition, 
$\type(\Phi_{-1})\in \{\tA_1\times \tA_1, \tA_3, \tD_{|J_{-1}|}\}$. 
For the application of \Cref{cor_tool} we prove the following statement. Recall $\ov V_{-1}=\spann<\n_{e_1}(\varpi), h_0>$, $\wt H_{-1}=\spann<\h_{J_{-1}}(\varpi), h_0>$, $H_{-1}=\wt H_{-1}\cap H_0$ and $G_{J_{-1}}=\spann<\bX_\al\mid \al\in\Phi_{-1}>^F$. As before we set $\II Kcirc-1@{K_{0,-1}}:=G_{J_{-1}}$ and $\II K-1@{K_{-1}}:=H_{-1}G_{J_{-1}}$.

\begin{prop}\label{prop3D}
Assume Hypothesis \ref{hyp_cuspD} holds for $\bG_{J_{-1}}^F$ if $\Phi_{-1}$ is of type $\tD$. Let $\II epsilon-1@{\epsilon_{-1}}:\ov V_{-1}\lra \ov V_{-1}/H_{-1}$ be the canonical epimorphism and 
$\wt L_{-1}:= (\bG_{J_{-1}}\bT_{J_{-1}})\cap \calL^{-1}(\spann<h_0, \h_{J_{-1}}(\varpi)>)$. Then:
\begin{enumerate}
\item There exists some $\ov V_{-1}\spann<F_p>$-stable $\wt L_{-1}$-transversal $\II T-1@{\TT_{-1}^\circ}$ in $\cusp(K_{0,-1})$.
\item There exists an $\eps_{-1}(\ov V_{-1})\spann<F_p>$-equivariant extension map $\II Lambdaepsilona-1@{\Lambda_{\epsilon_{-1}}}$ \wrt $K_{0,-1} \lhd K_{0,-1}\rtimes \eps_{-1}(\ov V_{-1}) $ for $\TT^\circ_{-1}$.
\item Maximal extendibility holds with respect to $K_{0,-1}\lhd \wt L_{-1}$ and $K_{-1}\lhd \wt L_{-1}$.
\end{enumerate}
 \end{prop}
\begin{proof} As in the proof of \Cref{lem36} we see that $\wt L_{-1}=T_{-1}{G_{J_{-1}}} \spann<t_{J_{-1}}, t_{J_{-1},2}>$, where  $T_{-1}:=\bT_{J_{-1}}^F$, $\zeta\in\FF^\times$ with $\zeta^{(q-1)_2}=\varpi$, $\II tJ-1@{t_{J_{-1}}}:=\h_{e_1}(\zeta^2)$ and $\II tJ-12@{t_{J_{-1},2}}:=\h_{J_{-1}}(\zeta)$. Note that the action of $\wt L$ on $G_{J_{-1}}$ coincides with the one of $\spann<T_{{-1}}, t_{J_{-1}}, t_{J_{-1},2}>$ up to inner automorphisms. By the definition of $G_{J_{-1}}$ we see 
\[ G_{J_{-1}}\cong \begin{cases}
\tD_{|J_{-1}|,sc}(q)& \text{ if } \type(\Phi_{-1})=\tD_{|J_{-1}|},\\
\SL_4(q) & \text{ if } \type(\Phi_{-1})=\tA_3,\\
\SL_2(q)\times \SL_2(q)&\text{ if }\type(\Phi_{-1})=\tA_1\times \tA_1.
\end{cases}\]
 
Assume $\type(\Phi_{-1})=\tD_{l_{-1}}$ with $l_{-1}:=|J_{-1}|$ and $l_{-1}>3$. Then $T_{-1}\leq G_{J_{-1}}$. The elements $t_{J_{-1}}$ and $t_{J_{-1},2}$ act as diagonal automorphisms on $G_{J_{-1}}$.
Part (a) follows from \Cref{thm_Malle}. By Hypothesis \ref{hyp_cuspD} we can choose a $\ov V_{-1}\spann<F_p>$-stable $\wt L_{-1}$-transversal $\TT^\circ_{-1}$ in $\cusp(K_{0,-1})$ such that maximal extendibility holds \wrt $K_{0,-1}\lhd K_{0,-1}\rtimes \spann<\gamma, F_p>$ for $\TT_{-1}$. Note that 
$K_{0,-1}\rtimes \spann<\gamma, F_p>=K_{0,-1}\rtimes (\eps_{-1}(\ov V_{-1})\times \spann<F_p>)$. By this choice we see that an extension map $\Lambda_{\eps_1}$ as required in part (b) exists. Note that the actions on $G_{J_{-1}}$ induced by $\gamma$ and $\n_{e_1}(\varpi)$ coincide by \ref{not_32}. According to \Cref{lem_ext_G_whG} maximal extendibility holds with respect to $G_{J_{-1}}\lhd \wt L_{-1}$. This proves part (c) in the case where $\type(\Phi_{-1})=\tD_{l_{-1}}$ with $l_{-1}>3$. 

Assume $\type(\Phi_{-1})=\tA_1\times \tA_1$, then $t_{J_{-1}}$ induces on both factors a non-inner diagonal automorphism, while $t_{J_{-1},2}$ induces a non-inner diagonal automorphism only on one factor, since $h_0=\h_{e_2-e_1}(-1)\h_{e_1+e_2}(-1)$ and $\h_{e_1}(\varpi) \h_{e_2}(\varpi)= \h_{e_1+e_2}(-1)$. Clearly $\ov V_{-1}$ acts by permutation of the two factors. Let $\TT(\SL_2(q))$ be an $\spann<F_p>$-stable $\GL_2(q)$-transversal in $\cusp(\SL_2(q))$, see \Cref{prop43}. Then $\TT^\circ_{-1}:=\TT(\SL_2(q))\times \TT(\SL_2(q))$ is a $\ov V_{-1}\spann<F_p>$-stable $\wt L_{-1}$-transversal in $\cusp(G_{J_{-1}})$. This proves part (a) in that case. Part (b) follows from the fact that $K_{0,-1}\rtimes \eps_{-1}(\ov V_{-1})\cong \SL_2(q)\wr \Cy_2$, see also \Cref{lem_wr}. Part (c) follows again from the fact that $ \wt L_{-1}$ is $(\wh \SL_2(q))^2$, where 
$$\wh \SL_2(q):=\left \{x\in\SL_2(\FF) \,   | \, F_q(x)=\pm x \right \}.$$

Assume $\type(\Phi_{-1})=\tA_3$.
Recall $\al_2=e_2+e_1$, $\al_1=e_2-e_1$ and $\al_i:=e_i-e_{i-1}$($i\geq 3$) for the simple roots in $\Delta$. In this case $G_{J_{-1}}\cong \SL_4(q)$ and $\n_{e_1}(\varpi)$ acts on $G_{J_{-1}}$ as a non-trivial graph automorphism. 
In order to see the automorphisms induced by $t_{J_{-1}}$ and $t_{J_{-1},2}$ we use again the equation $h_0=\h_{\al_1}(-1)\h_{\al_2}(-1)$ and additionally the equation \[\h_{\underline 3}(\varpi) =\h_{\al_1}(-\varpi)\h_{\al_2}(-\varpi) \h_{\al_3}(-1).\]
This implies that $t_{J_{-1}}$ induces on $G_{J_{-1}}$ some non-inner diagonal automorphism of $\SL_4(q)$ corresponding via the Lang map (see Remark~\ref{labeldiag}) to the central involution, while $t_{J_{-1},2}$ induces a diagonal automorphism of $\SL_4(q)$ associated to a generator of the centre. 
Let $E(\SL_4(q))$ be the subgroup of $\Aut(\SL_4(q))$ from \ref{prop43}.
According to \Cref{prop43}(a) there exists a $\GL_4(q)$-transversal $\TT(\SL_4(q))$ in $\Irr(\SL_4(q))$, that is stable under the group $E(\SL_4(q))$ generated by graph and field automorphisms of $\SL_4(q)$ and such that maximal extendibility holds \wrt $\SL_4(q)\lhd \SL_4(q)\rtimes E(\SL_4(q))$ for $\TT(\SL_4(q))$. This choice guarantees part (b). As $\wt L_{-1}/G_{J_{-1}}$ is cyclic, part (c) holds in that case, as well.
\end{proof}
Recall $\II Htilde1@{\pwt H_1}:=\spann<h_0, \h_{J_{-1}}(\varpi)>$.
\begin{lem}
\begin{thmlist}
\item There exists some $\ov V_{-1}$-equivariant extension map $\Lambda_{0,-1}$ \wrt $H_{-1}\lhd \ov V_{-1}$. 
\item If $\la\in \Irr(H_{-1})$ with $\la(h_0)=-1$ and $\wt \la \in \Irr(\wt H_{-1}\mid \la)$, then $(\ov V_{-1})_{\wt \la}=H_{-1}$. 
\end{thmlist}
\end{lem}
\begin{proof} 
As $\ov V_{-1}/H_{-1}$ is cyclic there exists an extension map as required in (a). For the proof of (b) note that the equality $[\ov \n_1,\hh_{J_{-1}}(\varpi)]=h_0$ implies $\wt \la^{\ov \n_1}\neq \wt \la$ as $\wt \la(h_0)=-1$. 
\end{proof}
%%%%%%%%%%%%%%%%%%%%%%%%%%%%%%%%%%%%%%%%%%%%%%%%%%%
\subsection{Proof of Theorem B}
We now finish the proof of \Cref{thm_loc} and therefore Theorem~B.
The above allows us now to verify the character-theoretic assumptions from \Cref{cor_tool} for the groups $K$, $\Kcirc$, $\Kcircd$ and $V_\tD$, introduced in \Cref{lem34} and \Cref{prop_lift}. From the definitions of $K_{0,d}$ before \ref{lemVactK} we see $\Kcirc=\Spann<\Kcircd| d\in \DD>$, even more $\Kcirc$ is the central product of the groups $\Kcircd$ ($d\in \DD$).

By abuse of notation we write $\cusp(K)$ for $\bigcup_{\chi\in\cusp(L)}\Irr(\restr \chi|K)\subseteq \Irr(K)$.
\begin{prop}
\label{def_TT} There exists a $\ov V\spann<F_p>$-stable $\wt L$-transversal $\II Kcirc@{\mathbb K_0}$ in $\cusp(\Kcirc)$. Moreover
$\II K@{\mathbb K}:=\Irr(K\mid \KK_0)$ and $\II T@{\mathbb T}=\Irr(L\mid \KK)$ are $N\EL$-stable $\wt L$-transversals in $\cusp(K)$ and $\cusp(L)$, respectively.
\end{prop}
Note that this implies Theorem \ref{thm_loc}(a). 
\begin{proof} 
For $d\in \DD\setminus\{1\} $ let $\TT_d^\circ$ be the $\ov V_d \spann<F_p>$-stable $\wt L_d$-transversal in $\cusp(\Kcircd)$ from \Cref{prop3C} and \Cref{prop3D}. Note $K_{0,1}=1$. 
The group $\Kcirc$ is a central product of the groups $\Kcircd$ ($d\in \DD$) according to \Cref{lem36}. Hence the irreducible characters of $K_0$ are obtained as the products of the irreducible characters of $\Kcircd$. The central product of the characters in $\TT^\circ_d$ form a subset $\KK_0\subseteq \Irr(\Kcirc)$. We see that $\KK_0$ is $V_\tD \EL$-stable since $V_{\tD}\EL$ and $\ov V \spann<F_p>$ act on each factor $\Kcircd$ as $\ov V_d\spann<F_p>$. Let $\II That@{\pwh T}:=\bT\cap \calL\inv(\spannh)$. The automorphisms of $\wt L $ induced on $\Kcirc$ are induced by $K_0$, $\wh T=\prod_{d\in \DD} \wh T_d$ and $t_{\ul,2}=\prod_{d\in \DD} t_{J_d,2}$, see \Cref{lem33_loc}. According to \Cref{lem_act_onGI} the element $t_{J_d,2}$ acts trivially on $G_d$, whenever $d\geq 1$. Hence $\KK_0$ is an $\wt L$-transversal of $\cusp(\Kcirc)$ as well. 
According to \Cref{prop3C} and \Cref{prop3D} maximal extendibility holds with respect to $\Kcircd\lhd \wt L_d$. Since $[\wt L_d,\wt L_{d'}]=1$ for every $d,d'\in \DD$ with $d\neq d'$, this implies that maximal extendibility holds also \wrt $\Kcirc \lhd \wt L$ as $\wt L\leq \Spann<\wt L_d| d\in \DD>$. Since $\wt L/ \Kcirc$ is abelian by \Cref{lem36}, $\KK$ and $\TT$ are again $N\EL$-stable $\wt L$-transversals in $\cusp(K)$ and $\cusp(L)$, respectively. 
\end{proof}
We apply the following statement in order to construct some extension map \wrt $L\lhd N$ for $\cusp(L)$ satisfying Equation \eqref{incl} from \Cref{HC_form}.

\begin{prop}\label{cor4_10}
There exists a $V_\tD\rtimes \EL$-equivariant extension map $\II Lambda0@{\Lambda_0}$ \wrt $H\lhd V_\tD$. 
\end{prop}
This ensures Assumption \ref{cor_tool}(ii.1) with the choice made in Lemma \ref{lem34}.
\begin{proof} 
Recall $\ov V:=H\spann<\ov V_d\mid d\in \DD>$ and $V_\tD:=\ov V\cap \bG$, see \Cref{prop_lift}.
Let $\II Htildeeps@{\pwt H_{\epsilon}}:=\spann<\wt H_d\mid d\in \DD_\eps>$.
We apply the extension maps from \Cref{prop4_1}, \Cref{prop320} and \Cref{prop3D}(c) for constructing a $\ov V$-equivariant extension map for $H\lhd V_\tD$.  
Note that by the definition of $\ov V$, $\ov \nn_1 \in \ov V\setminus V_\tD$ whenever $\gamma\in \EL$, and then $\ov \nn_1$ and $\gamma$ induce the same automorphism on $\bG$ according to \Cref{gammane1}. By \Cref{prop_lift}, $[F_p,V_\tD]=1$. Altogether it is sufficient to prove that maximal extendibility holds \wrt $H\lhd \ov V$.

Let $\la\in\Irr(H)$ and $\wt \la\in \Irr(\wt H\mid \la)$. Then $\wt \la=\odot_{d\in\DD}\wt \la_{d}$ for some $\wt \la_d\in \Irr(\wt H_d)$ ($d\in\DD$).
Let $\psi_d$ be the extension of $\la_d:=\restr \la|{H_d}$ to $(V_d)_{\la_d}$ given by \Cref{prop4_1}, \Cref{prop320} and \Cref{prop3D}(c).

Assume $\la(h_0)=1$. Let $\ov \la\in \Irr(H/\spann<h_0>)$ be associated with $\la$. It is sufficient to show that $\ov \la$ extends to $\ov V_{\ov \la}/\spannh$. Since $[\ov V_d/\spannh,\ov V_{d'}/\spannh]=1$ according to \eqref{VIcommutator}, the group $\ov V/\spannh$ is the central product of the groups $\ov V_d/\spannh$. The characters $\psi_d$ ($d\in\DD$) define extensions $\ov \psi_d$ of $\ov \la_d$ to $(\ov V_d)_{\la_d}/\spannh$ and $\ov \psi:=
\odot_{d\in\DD}\ov\psi_{d}$ lifts to a character $\psi^\circ$ of $\Spann<(\ov V_d)_{\la_d}| {d\in \DD}>$. Recall $H\geq \spann<H_d\mid d\in\DD>$. According to \cite[4.1(a)]{S10b} we see that $ \la$ has an extension $\psi$ to $\ov V_\la$ such that 
$\restr \psi|{\Spann< (V_d)_{\la_d}| {d\in \DD}>_{\la}} = \restr \psi^\circ|{\Spann< (V_d)_{\la_d}| {d\in \DD}>_{\la}}$. 
The extension map \wrt $H\lhd V_\tD$ for $\Irr(H\mid 1_{\spannh })$ obtained this way is then automatically $\ov V\rtimes \spann<F_p>$-equivariant. 

Assume otherwise $\la(h_0)=-1$. As in \ref{lem3_9} let $\II Dodd@{\DD_\odd}:= \{ i\in \DD\mid i \,\,\, \odd\}$ and $\II Deven@{\DD_\even}:= \{ i\in \DD \mid i \,\,\, \even\}$. For $\eps\in \{ \odd,\even\}$ recall
\begin{align*}
\II Htildeeps @{\pwt H_{\epsilon}}&:=\Spann<\wt H_d|d\in \DD_\eps>,&\quad &
\II Heps@{H_{{\epsilon}}}:= \wt H_{\eps}\cap H_0, 
\end{align*}
and $H=H_\even\times H_\odd$, see \Cref{lem3_9}.
Analogously we define
\begin{align*}
\II Vovlerlineeps@{\ov V_\eps}&:=H_\eps \Spann<\ov V_d|d\in \DD_\eps>& \quad \und\quad &
\II Vtildeeps@{\pwt V_{\epsilon}}:=\wt H_\eps V_\eps.
\end{align*}
Notice that by this definition $V_\even\leq V_\tD$ and hence $V_\tD=H (V_\even. (V_\odd\cap V_\tD))$. 

Let $\wt \la_{\eps}:=\restr \wt \la|{\wt H_\eps}$ and $\la_{\eps}:=\restr \la|{ H_\eps}$. 
Since $[\ov V_d,\ov V_{d'}]=1$ for every $d\in\DD_\even$ and $d'\in\DD$ by \eqref{VIcommutator}, the extensions $\psi_d$ ($d\in \DD_\even$) allow us to define an extension of $\la_\even$ to $(\ov V_{\even})_{\la_\even}$.

Now $H_\even$ is the central product of the groups $H_d$ ($d\in\DD_\even$) and $(V_\even)_{\la_\even}$ is analogously the central product of the groups $(\ov V_d)_{\la_d}$. Hence the product of the characters $\psi_d$ ($d\in\DD_{\even}$) defines an extension $\wh \la_\even\in \Irr((V_\even)_{ \la_\even})$ of $\la_\even$.

In order to extend $\la_\odd$ to $(V_{\odd})_{\la_\odd}$ we first extend $\wt \la_\odd:=\restr \wt \la|{\wt H_\odd}$. Again $\wt \la_\odd$ is the central product of characters $\wt \la_d$ ($d\in\DD_\odd$). According to \Cref{prop4_1}(b) and \Cref{prop320}(b) we have $(\ov V_d)_{\wt \la_d}\leq V_\tD$. The same holds also for $d=-1$ by straight-forward calculations.

Let $\nu \in \Irr(\wt H_\odd)$ with $\ker(\nu)=H_\odd$. 
According to \Cref{prop320}(b) there exists some element $c_d\in \ov V_{d}$ such that $(\ov V_d)_{\la_d}=\spann<(\ov V_d)_{\wt \la_d}, c_d>$, which satisfies $\wt \la_d^{c_d}=\wt \la_d \restr \nu|{\wt H_d}$. The extensions $\restr \psi_d|{(\ov V_d)_{\wt \la_d }}$ define easily extensions $\psi'_d$ of $\wt \la_d$ to $\wt H_d (\ov V_d)_{\wt \la_d }$.
The restriction $\restr \psi'_d|{(V_d)_{\wt\la_d}}$ is then $c_d$-stable.
Since $(\ov V_d)_{\wt \la_d}$ is contained in $V_\tD$ the group $(\ov V_\odd)_{\wt \la_\odd }$ is the central product of the groups $\wt H_d (\ov V_d)_{\wt \la_d}$ ($d\in\DD_\odd$). The product $\psi':=\prod_{d\in\DD_\odd}\psi'_d$ determines uniquely an extension $\psi''$ of $\wt \la_{\odd}$ to $\wt H_\odd (\ov V_\odd)_{\wt \la_\odd}$.
Routine calculations show that $(\ov V_{\odd})_{\la_\odd}=(\ov V_\odd)_{ \wt \la_\odd }\spann<c_\odd>$ where $c_\odd=\prod_{d\in\DD_\odd} c_d$. The character $\restr \psi''|{H_\odd(\ov V_\odd)_{\wt \la_\odd}}$ is then $c_\odd$-stable and extends to $(\ov V_{\odd})_{\la_\odd}$. This way we obtain an extension $\wh \la_\odd$ of $\la_\odd$ to $(\ov V_\odd)_{ \la_\odd}$.

Recall $[\ov V_\odd, \ov V_\even]=1$. Hence the extensions $\wh \la_\odd$ and $\wh \la_\even$ determine an extension of $\la$ to $\ov V_\la$, see \cite[Lem.~4.2]{S09}. 
\end{proof}
\noindent 
In the next step we show that there exists an extension map \wrt $\Kcirc\lhd \Kcirc \rtimes \epsilon(V_\tD)$ for
the set $\KK_0$ from \Cref{def_TT} as required in Proposition \ref{cor_tool}. 

\begin{prop} \label{propLaeps}
There exists a $ V_{\tD} \EL$-equivariant extension map $\II Lambdaeps@{\Lambda_\epsilon}$ \wrt $\Kcirc\lhd \Kcirc \rtimes \epsilon(V_\tD)$ for $\KK_0$, where $\epsilon: V_{\tD} \EL \rightarrow V_{\tD}\EL/H$ is the canonical morphism. 
\end{prop} 
\begin{proof} 
By \Cref{prop3C} and \Cref{prop3D} there exist $\ov V_d\spann<F_p>$-equivariant extension maps $\Lambda_{\eps_d}$ \wrt $K_{0,d}\lhd K_{0,d} \rtimes \eps(\ov V_d)$ for $\TT_d^\circ$, whenever $d\in \DD$ with $d\neq 1$.
Note that the case $d=1$ is trivial since $K_{0,1}=1$. 
The group $\Kcirc\rtimes \eps(\ov V)$ is the direct product of the groups $K_{0,d} \rtimes \eps_d(\ov V_d)$. Using the 
maps $\Lambda_{\eps_d}$ ($d\in\DD$) we therefore obtain an extension map $\Lambda_\eps$ as required. 
\end{proof}
This leads to the following statement. We use the set $$\KK:=\Irr(K\mid \KK_0)$$ with $\KK_0$ from \Cref{def_TT}. For the application of \Cref{cor_tool} we use the group $M=KV_\tD$, see also \Cref{cor322}. 
\begin{prop} \label{prop329} There exists a $V \EL$-equivariant extension map $\II LambdaKM@{\Lambda_{K\lhd M}}$ \wrt $K\lhd M$ for $\KK$. 
\end{prop}
\begin{proof}
By the above all the assumptions of \Cref{cor_tool} are satisfied. 
The groups satisfy the required assumptions in \ref{cor32i} according to \ref{cor322}. Using the set $\KK$, given as $\Irr(K\mid \KK_0)$ from \Cref{def_TT} the set $\KK_0$ coincides with $\bigcup_{\la\in\KK}\Irr(\restr \la|{\Kcirc})$. With the $V_\tD \EL$-equivariant extension map $\Lambda_0$ for $H\lhd V_\tD$ from \Cref{cor4_10} and the extension map $\Lambda_\eps$ for $\Kcirc \lhd \Kcirc \rtimes \eps(\ov V)$ from \Cref{propLaeps} the Assumption \ref{cor23ii} is satisfied. The application of this statement implies the result. 
\end{proof}
For the set $\TT$ defined as $\Irr(L\mid \KK)$ in \Cref{def_TT} we verify that there exists an $N\EL$-equivariant extension map \wrt $L\lhd N$ for $\TT$.
\begin{proof}[Proof of Theorem~\ref{stabcuspN} and Theorem~\ref{theoC}]
For the proof it is sufficient to construct for every $\la\in\TT=\Irr(L\mid \KK)$ some $N\EL$-stable extension to $N_\la$. A character $\la\in\TT$ lies above a unique $\la_0\in \KK=\Irr(K\mid \KK_0)$. Moreover some extension $\wt \la_0\in\Irr(L_{\la_0})$  to $L_{\la_0}$ satisfies $(\wt \la_0)^L=\la$. 
By the properties of $\KK$ we see $N_{\la_0}=L_{\la_0} M_{\la_0}$, which is normalized by $(N\EL)_{\wt\la_0}$-.  By \Cref{prop329} the character $\la_0$ has a $(\ov V \spann<F_p>)_{\la_0}$-stable extension to $M_{\la_0}$. According to \cite[4.1]{S10b} this defines an extension $\phi$ of $\wt \la_0$ to $N_{\wt\la_0}$ since $N_{\wt \la_0}\leq L_{\la_0}M_{\la_0}$. By the construction we see that $\phi^{N_{\la}}$ is an extension of $\la$.

As $\TT$ is an $M$-stable $\wt L$-transversal, $\wt N_{\la_0}=\wt L_{\la_0} M_{\la_0}$ and $(\wt N\EL)_{\la_0}=\wt L_{\la_0} \wh M_{\la_0}$. Hence this extension of $\la_0$ defines an extension of $\la$ as required. \end{proof}
Later this ensures Assumption \ref{prop23ii}. 

%%%%%%%%%%%%%%%%%%%%%%%%%%%%%%%%%%%%%%%%%%%%%%%%%%%%
\begin{remark}
While \Cref{thm_loc}(b) assumes $q$ to be odd, the proof would give a similar conclusion in the other case. For even $q$ every $\chi\in \Irr(G)$ satisfies $(\wt G E)_\chi=\wt G_\chi E_\chi$ since $\wt G=G$ in the notation of \ref{not}. Nevertheless the conclusion of \Cref{thm_loc}(b) holds as well. We observe that the arguments from before prove that there exists some $N \EL$-equivariant extension map $\Kcirc\lhd \Kcirc V_\tD$ for $\cusp(L)$, where $V_\tD$ is isomorphic to $N/L$ and is defined as before with $1=\pm \varpi$ in the argument of the Chevalley generators. 
\end{remark}

%%%%%%%%%%%%%%%%%%%%%%%%%%%%%%%%%%%%%%%%%%%%%%%%%%%

\section{More on cuspidal characters}\label{sec4}
In order to prove our main theorem we need more specific results on cuspidal characters, especially with regard to automorphisms. We keep $q$ a power of an odd prime.
\begin{prop} \label{prop51} 
Let $n\geq 3$, $\chi\in\cusp(\GL_n(q))$ and $\gamma \in \Aut(\GL_n(q))$ given by transpose-inverse up to some inner automorphism.
	\begin{thmlist}
	\item If $\chi^\gamma=\chi$ , then $2\mid n$ and $\Z(\GL_n(q)) \leq\ker(\chi)$.
\item If $\chi^\gamma = \chi \delta$ for $\delta \in \Irr(\GL_n(q))$ a linear character of multiplicative order 2, then $2\mid n$.
\end{thmlist}
\end{prop}

\begin{proof} Let us recall the form of elements of $\cusp(\GL_n(q))$, see also \cite[16.1]{Cedric}. We write $\bK:=\GL_n(\FF)$  and $\bK^*:=\GL_n(\FF)$ as the dual with $\FF_q$-structures given by $F$. Let $s\in (\bK^*)^F=\GL_n(q)$ be such that the Lusztig series $\cE(\bK^F,(s))$ associated to $s$ contains a cuspidal character. Combining \cite[3.2.22]{GM} and the fact that the group $\cent{\bK^*}{s}^F$ of type A can have cuspidal unipotent characters only when it is a torus (see for instance \cite[Ex.~2.4.20]{GM}), we get that $s$ is regular and $\Cent_{\bK^*}(s)$ is a Coxeter torus. This can be summed up in the fact that the spectrum of $s$ is a single orbit of length $n$ under $F$, or equivalently $\FF_{q^{}}{[\zeta]}=\FF_{q^{n}}$ for any eigenvalue $\zeta$ of $s$. Concerning the action of $\gamma$, note that an element of $\cE(\bK^F,(s))$ is sent to an element of $\cE(\bK^F,(s^{-1}))$ (apply \cite[3.1]{CS13}).

For the proof of (a) let  $\chi\in \cE(\bK^F,(s))$ be invariant under $\gamma$. Then $s$ and $s^{-1}$ have the same spectrum. If $1$ or $-1$ is an eigenvalue of $s$, then $s\in \{\Id_n,-\Id_n\}$ and $n=1$ since the eigenvalues of $s$ form a single $F$-orbit. This is impossible so inversion is without fixed point on the spectrum of $s$. This implies that $n$ is even and that the product of the eigenvalues of $s$ is $1$. So $s\in [\bK^*,\bK^*]^F$ and this implies that all characters of $\cE (\bK^F,(s))$ have $\Z(\GL_n(q))$ in their kernel (see \cite[p. 207]{CE04}).

For the proof of part (b) note that by the assumptions $q$ is odd and $\SL_n(q)$ is perfect, see \cite[24.17]{MT}. By the correspondence induced by duality between (linear) characters of $\bK^F/[\bK ,\bK]^F$ and elements of $\Z (\bK^*)^F$  (see for instance  \cite[11.4.12]{DiMi2}) we have  $\delta \cE(\bK^F,(s)) = \cE(\bK^F,(-s))$. Assuming $\chi^\gamma = \chi \delta$, the same considerations as above show that $s$ and $-s^{-1}$ have the same eigenvalues. The spectrum of $s$ is of the form $\{ F(\zeta), F^2(\zeta), \dots ,F^n(\zeta)=\zeta  \}$ with $\FF_{q^{}}[\zeta]=\FF_{q^{n}}$.  Since $s$ and $-s^{-1}$ have the same eigenvalues, then $-\zeta^{-1} = F^a(\zeta)$ for some $1\leq a\leq n$. We have $F^{2a}(\zeta)=-F^a(\zeta)^{-1}=\zeta$ and therefore $\FF_{q^{2a}}\supseteq \FF_{q^{}}[\zeta]=\FF_{q^{n}}$. Then $n$ divides $2a$. Assume now that $n$ is odd. This implies that $n$ divides $a\leq n$. So $a=n$ and $-\zeta^{-1} =F^n(\zeta)=\zeta$. But then $\zeta^2=-1$ and $\FF_{q^{}}[\zeta]\subseteq \FF_{q^{2}}$ which contradicts $n\geq 3$. So we get our claim that $2\mid n$.
\end{proof}

\noindent 
The following statement is used later for computing the relative Weyl groups associated to cuspidal characters of a Levi subgroup of $\tDlsc(q)$.
\begin{prop}\label{prop_SL2_cusp} 
	Let $n\geq 2$, $\psi\in\cusp(\SL_n(q))$ and $\gamma \in \Aut(\GL_n(q))$ given by transpose-inverse up to some inner automorphism. If $ |\GL_n(q):\GL_n(q)_\psi|$ is even and $\psi^\gamma=\psi$, then $n=2$ and $\psi$ is one of the two characters $\R'_\si(\theta_0)$ ($\si\in\{\pm 1\}$) of degree  $\frac{q-1}{2}$ from Table 5.4 of \cite{BonSL2}.
\end{prop}
\begin{proof}
	According to Table 5.4 of \cite{BonSL2} the two characters $\R'_\si(\theta_0)$ ($\si\in\{\pm 1\}$)  are the only characters of $\SL_2(q)$ that are  cuspidal and not $\GL_2(q)$-stable. The characters $\R(\theta)$ given there are $\GL_2(q)$-stable and the other characters $R_\si(\alpha_0)$ ($\si\in\{\pm 1\}$) are in the principal Harish-Chandra series. Note that $\gamma$ then restricts to an inner automorphism of $\SL_2(q)$.
	
	Now consider $n\geq 3$. Let $\psi$ be as in the proposition, let $\chi\in \Irr(\GL_n(q)\mid\psi)$, so that $\chi$ is cuspidal thanks to Lemma~\ref{lem_cuspc}.
	We keep the notation of the proof of \Cref{prop51} with $\chi\in\cE (\GL_n(q),(s))$ and $\zeta$ some eigenvalue of $s$.
	
	By Clifford theory $\chi$ is induced from a character of $\GL_n(q)_\psi$. Then the assumption $2\mid |\GL_n(q):\GL_n(q)_\psi|$ implies $\chi= \nu_1\chi$ for $\nu_1\in\Irr(\GL_n(q))$ the linear character of order $2$ with kernel containing $\SL_n(q)$. Hence $s$ is $\GL_n(q)$-conjugate to $-s$. Then $-\zeta\in\{ F(\zeta), F^2(\zeta),\ldots, F^n(\zeta)=\zeta\}$ since this is the spectrum of $s$.

	Clifford theory also tells us that the assumption $\psi^\gamma=\psi$ implies $\chi^\gamma=\nu_2\chi$ for some linear character $\nu_2$ of $\GL_n(q)$ with $\SL_n(q)$ in its kernel. Then $s^{-1}$ is conjugate to $\lambda s$ for some $\lambda\in\FF_q^\times$. As before we obtain $\zeta^{-1}\in \{\lambda F(\zeta), \lambda F^2(\zeta), \ldots, \lambda F^n(\zeta)=\lambda \zeta \}$. 
	
	We can now write $-\zeta=F^a(\zeta)$ and $\lambda \zeta^{-1}=F^b(\zeta)$ for $1\leq a,b\leq n$. The first equality gives $F^{2a}(\zeta)=-F^a(\zeta)=\zeta$ and the second $F^{2b}(\zeta)=\lambda F^b(\zeta)^{-1}=\zeta$ since $\lambda\in\FF_q$. So $\zeta\in \FF_{q^{2a}}\cap \FF_{q^{2b}}$, but since $\FF_q[\zeta]=\FF_{q^{n}}$ we get that $n$ divides both $2a$ and $2b$. The latter are at most $ 2n$, so $2a,2b\in \{ {n} ,2n\}$. Having $a=n$ would imply $-\zeta=F^n(\zeta)=\zeta$ which is impossible because $q$ is odd. So $n$ is even and $a={\frac n 2}$. 
	On the other hand, if $b=n$ then $\zeta=F^n(\zeta)=\lambda \zeta^{-1}$ and therefore $\zeta^2\in \FF_q$. Then $\FF_q[\zeta]\subseteq \FF_{q^{2}}$ and this implies $n=2$. 
	
	There remains the case when $b={\frac n 2}=a$. Then $\lambda \zeta^{-1}=F^a(\zeta)=-\zeta$ and again $\zeta^2\in \FF_q$. This yields $n=2$ as seen before.
\end{proof}
We complement the above by a result on cuspidal characters of $\tDlsc(q)$, which follows from a combination of results from \cite{Ma17} and \cite{S12}. We use $\bG$, $F$, $\gamma$ from Notation~\ref{not_32} and $h_0$ from Notation~\ref{defhI}. Recall the Lang map $\calL$  defined on $\bG$ by  $\calL(g)=g^{-1}F(g)$.  Note that $\calL^{-1}(\spannh)/\spannh=(\bG/\spannh)^F=\SO_{2l}(\FF_q)$.
\begin{prop}\label{prop5E}
Recall  $\wt G:=\calL^{-1}(\Z(\bG))=\norm{\bG}{\GF}$, see \Cref{rem_whG}. If $\la\in\cusp(\GF\mid 1_{\spannh})$ with $\wt G_\la\leq \calL^{-1}(\spannh)$, then $\gamma$ acts trivially on $\la$ and $\Irr(\calL^{-1}(\spannh) \mid \la)$.
\end{prop}

\begin{proof} Recall that  a character of $\wt\bG^F$ is called semisimple when it corresponds to a trivial unipotent character through the Jordan decomposition of characters. The components of their restrictions to $\bG^F$ are also called the semisimple characters of $\bG^F$. In particular both are of degree prime to $p$ (see \cite[2.6.11]{GM}).
	
According to \cite[Thm.~1]{Ma17} there exists a semisimple character $\rho\in\Irr(\GF)$ with $(\wG E)_\rho=(\wG E)_\la$, where $\rho$ and $\la$ lie in the same rational Lusztig series. We use now results from \cite{S12} to investigate $\rho$ further.  In a first step we prove that $\gamma$ acts trivially on $\rho$ and $\Irr(\calL^{-1}(\spannh) \mid \rho)$.

We assume that $\bG$, $\bT$ and $\Delta$ are as given in \ref{not_32} and let $\bU:=\spann<\bX_\al\mid \al\in \Delta>$ and  $\bB:=\bT\bU$. 
As group $\wt\bG$ introduced in \ref{not} we use the particular choice from \cite[3.1]{S12}. 
Then $\wGF$ and $\wt G$ induce the same automorphisms on $\GF$. Let $\wt \bB:=\bB\Z(\wbG)$. Let 
$\wt \Omega:\Irr_{p'}(\wGF)\lra \Irr_{p'}(\wt\bB^F)$ 
be the $\Irr(\wGF/\GF)\rtimes E(\GF) $-equivariant bijection 
with $\Irr(\restr \wt \psi|{\Z(\wGF)})= 
\Irr(\restr \wt\Omega(\wt \psi)|{\Z(\wGF)})$ for every $\psi\in\Irr_{p'}(\wGF)$ from \cite[3.3(a)]{S12}. 

Let $\wt \rho\in\Irr(\wt \bG^F\mid \rho)$, $\wt \phi := \wt \Omega(\wt \rho)$ and $\phi\in\Irr(\restr \wt \phi |{ \bB^F})$. 
Let $C$ be the Cartan matrix associated to $\Delta$ and $C^{-1}=(c'_{\al \beta})$ its inverse.
Let $\zeta\in \FF^\times$ be a root of unity of order $\det(C)  (q-1)=4(q-1)$. 
For $\al\in\Delta$ we set 
$$t^{(0)}_\al:=\prod_{\beta\in\Delta} \h_\beta(\zeta^{\det(C) c'_{\al \beta}}),$$ 
see also \cite[8.1]{Maslowski}. Then we can choose elements 
$t_\al\in t^{(0)}_\al \Z(\wbG)\cap \wt \bT^F$ such that $\wt \bT^F=\Z(\wGF)\spann<t_\al\mid \al \in \Delta>$, see \cite[\S 8]{Maslowski}. 
Assume that $\Delta$ is given as in \ref{not_32} and let $\al\in \{e_2\pm e_1\}$. 
The entries of $C^{-1}$ can be found in \cite[p. 296]{OV}. We see $\calL(t^{(0)}_\al)=(t^{(0)}_\al)^{q-1}\in \h_\ul(\varpi)\spannh$. Hence
 $t_\al$ induces a diagonal automorphism of $\GF$ associated to some element in $\h_\ul(\varpi)\spannh$ in the notation of \Cref{rem_whG}(b).

We abbreviate $\II Ghat@{\pwh G}:=\calL^{-1}(\spannh)$. %(Note that $\wh G$ then differs from the group introduced in \ref{rem_whG}.)
The assumption $\wt G_\la\leq \wh G$ implies $\la^{t_\al}\neq \la$. Via the construction we have $(\wt GE)_\rho=(\wt GE)_\la$ and hence $ \rho^{t_\al}\neq \rho$. By Clifford theory the character $\wt \rho$ satisfies $\wt \rho (t_\al \GF)=0$ and is stable under multiplying with linear characters with kernel $$\{g \in \wGF\mid g \text{ induces diagonal autom. of $\GF$ associated to an element of} \spannh\} ,$$ see \Cref{rem_whG}(b).
As $\wt \Omega$ is an $\Irr(\wGF/\GF)$-equivariant bijection, the character $\wt \phi$ has to satisfy $\wt \phi(t_\al)=0$ as well. 
As in Remark \ref{rem_whG}, $\wt \phi$ can be extended to some character $\kappa$ on $\wt \bB^F\wt Z=\wt B_0.\wt Z $, where $\wt \calL$ is the Lang map on $\wt \bG$,  $\wt Z:=\wt \calL^{-1}(\Z(\bG))\cap \Z(\wbG)$ and $\wt B_0:= \calL^{-1}(\Z(\bG))\cap \bB$. Note that ${\wt B_0}=\Z(\GF)\spann<t_\beta^{(0)}\mid \beta\in \Delta> \bU^F$. 
Then $\kappa(t_\al^{(0)})=0$. The character $\restr \kappa |{\wt B_0}$ is $\gamma$-stable, since $\Irr(\restr\kappa|{\Z(\GF)})=\Irr(\restr\phi|{\Z(\GF)})$ is $\gamma$-stable because of $h_0\in \ker(\phi)$ and $t_{\beta}^{(0)}$ is $\gamma$-fixed for every $\beta\in\Delta\setminus \{ e_2\pm e_1\}$ according to the explicit value of $C^{-1}$. 

As $\kappa$ is $\gamma$-stable,  $\phi\in \Irr(\restr\kappa |{\bB^F})=\Irr(\restr\wt \phi |{\bB^F})$ can be assumed to be $\gamma$-stable, see \cite[3.6(a)]{S12}.
By Clifford theory $\restr \kappa|{\wt B_0}$ is of the form $\wh \phi^{\wt B_0}$ for a unique $\wh \phi\in \Irr((\wt B_0)_{\phi})$.
As $\phi$ extends to $\wt \bB^F_\phi$ according to \cite[Thm.~3.5(a)]{S12}, the character $\wh \phi$ is an extension of $\phi$. As $\kappa$ and $\phi$ are $\gamma$-stable, $\wh\phi$ is $\gamma $-stable. Note that $(\wt B_0)_{\phi}\leq \wh B:=\calL^{-1}(\spannh)\cap \bB$. 

Via the statement given in \Cref{gammastablewhG_extension} some $\wG$-conjugate $\rho'$ of $\rho$ is $\gamma$-stable and has also a $\gamma$-stable extension to $\wh G$. If $\rho'\neq \rho$ we observe that 
$(\wh G \spann<\gamma>)^{t'}=\wh G \spann<\gamma>$ for every $t'\in \calL^{-1}(\h_\ul(\varpi))$. Hence the character $\rho$ extends to $\wh G \spann<\gamma>$, as well.  

We deduce from this result on $\rho$ the analogous property of $ \la$. 
Recall that $\la$ and $\rho$ are in the same rational Lusztig series and that $(\wGF E)_\rho=(\wGF E)_\la$, in particular $\wGF_\rho=\wGF_\la$. Recall $\wt \rho\in\Irr(\wGF\mid \rho)$ and $\Irr(\wGF/\GF)$ acts on $\Irr(\wGF)$ by multiplication with linear characters.  As $\wGF/\GF$ is abelian and maximal extendiblity holds \wrt $\GF\lhd \wGF$ we see 
$$\Irr(\wGF/\GF)_{\wt \la}=\Irr(\wGF/\wGF_\la)= \Irr(\wGF/\wGF_{ \rho})=\Irr(\wGF/\GF)_{\wt \rho}.$$ 

Let $\cE(\wGF,s)$ be the rational Lusztig series containing $\wt \rho$.
The character $\wt \rho$ is semisimple. The series $\cE(\wGF,s)$ contains exactly one regular character, see \cite[12.4.10]{DiMi2}. By the definition of semisimple and regular in \cite[12.4.1]{DiMi2} we see that there exists also a unique regular character in that series. 
Let $\wt \rho'\in\Irr(\wGF)$ be the Alvis--Curtis dual of $\wt \rho$ up to a sign, see \cite[7.2]{DiMi2}. 
Then $\{\wt \rho'\}=\Irr(\Gamma^{(\wGF)})\cap \cE(\wGF,s)$, where $\Gamma^{(\wGF)}$ denotes the Gelfand--Graev character of $\wGF$. Since it vanishes outside unipotent elements, the Gelfand--Graev character is stable under $\Irr(\wGF/\GF)$. Hence $\Irr(\wGF/\GF)_{\wt \rho'}$ coincides with the stabiliser of $\cE(\wGF,s)$ in $\Irr(\wGF/\GF)$. This group is called 
$B(s)$ in \cite[15.13]{CE04}. By the construction of Alvis--Curtis duality this implies $\Irr(\wGF/\GF)_{\wt \rho}=B(s)$. The characters $\wt \rho^\gamma$ and $\wt \la^\gamma$ belong to $\cE(\wGF,\gamma^{-1}(s))$. As $\la$ and $\rho$ are $\gamma$-stable, $\wt \rho^\gamma = \wt \rho \mu$ and $\wt \la^\gamma=\wt \la \mu'$ for linear characters $\mu,\mu'\in\Irr(\wGF)$. Since $\wt \rho^\gamma$ and $\wt \la^\gamma$ are in the same rational series, $\mu\in \mu'B(s)$ or equivalently $\restr \mu|{\wGF_\rho}= \restr \mu'|{\wGF_{\la}}$. 

Because of $\Irr(\restr\wt \rho|{\Z(\wGF)})=\Irr(\restr\wt \la|{\Z(\wGF)})$, \Cref{gammastablewhG_extension} allows to conclude that $\la$ has a $\gamma$-stable extension to $\wh G$ as $\rho$ has such an extension. 
\end{proof}
%%%%%%%%%%%%%%%%%%%%%%%%%%%%%%%%%%%%%%%%%%%%%%%%%%%%%%

\section{Character theory for the relative inertia groups $W(\la)$}\label{secWla}
The aim of this section is to ensure Assumption \ref{loc*}, namely to prove  (a main step towards) the following statement. 
\begin{theorem}\label{thm_loc*} 
Let $l\geq 4$. Let $\bG^F=\tDlsc(q)$ with odd $q$ and $L=\bL^F$ a standard Levi subgroup of $\GF$ (see Notation~\ref{not_32}). Let $N$, $\wt N':= \wt T_0 N$ and $\EL:=\Stab_{E(\GF)}(\bL^F)$ be associated to $L$ as in \Cref{not_L_N_E1}. If \Cref{hyp_cuspD_ext} holds for every $l'$ with $4\leq l'< l$, then there exists some $\EL$-stable $\wt N'$-transversal in $\cusp(N)$.
\end{theorem}
Some technicalities (mainly in the case where $\bG=\tD_{4,sc}(\FF)$) delay the complete proof until \Cref{ssec6E}.
We construct the $\EL$-stable $\wt N'$-transversal as a subset of $\Irr(N\mid \TT)$, where $\TT$ is the $\wh N$-stable ${\wt N'}$-transversal from \Cref{thm_loc}(a). In \Cref{lem6_3} we find some $\EL$-stable $\wt N'$-transversal  of $\Irr(N\mid \{ \la\in \cusp(L)\mid \wt L'_\la=L \})$ where $\wt L'=\wt T_0 L$ as in \Cref{thm_MS}. 

In order to find the  transversal of $\Irr(N\mid \{ \la\in \cusp(L)\mid \wt L'_\la\neq L \})$ with the required properties we apply the strategy mapped out by \Cref{prop23_neu}, itself based on the parametrization of \Cref{propcuspN}. Thanks to \Cref{thm_loc} the two first assumptions of \Cref{prop23_neu} can be assumed, in particular there exist some  $\wh N$-equivariant extension map $\Lambda_{L,\TT}$ \wrt $L\lhd N$ for $\TT$, where $\wh N=N\EL$. We have to ensure the remaining assumption \ref{prop23_neu}(iii) and study the characters of the relative Weyl groups and their Clifford theory.

As already discussed in \ref{ssec1C} characters in such a  transversal have a stabilizer in $\wt N' \EL$ with a specific structure, namely such a $\psi\in \Irr(N)$ satisfies
\begin{align} \label{eq5}
(\wh N \wt L')_\psi&=\wh N_\psi \wt L'_\psi, 
\end{align}
see also \Cref{*cond_trans}. 
For studying a character $\psi\in\Irr(N\mid \TT)$ we apply the parametrization $\Upsilon$ from \Cref{propcuspN}(a) and the extension map $\Lambda_{L,\TT}$ \wrt $L\lhd N$ for $\TT$. Then  $\psi =\Upsilon(\ov{(\la,\eta)})=(\Lambda_{L,\TT}(\la)\eta)^N$ with $\la\in \TT$ and $\eta\in\Irr(W(\la))$. 
According to \Cref{prop23} the character $\psi=\Upsilon(\ov{(\la,\eta)})$ 
satisfies Equation \eqref{eq5} if
\begin{center}
$\eta$ is $\wh K(\la)_{\eta_0}$-stable, where $\wt \la\in\Irr(\wt L'_\la\mid \la)$ and $\eta_0\in\Irr(W(\wt \la))$,
\end{center}
where $\wh W=N\EL/L$ and $\wh K(\la)=\wh W_\la$. The aim of this section is \Cref{corsec5}, namely to prove that for every $\la\in \TT$, $\wt \la\in\Irr(\wt L'_\la\mid \la)$ and  $\eta_0\in\Irr(W(\wt \la))$
\begin{center}
there exists some  $ K(\la)_{\eta_0}$-stable $\eta \in \Irr(W(\la)\mid \eta_0)$,
\end{center}
where $K(\la)$ is the group from \ref{not52}. According to \Cref{KlawhKla} such a character $\eta$ is also  $\wh K(\la)_{\eta_0}$-stable, whenever $\bG$ is not of type $\tD_4$. 

In the proof some arguments depend on the group $\wt L'_\la$.  As in \Cref{rem_whG} we relate the group $\wt L'$ to subgroups of $\bG$. 
\begin{notation}
Recall the definitions $\wt L:=\calL\inv(\Z(\bG))\cap \bL$ and
$\wh L:=\calL\inv(\spannh)\cap \bL$ from \Cref{lem36}, where $\calL:\bG\ra \bG$ is given by $x\mapsto x\inv F(x)$. Recall $\wt N':=\wt T_0 N$ and set analogously $\II Ntilde @{\pwt N}:=\wt L N$. 
Then $\wt L'$ and $\wt N'$ induce on $\bG$ the same automorphisms as $\wt L$ and $\wt N$, respectively.
\end{notation}
Note that by an application of Lang's theorem $\calL(\bL)=\bL\supseteq\Z(\bG)$ so that $L\lneq \wh L\lneq \wt L$. 

For $\la \in\TT$ the characters of $W(\la)$ and $W(\wt \la)$ defined as above will be investigated depending on the value of $\wt L_\la$. 
The set $\cusp(L)$ can be partitioned in the following way: 
\[\cusp(L) = \MM^{(L)} \sqcup \MM^{(\whL)} \sqcup \MM^{(\wtL)} \sqcup \MM_0, \]
where $\II MX@{\MM^{(X)}}:=\{ \la \in \cusp (L)\mid \wt L_\la=X\}$ for any subgroup $L\leq X\leq \wt L$ and 
$\II M0@{\MM_0}:=\cusp(L)\setminus (\MM^{(L)} \cup \MM^{(\whL)}\cup \MM^{(\wtL)})$.
(In case of $|\Z(\GF)|=2$ one has $\MM^{(L)}=\emptyset$.) Note that the sets are by definition ${\wt N} \EL$-stable as $L$, $\whL$ and $\wtL$ are $\wt N \EL$-stable. In the following we construct an $\EL$-stable $\wt N$-transversal in $\cusp(N\mid \MM')$ for each of those four given $N\EL$-stable subsets $\MM'\subseteq \cusp(L)$. 

\begin{lem} \label{lem6_3}Let $\II TL@{\mathbb T^{(L)}}:= \TT \cap \MM^{(L)}$. Then $\Irr(N\mid \TT^{(L)})$ is an $\EL$-stable $\wt N'$-transversal in $\Irr(N\mid \MM^{(L)})$. \end{lem}
\begin{proof}
The set $\TT$  is  by construction $N \EL$-stable and no two elements are $\wt L$-conjugate. Hence for $\la\in \TT^{(L)}$ we have $(N\EL \wt L)_\la=( N\EL)_\la$ by \Cref{*cond_trans}. By Clifford theory $\Irr(N\mid \TT^{(L)})$ is an $\wt N$-transversal in $\Irr(L\mid \MM^{(L)})$ and is $N\EL$-stable. \end{proof}
Determining an $ N \EL$-stable $\wt L$-transversal in $\Irr(N\mid \MM')$ for the other sets $\MM'$ is more involved. 
We start by some general descriptions of  $W(\la)$ and related groups for $\la \in \cusp(L)$, see \Cref{prop_sec4B}. Afterwards we collect some particular results on cuspidal characters. In the following two subsections we verify for characters of $W(\la)$ the above condition under the assumption that $\la \in \MM^{(\wt L)}\cup \MM_0$ or $\la\in \MM^{(\wh L)}$.

In \ref{ssec5E} we ensure a closely related condition on characters of $W(\la)$ for $\la\in\cusp(L)$ with $\wt L_\la=\wh L$.  
In Section~\ref{ssec6E} we show how these considerations prove \Cref{thm_loc*} and how this implies \Cref{thm1}.

\subsection{Understanding $\cusp(N)$ via characters of subgroups of $W$}
We start by recalling some basic notation and introducing subgroups of  ${\wh W}:=\wh N/L = N\EL/L$  as in \ref{not_L_N_E1}. Additionally let $\ov N:=\NNN_{\obG^F}(\bL)$ and $\II Woverline@{\ov W}=\ov N/L$, see also \Cref{prop64}.
\begin{notation}\label{not52}
Let $\bG$ and $F:\bG\ra\bG$ be as in \ref{not_32} with odd $q$. Let $\bL$ be the standard Levi subgroup of $(\bG,F)$ such that $L=\bL^F$. For any $J$ with $L\leq J\leq \wt L$ and $\la\in\ZZ\Irr(J)$ we set $W(\la):=N_\la/L$. If additionally 
$J$ is $\EL$-stable, $\ov W$ acts on $\Char(J)$, hence we can define $\ov W(\la):=\ov W_\la$ and $\II Klambda@ {K(\la)}:=\ov W_{\restr\la^{J \spannFp}|J}$.
\end{notation}
The groups $K(\la)$ and $\wh K(\la)$ are strongly related, in particular by the following result it is sufficient to consider $K(\la)$ instead of $\wh K(\la)$ if $\bG$ is not of type $\tD_4$. Recall that $\gamma$ is the graph automorphism of $\bG$ of order $2$ swapping $\al_1$ and $\al_2$. 
\begin{lem}\label{KlawhKla}
Let $E^\circ:=\spann<F_p,\gamma>$, $\II EcircL@{E^\circ_L}:=E^\circ\cap \EL$,  $\la\in\TT$, $\wt \la$ defined as above,   $\eta_0\in\Irr(W(\wt\la))$ and $\eta\in\Irr(W(\la))$. Then $\eta$ is $\wh K(\la)_{\eta_0}\cap(W\rtimes E^\circ _L)$-stable if and only if it is $ K(\la)_{\eta_0}$-stable. 
\end{lem}
\begin{proof}
Note $F_p\in \Z(\wh W)$, hence $\eta$ and $\eta_0$ are $F_p$-stable.  Recall $\ov W$ can be identified with the quotient $(\wh W\cap ( W \rtimes E^\circ_L ))/\spannFp$.  The group $\wh K(\la)\cap (W\rtimes E^\circ_L)$ then projects to $K(\la)$, i.e.  for every $w\in \ov W$ and $e\in \spannFp$ with $\la^{we}=\la$ we see $(\restr \la^{LE^\circ_L}|L)^w=\la$.
This implies $\wh K(\la)\spannFp \cap \ov W=  K(\la)$ and 
$\wh K(\la)_{\eta_0}\spannFp \cap \ov W=  K(\la)_{\eta_0}$. As $F_p$ stabilizes $\eta$ this implies the statement. 
\end{proof}

For $\la\in \TT$ and $\wt \la\in \Irr(\wt L_\la\mid \la)$ the group $W(\wt \la)$ is determined by $\restr \wt \la|L$, as $W$ acts trivially on the characters of $\wt L_\la /L$. Note that $ W(\wt \la)\neq W(\restr \wt \la|L)$ in general. We can work with the group $\wt L$ instead of $\wt L'$ because of the following observation. 
\begin{lem}\label{lem41}
Let $\la\in \cusp(L)$, $\wh \la\in\Irr(\wh L\mid \la)$ and $\wt \la\in\Irr(\wt L \mid \wh\la)$. 
\begin{enumerate}
	\item Then $\ov W(\wt \la)=\ov W(\wt \la') \und W(\wt \la)= W(\wt \la')$ for every $\wt \la'\in \Irr(\wt L' \mid \la)$.
	\item Then $\ov W(\wt \la)\leq \ov W(\wh \la) \leq \ov W(\la) \und W(\wt \la)\leq W(\wh \la) \leq W(\la)$. 
\end{enumerate}
\end{lem}
\begin{proof} By the construction of $\TT$, the character $\la\in \TT$ satisfies $(\wh N \wt L)_\la= \wh N_\la \wt L_\la$. Because of  $\wbG=\Z(\wbG)\bG$ the group $\wt L'$ is a subgroup of $\wt L \Z(\wbG)$ and $\wt L'=\bL\cap \wt L \Z(\wbG)$. This implies  $(N\EL)_{\wt \la}=(N\EL)_{\wt \la'}$ and hence part (a), see \Cref{rem_whG} for a similar argument. 

As $h_0$ is centralized by $\wh N$, the group $\wh L$  is normalized by $N$ and $\ov N$. The containments from part (b) follow from this by straightforward considerations. \end{proof}
\label{sec_computeWxi}
For $\la\in\TT$ and $\wt \la\in \Irr(\wt L \mid \la)$ we compute $\ov W(\wt \la)$ as an approximation of $W(\wt \la)$. 
For $I\in \cO$ and $d\in\DD$ we use the groups $\bG_I$, $L_I$, $\wh L_I$ and $\bL_d$ from \Cref{defhI},  Lemma~\ref{structL} and \Cref{lem36}. 
In  \Cref{lem36}  the structure of $\bL$ and some of its subgroups was already studied. Additionally we use the following properties of $\wh L$. 
\begin{prop}[The structure of $\wh L$]\label{prop55}
For $d\in \DD$ and $I\in\cO$ let $\II Lhatd@{\pwh L_d}:=\wh L\cap \bL_d$. 
\begin{thmlist}
\item $L_1$ is a split torus of rank $|J_1|$.
\item $\wh L$ is the central product of $\wh L_d$ ($d\in\DD$) over $\spannh$.
\item $\wh L_d$ is the central product of $\wh L_I$ ($I\in \cO_d$) over $\spannh$.
\item\label{commut} $[\wh L_I,\wh L_{I'}]=1$ for all $I,I'\in\cO$ with $I\neq I'$.
\end{thmlist}
\end{prop}
\begin{proof}
The first three parts follow from \Cref{lem36}(c). 
	
Part (d) is clear if $I\in \cO_1$ or $I'\in \cO_1$. Note that the groups $\wh L_I$ and $\wh L_{I'}$ contain the root subgroups for $\Phi_I$ and $\Phi_{I'}$, which are orthogonal to each other. At least one of them is of type $\tA_l$. Hence no non-trivial linear combination of roots from $\Phi_I$ and $\Phi_{I'}$ is a root itself. Hence by Chevalley's commutator formula we see that the commutator of the groups is trivial.
\end{proof}
We continue using the groups $\ov V_d$ from \ref{lem3_20} for the description of $\ov W(\wh \la)$.  We write $\II IrrcuspLhat@{\protect\cusp(\pwh L)}$ for $\Irr(\wh L\mid \cusp(L))$.
\begin{lem}[Characters of $\wh L$]\label{lem6_7}
Let $\wh\la\in\cusp(\wh L)$, 
$\wh\la_d\in \Irr(\restr\wh \la|{\wh L_d})$ for every $d\in\DD$ and $\wh \la_I\in \Irr(\restr \wh \la|{\wh L_I})$ for every $I \in \cO$. Then:
\begin{thmlist}
	\item $\wh \la = \odot_{d\in \DD}\wh \la_{d} $ and $\wh \la_d=\odot_{I\in \cO_d}\wh \la_{I}$ for every $d\in\DD$, 
	\item  $ \wh \la_d \in \cusp(\wh L_d)$ and $ \wh \la_I \in \cusp(\wh L_I)$,
	\item $\ov W(\wh \la)$ is the direct product of the groups $ \II Woverlinedlambda@{\ov W_{d}(\phat \lambda)}:=(\ov V_d)_{\wh \la}/H_d $ ($d\in \DD$), and 
	\item $(\ov V_d)_{\wh \la}/H_d = (\ov V_d)_{\wh \la_d}/H_d$. 
	\end{thmlist}
\end{lem}
\begin{proof}
The description of $\wh \la$ and $\wh \la_d$ in (a)  follows from the structure of $\wh L$ and $\wh L_d$ given in \Cref{prop55}. The characters $\wh \la_d$ and $\wh\la_I$ cover a cuspidal character of $L_d$ and $L_I$, respectively, by \Cref{lem_cusp} which then also gives (b). Considering the roots underlying $V_d$ and $\bL_{d'}$ we see that the Chevalley relations imply $[V_d,\bL_{d'}]=1$ for $d,d'\in \DD$ with $d\neq d'$. This implies the parts (c) and (d).
\end{proof} 
For a more explicit description of the groups $\ov W(\wh \la)$ we introduce some elements of $\ov V$ using the maps $\kappa_d$ ($d\in\DD$) from \Cref{lem3_20}. For $d\in\DD\setminus\{-1\}$, recall $\cO_d=\{I_{d,1},\ldots, I_{d,a_d}\}$ from \Cref{not2_17}.
\begin{notation}\label{not4_15} Let  $d\in\DD\setminus\{-1\}$ and  $\II cI@{c_{I_{d,j}}}:=\kappa_d(\nn_{e_j}(\varpi))\in \ov V_d$ for every $j\in \uad$. Note that for every $I\in \cO_d$, $c_{I}$ is some $\ov V_d$-conjugate of $\ov \nn_1^{(d)}$ and $\rho_{\bT}(c_I)=\prod_{i\in I}(i,-i)$, where $\rho_{\bT}: \ov N_0\ra \Sym_{\pm \ul}$ is the natural epimorphism, see before  \Cref{prop64}.
If $2\nmid |I|$ and $I\notin \cO_{-1}\cup \cO_1$, then by considerations as in the proof of \ref{lemVactKc}, $c_I$ acts as transpose-inverting on $L_I$ via the identification of $L_I$ with $\GL_{|I|}(q)$. 

We define additionally the subgroups
\begin{align}\label{42}
\II{VoverlinedS}@{\overline V_{d,S}}:=H_d\Spann<\kappa_d(\n_{e_i-e_{i+1}}(-1))| i\in \underline{a_d-1} >
\end{align}
and $\II VS@{V_S}:=\Spann< V_{d,S} |d\in\DD\setminus \{-1\}>$. Then $\rho_{\bT}(V_S(L\cap N_0))/\rho_\bT(L\cap N_0)= \Sym_{\cO}\leq \Sym_{\pm \cO}$.

If $-1\in\DD$, then we set $c_{J_{-1}}:={\ov{\mathbf n}}_1$ from \ref{not38}.
\end{notation}
Using the notation of permutation groups given in \ref{not_Sym} we identify the group $\ov W=\ov N/L$ with $\Sym_{\pm \cO}$. Computations in $\ov W$ show that $\ov V=H \spann<c_I\mid I \in \cO> V_S$.  
\begin{defi}\label{def_standardized}
Let $\wh \la\in\Irr(\wh L)$. 
 We call $\wh \la$ \emph{\color{ocre}{standardized}} if for every $I,I'\in\cO$ the characters $\wh \la_{I}$ and $\wh \la_{I'}$ are either $V_S$-conjugate or not $\ov V$-conjugate. 
 For such $\wh\la$, we call the characters in $\Irr(\restr\wh \la|L)$ also \emph{standardized}. \index{standardized character}
 \index{character! standardized}
\end{defi}
Computations show that every standardized character $\wh \la$ satisfies $\ov V_{\wh\la}= H \spann<c_I\mid I \in \cO>_{\wh \la} (V_S)_{ \wh \la } $ and every $\ov N$-orbit in $\Irr(\wh L)$ contains a standardized character. For a more explicit description of $\ov W_{d}( \wh \la)$ we introduce the following notation. 
\begin{notation}\label{not59} Let $E$ be a set and $M$ a subset of $2^E$, the set of all subsets of $E$. For $m'\subseteq E$
we write \emph{$\III{m'\subsubset M}$} if $m'\subseteq m$ for some $m\in M$.
\end{notation}
Using the notation of permutation groups given in \ref{not_Sym} we identify the group $\ov W=\ov N/L$ with $\Sym_{\pm \cO}$. In the following we describe $\ov W_d(\wh \la_d)$ as a subgroup of $\Sym_{\pm \cO_d}$. We use the Young-like subgroups of $\Sym_{\pm \cO_d}$ from \Cref{not_Sym} that are associated to a partition of $\cO_d$. 
\begin{lem}\label{lem510} Let $\la\in \cusp(L)$ be standardized. We set
	\[\II Oclambdahat@{\cO_{c}(\phat\lambda)}:=\{ I\in \cO\mid (\wh \la_I)^{c_I}=\wh \la _I \}.\] 
Let $\II Yhatlambda@{Y(\pwh\lambda)} \vdash \cO_{c}(\wh \la)$ and $
\II Y'hatlambda@{Y'(\pwh\lambda)}\vdash (\cO(\wh \la) \setminus \cO_{c}(\wh \la))$ be the partitions such that $\{I,I'\}\subsubset Y(\wh \la)$ 
or $\{I, I'\}\subsubset Y'(\wh \la)$ if and only if $\wh \la_{I}$ and $\wh \la_{I'}$ are $V_{S}$-conjugate. 
Then \[\ov W(\wh \la)= \Young_{\pm Y(\wh \la)} \times \Young_{Y'(\wh \la)},\] 
where $\Young_{\pm Y(\wh \la)}$ and $\Young_{Y'(\wh \la)}$ are defined as in \ref{not_Sym}.
\end{lem}
\begin{proof}
Note $Y'(\wh \la)\cup Y(\wh \la) \vdash \cO$.
As $\wh \la$ is standardized 
\begin{align*}
\ov W_d(\wh \la )&= \spann<(I,-I)\mid I\in \cO_d>_{\wh \la} \rtimes \spann<(I, I')(-I,-I')\mid I,I'\in \cO_d>_{\wh \la} \text{ for every } d\in \DD.  
\end{align*} This gives our claim.
\end{proof}
Let $\zeta\in\FF^\times$ with $\varpi=\zeta^{(q-1)_2}$ and $\II tI2@{t_{I,2}}:= \h_{I}(\zeta)$ for every $I \subseteq \ul$ as in \Cref{lem36}. For $I\in \cO\setminus \{J_{-1}\}$ the element $t_{I,2}$ satisfies $[\bL_I,t_{I,2}]=1$. 
\begin{lem}[Structure of $\ov W(\wt \la)$] \label{lem411}
Let $\la\in \cusp(L)$, $\wh \la\in\Irr(\wh L\mid \la)$, $\wt \la\in\Irr(\wt L\mid \wh \la)$
$\wh \la_I\in\Irr(\restr \wh \la|{\wh L_I})$ and $\wt \la_I\in \Irr(\wt L_I\mid \wh \la_I)$ ($I\in \cO_c(\wh\la)$). Assume that $\wh \la$ is standardized and $\wt L_{\wh \la}=\wt L$. We set 
\begin{align*}
\II Oc1lambdahat@{\cO_{c,1}(\pwh\lambda)}:=\{I \in \cO_{c}(\wh \la) \mid (\wt \la_I)^{c_I}=\wt \la_I \} \und 
\II Oc-1lambdahat@{\cO_{c,-1}(\pwh\lambda)}:=\{I \in \cO_{c}(\wh \la) \mid (\wt \la_I)^{c_I}\neq \wt \la_I\}.
\end{align*}
\begin{thmlist}
\item For  $I\in \cO_c(\wh\la)\setminus\{J_{-1}\}$ and $\eps=\pm 1$ we have $I \in \cO_{c,\eps} (\wh\la) \Leftrightarrow
\wh\la_I(t_{I,2}^2)= \epsilon \wh \la_I(1)$.
\item $ \ov W(\wt\la)\leq 
\Sym_{\pm \cO_{c,1}(\wh\la)}\times 
\Sym^\tD_{\pm \cO_{c,-1}(\wh\la)}$, more precisely 
\[\ov W(\wt\la)= 
\left(\Spann<(I,-I)|I\in \cO_{c,1}(\wh\la)> \Spann<(I,-I)(I',-I') | I,I'\in \cO_{c,-1}(\wh\la)>\right)
\rtimes \Young_{Y(\wh \la)\cup Y'(\wh \la)}.\] 
\end{thmlist}
\end{lem}
If the character $\wh \la$ is clear from the context we write $\II Ocepsilona@{\cO_{c,\epsilon}}$ instead of $\cO_{c,\eps}(\wh\la)$. 
\begin{proof} Note that the description of $\wt L$ given in \Cref{lem36}(e) shows that $\wh \la$ extends to $\wt L$ if and only if $\wh \la_I$ extends to $\wt L_I:=\wh L_I\spann< t_{I,2}>$ for every $I\in \cO$. 
	
We have  $t_{\ul,2}:=\prod_{I\in\cO} t_{I,2}$ , $\calL(t_{\ul,2})= \h_{\underline l}(\varpi)$ and $\wt L=\spann<\wh L, t_{\ul,2}>$, see \Cref{lem36}. This implies $\wt L\leq \spann<\wt L_I\mid I \in \cO>$. By the Chevalley relations we see $[V_{S},t_{\ul,2}]=1$. Let $I,I'\in \cO$ such that $\wh \la_I$ and $\wh \la_{I'}$ are $V_S$-conjugate. Then we can choose their extensions $\wt \la_I$ and $\wt \la_{I'}$ to  $\wt L_I$ and $\wt L_{I'}$ such that they are $V_S$-conjugate, as $\NNN_{V_S}(\wh L_I)= H \Cent_{V_S}(\wh L_I)$, and therefore $\wt \la_{I'}$ is uniquely determined by $\wt \la_{I}$. 

Let $\phi\in\Irr(\spann<\wt L_I\mid I \in \cO>)$ with $\restr \phi|{\wt L_I}=\wt \la_I$ for every $I\in \cO$. Without loss of generality we may assume $\restr \phi|{\wt L}=\wt \la$. By the above construction we have $(V_S)_{\phi}=(V_S)_{\wh \la}$. 
Because of $\ov V_{\wh \la}=H \spann<c_I\mid I \in \cO>_{\wh \la} (V_S)_{\wh \la}$, it is sufficient to determine $\spann<c_I\mid I \in \cO>_{\wt \la}$ for computing $\ov V_{\wt \la}$. 

Let $\mu_I\in\Irr(\wt L_I)$ be the linear character with $\ker(\mu_I)= \wh L_I$. For any $Q \subseteq \cO$ let $\mu_{Q}\in \Irr(\spann<\wt L_I\mid I \in \cO>)$ be the linear character with $\spann< L_I\mid I \in \cO> \leq \ker (\mu_Q)$ such that for every $I\in \cO$ the inclusion $\wt L_I\leq \ker(\mu _{Q })$ holds if and only if $I \notin Q$. Note that $\mu_{Q}(t_{\ul,2})=1$ if and only if $|Q|$ is even.

For $Q \subseteq \cO$ let $c_{Q }:=\prod_{I\in {Q }} c_I\in \ov V$. If $Q'\subseteq \cO_c(\wh \la)$, then $c_{Q'}\in \ov V_{\wh \la}$ and we see that $\phi ^{c_{Q'}}= \phi \mu_{Q' \cap \cO_{c,-1}}$. 
As $\mu_{{Q'} \cap \cO_{c,-1}} 
( t_{\ul,2} )=
(-1)^{|{Q' \cap \cO_{c,-1}}|}$ this leads to a proof of part (b), in particular
\[\ov W(\wt\la)= 
\left(\Spann<(I,-I)|I\in \cO_{c,1}>\ \ \Spann<(I,-I)(I',-I') | I,I'\in \cO_{c,-1}>\right)
\rtimes \Young_{Y(\wh \la)\cup Y'(\wh \la)}.\] 

Let $I\in \cO_c(\wh \la)\setminus \{J_{-1}\}$. Then $c_I$ acts by inverting on $\bT_I$, in particular $t_{I,2}^{c_I}=t_{I,2}^{-1}$ and $[c_I,t_{I,2}]=t_{I,2}^{-2}$. Because of $t_{I,2}\in \Z(\bL_I)$ we see that $[t_\ul,\ov V_d] \subseteq \Z(\bL)$. Any extension $\wt \la_I$ of $\wh \la_I$ to $\spann<\wh L_I,t_{I,2}>$ satisfies $\wt \la_I(t_{I,2})\neq 0$ since $t_{I,2} \in \Z(\wt L_I)$. 

Note $\nu_I\in\Irr(\restr\wt \la|{\Z(\wh L_I)})$ is linear. As $\wh \la_I$ is $c_I$-stable, $\nu_I$ has multiplicative order $1$ or $2$. We observe that $t_{I,2}^{c_I}= t_{I,2}^{-1}\in \Z(\wt L_I)$ and hence \[\wt \la_I(t_{I,2}^{c_I} )=\la_I(1)\nu_I(t_{I,2}^{c_I} ) = 
\wt\la(t_{I,2}) \nu_I( [t_{I,2},{c_I}])=
\wt\la(t_{I,2}) \nu_I( t_{I,2}^{-2}).\] Accordingly $\wt \la_I$ is $c_I$-invariant if and only if $[c_I,t_{I,2}] \in \ker(\nu_I)= \ker(\widehat \la)\cap \Z( \wh L_I )$. This proves (a).
\end{proof}
\noindent
The group $W(\la)$ is then generated by $W(\wh \la)$ and an element that is described below.
\begin{lem}\label{lem6_12}
Let $\mu\in\Irr(\wh L)$ with $\ker(\mu)=L$ and $\wh \la\in\Irr(\wh L)$. 
Additionally  for every $I\in\cO$ let $\mu_I\in\Irr(\wh L_I)$ with $\ker(\mu_I)=L_I$, $\wh \la_I\in\Irr(\restr \wh \la|{\wh L_I})$ and  $\la_I\in\Irr(\restr \wh \la_I|{L_I})$.
\begin{thmlist}
\item Let $x\in\ov W\setminus \ov W(\wh \la)$ and $\la\in\Irr(\restr \wh \la|{L})$. Then $x\in \ov W(\la)$ if and only if for every $I\in\cO$ the equality $(\wh \la_I)^x=\wh \la_{I'}\mu_{I'}$ holds, where $I'\in\cO$ with $(\wh L_I)^x=\wh L_{I'}$. 
\item We set $\cO_{ext}:=\{ I \in \cO \mid \restr \wh \la_I|{L_I}= \la_I \}$ and $\cO_{ind}:= \cO\setminus \cO_{ext}$. Then $W(\la)$ stabilizes $\cO_{ext}$ and $\cO_{ind}$. 
\end{thmlist}
\end{lem}
\begin{proof}
Since $\wh L/L$ has order 2, we see that $\mu_{\cO}$, the product of the characters $\wh \mu_I$ ($I\in\cO$) defined as in the proof of \Cref{lem411}, is an extension of $\mu$. This implies part (a). 

For part (b) we observe that for $I\in\cO$, $\si\in W(\la)$ and $I':=\si^{-1}(I)$ the characters $\restr \wh \la_I|{L_I}$ and $\restr (\wh \la_I)^\si|{L_{I'}} =\restr \wh \la_{I'}\mu_{I'}|{L_{I'}}$ have the same number of constituents. This proves part (b) since $I\in\cO_{ind}$ if and only if $\restr \wh \la_I|{L_I}$ is reducible. 
\end{proof}

\subsection{Cuspidal characters of $L_I$}
The aim here is to describe the structure of $\ov W(\wt \la)$ by analysing $\cO_{c,-1}(\pwh\lambda) $ (see \Cref{lem411}). 
We show in this section that for some $I\in\cO$ there exists no or only few $c_I$-stable cuspidal characters of $L_I$ and study the kernel of those characters, see \Cref{cor6_20}. 

For $I\in\cO$, let $\II{cusphatLI}@{\protect\cusp(\pwh L_I)}:=\Irr(\wh L_I\mid \cusp(L_I))$ and call those characters {\it cuspidal} as well. 
\index{character! cuspidal of $\wh L_I$} 
\index{cuspidal character of $\wh L_I$}
\begin{lem} \label{lem6_13}
Let $I\in\cO_d$ for some $d\in\DD_\odd\setminus\{\pm 1\}$. There exists no $c_I$-stable character in $\cusp(\wh L_I)$. 
\end{lem}
\begin{proof}
According to Lemma \ref{lem33_loc}(c), $L_I\cong\GL_d(q)$ and the element $c_I$ defined in \ref{not4_15} induces on $G_I$ a combination of an inner automorphism and the non-trivial graph automorphism according to Lemma \ref{lemVactKc}. 
The element $c_I$ acts on the torus $\bZ_I:=\h_I(\FF^\times)$ from \Cref{lem33_loc} by inverting. Hence via the isomorphism $L_I \cong \GL_d(q)$ the element $c_I$ induces on $L_I$ a combination of an inner automorphism and the non-trivial graph automorphism. 

 According to \Cref{prop51}(a) there is no cuspidal character of $\GL_d(q)$ that is invariant under transpose-inverse. So no character in $\cusp(L_I)$ is $c_I$-stable. Now the element $t_I$ from \Cref{lem36} can be chosen such that $[t_I,\bL_I]=1$, see \Cref{lem_act_onGI}. This implies that every cuspidal character of $\wh L_I$ is an extension of a cuspidal character of $L_I$. This proves that there is no $c_I$-stable character in $\cusp(\wh L_I)$.
\end{proof}
With the following statement the above shows that $\cO_{c,1}(\wh \la)\cap \cO_d=\emptyset $ for every $d\in\DD\setminus\{\pm 1\}$ and $\wh \la\in\cusp(\wh L)$ with $h_0\in \ker(\wh\la)$.
\begin{prop}\label{prop_struct}
Let $I\in\cO_d$ for some $d\in\DD\setminus\{\pm 1\}$. Then every $\psi \in \cusp(\wh L_I\mid 1_{\spannh})$ with $\psi^{c_I}=\psi$ satisfies $\bZ_I^F\leq \ker(\psi)$, where  $\bZ_I:=\h_I(\FF^\times)$ is as in  \Cref{lem33_loc}.
\end{prop} 
\begin{proof} 
 Under the isomorphism $\bL_I/\spann<h_0>\cong \GL_d(\FF)$ from \Cref{lem33_loc} we obtain $\wh L_I/\spann<h_0> \cong \GL_d(q)$. Via this isomorphism $\bZ_I^F$ is mapped to $\Z(\GL_d(q))$. Let $\psi \in \cusp(\wh L_I\mid 1_{\spannh})$.
If $\psi$ is $c_I$-invariant, then it corresponds to a cuspidal character of $\GL_d(q)$ that is invariant under transpose-inverse, see Lemma \ref{lemVactKc}. According to \Cref{prop51}(a) such a character is trivial on the centre. This implies $\bZ_I^F\leq \ker(\psi)$.
\end{proof}

\begin{theorem}\label{prop_struct2}
Let $\nu \in\Irr(\spannh)$ be non-trivial,  $d\in\DD_{\even}$, $I\in \cO_d$ and let $t_{I,2}$ be as defined before \Cref{lem411}. 
\begin{thmlist}
\item If $d\geq 4$, every $\psi\in\cusp(\wh L_I\mid \nu)$ with $\psi^{c_I}=\psi$ satisfies $t_{I,2}^2\in\ker(\psi)$.
\item If $d=2$ and $4\mid (q-1)$, there is a unique $\psi\in\cusp(\wh L_I\mid \nu)$ with $\psi^{c_I}=\psi$ and $t_{I,2}^2\notin\ker(\psi)$.
\end{thmlist}
\end{theorem} 

The proof goes through the next three lemmas. We keep $\nu$ the non-trivial irreducible character of $\spannh$.
As a first step towards a proof of the above we determine the inertia group  in $\wh L_I$ of cuspidal $c_I$-stable characters of $L_I$. 
\begin{lem}\label{lem48_ersatz}
Let $d\in\DD_{\even}$, $I\in \cO_d$, $\psi\in \cusp(L_I\mid \nu)$ with $\psi^{c_I}=\psi$ and $t_{I,2}^2\notin\ker(\psi)$. Then $(\wh L_I) _{\psi}=L_I$. 
\end{lem}
\begin{proof} 
For the proof it is sufficient to show that a character $\psi$ with the above properties and $(\wh L_I)_{\psi}=\wh L_{I}$ cannot exist. Recall $t_{I,2}^2=\h_I(\zeta ')$, where $\zeta '\in\FF^\times$ is a root of unity of order $2(q-1)_2$. 

Let $\bG':=\tD_{2d,sc}(\FF)$ with an $\FF_q$-structure given by a standard Frobenius endomorphism $F_1: \bG'\ra \bG'$. Let $\bL'$ be the Levi subgroup of $\bG'$ of type $\tA_{d-1}\times \tA_{d-1}$ such that $\cO(\bL')=\cO_d(\bL')= \{I_1,I_2\}$ be defined by $\bL'$ as in \ref{defcO}. Then $\psi$ defines cuspidal characters $\la_{I_1}\in\cusp(L_{I_1})$ and $\la_{I_2}\in\cusp(L_{I_2})$ that have extensions to $\wh L_{I_1}$ and $\wh L_{I_2}$ and are $V_S'$-stable, where $V'_S$ is associated to $\bG'$ and $\bL'$ as in \Cref{def_standardized}. We can choose $\wh \la_{I_j}\in\Irr(\wh L_{I_j}\mid \la_{I_j})$ ($j=1,2$) such that they are not $V'_S$-conjugate. The group $\wh L':=\wh L_{I_1}.\wh L_{I_2}$ is a central product of the groups $\wh L_{I_j}$ ($j\in \underline 2$) over $\spannh$. Let $\wh \la':=\wh \la_{I_1}\ .\ \wh\la_{I_2}\in\Irr(\wh L')$, $\la'=\restr\wh \la'|{L'}$ and $\wt \la'\in\Irr(\wt L'\mid \wh \la')$ where $L':=(\bL')^F$, $\wh L':={\calL'}^{-1}(\spannh)\cap \bL'$ and $\wt L':=\calL '{}^{-1}(\Z(\bG'))\cap \bL'$ for the Lang map $\calL '\colon x\mapsto x^{-1}F'(x)$ of $\bG '$.

Defining $W$, $\ov W$ from the above for  $\bG '$ and $\bL '$, note that $ W(\wh \la')=\ov W(\wh \la')=\spann<(I_1,-I_1), (I_2,-I_2)>$ and $W(\la')=\ov W(\la ') =\Sym_{\pm \cO(\bL')}$. Note also that $W(\wt \la') =\spann<(I_1,-I_1) (I_2,-I_2)>\leq \Z(W(\la'))= [W(\la'), W(\la')]$. 

Now observe that the non-trivial character of $W(\wt \la')$ is $W(\la')$-stable but does not extend to $W(\la')$ as the kernel of any linear character of $W(\la')$ contains $\Z(W(\la'))= [W(\la'), W(\la')]$. This also implies that for some character $\eta\in\Irr(W(\la'))$ the constituent $\eta_0$ of $\restr\eta|{W(\wt \la')}$ has multiplicity $2$ in $\restr\eta|{W(\wt \la')}$. The character $\R_{\wt L'}^{\wt G'}(\wt \la')_{\eta_0}$ restricts to $(\bG')^F$ and has only constituents with multiplicity $1$ according to \cite[15.11]{CE04}. 

Like in other places these results are considering first the situation of Harish-Chandra induction for a group $(\wt \bG')^{F_1}$ that comes from a regular embedding of $\bG'$ into a group with connected centre. These results can then be applied to the groups $\wt G':=\calL^{-1}(\Z(\bG'))$ and the subgroup $\wt L'$. 

On the other hand according to \cite[13.9(b)]{Cedric}, the character $\R_{L'}^{G'}(\la')_{\eta}$ has multiplicity $2$ in $\R_{\wt L'}^{\wt G'}(\wt \la')_{\eta_0}$. This is a contradiction. This implies that a character $\psi$ with the above properties cannot exist and proves the statement.
\end{proof}
In the next step we continue to consider the case where $I\in \cO_d$ with $2\mid d$. 
\begin{lem}\label{lem513} 
Let $\nu\in\Irr(\spannh)$ be non-trivial and $I\in \cO_d$ for some $d\in\DD_\even$ with $d>2$. Then every $\psi\in\cusp( L_I\mid \nu)$ with $\psi^{c_I}=\psi$ satisfies $t_{I,2}^2\in\ker(\psi)$.
\end{lem}
\begin{proof}
	Let $z:=t_{I,2}^2=\h_\ul(\zeta ')$ for some $\zeta '\in \FF^\times$ a root of unity of order $2(q-1)_2$, and $\psi\in\cusp(L_I\mid \nu)$ with $\psi^{c_I}=\psi$ and $z\notin\ker(\psi)$. According to \Cref{lem48_ersatz}, $\psi^{\wh L_I}$ is irreducible. Note $z\in \Z(\bL_I)$. 
	Since $d\geq 4$ it is sufficient to show the statement in the case where $I=\ul$ and hence $L_I=L$. 
	
	The group $L_0:=[\bL,\bL]^F$ satisfies $L_0\cong \SL_l(q)$, see \Cref{lem33_loc}. Let $\psi_0\in \Irr( \restr\psi|{L_0})$. According to \Cref{lem_cusp}, $\psi_0$ is cuspidal. Following \Cref{lem_act_onGI} the automorphisms of $L_0$ induced by $\wh L$ are diagonal automorphisms of $L_0$. Since $\Cent_{\wh L}(L_0)\leq \Z(\bL)$ and $\wh L/(\Cent_{\wh L}(L_0)L_0)$ is cyclic, we can see that maximal extendibility holds with respect to $L_0\lhd L$.
	As $\psi^{\wh L}$ is irreducible, $\wh L_\psi=L$. As $\wh L/L_0$ is abelian this implies $\wh L_{\psi_0}\leq L$. 
	
	We now use the fact $L_0\cong \SL_l(q)$. Let $\bH:=\GL_l(\FF)$ and let $F':\bH\lra \bH$ be a Frobenius endomorphisms giving an $\FF_q$-structure such that $\bH^{F'}\cong \GL_l(q)$. Via $[\bH,\bH]\cong [\bL,\bL]$ we identify $[\bH,\bH]^F$ with $L_0$. 
	Hence $\psi_0$ defines $\psi_0'\in\cusp([\bH,\bH]^F)$. By the above this implies 
	\begin{align*} \label{2midGLpsi} 2\mid |\bH^F:\bH^F_{\psi'_0}|.
	\end{align*}
	The character $\psi^{\wh L}$ is $c_I$-stable. Hence $\restr \psi|{L_0}$ is $c_I$-stable. Following \Cref{not4_15}, $c_I$ acts on $L_0$ by a graph automorphism and $\wh L$ acts on $L_0$ as diagonal automorphisms. 
	
	As $\restr \psi|{L_0}$ is $c_I$-stable, we can choose $\psi'$ to be stable under the graph automorphism of $\SL_l(q)$ and it is cuspidal according to \Cref{lem_cusp}. In this situation $\psi_0$ only exists if $l=2$, see \Cref{prop_SL2_cusp}. By the assumption $d>2$, so we get a contradiction. This implies our claim that any $c_I$-stable character $\psi$ satisfies  $t_{I,2}^2\in\ker(\psi)$. 
\end{proof}

\begin{lem} \label{lem6_18}
Let $I\in \cO_2$ and $\nu$ as in \Cref{lem513}. There are exactly two characters $\psi\in\cusp( L_I\mid \nu)$ with $\psi^{c_I}=\psi$ and  $t_{I,2}^2\notin\ker(\psi)$. Those characters are $\wh L$-conjugate.
\end{lem}
\begin{proof}
From the proof of \Cref{lem513} and \Cref{prop_SL2_cusp} we see that there are two $\GL_2(q)$-conjugate characters $\psi_0\in \Irr(\restr\psi|{[\bL_I,\bL_I]^F})$, that are the only possible constituents of $\psi$. 
If $I=\{i,i'\}$, then $\psi_0(\h_{e_i-e_{i'}}(-1))=(-1)^{\frac{q+1}{2}} \psi_0(1)$ according to \cite[Tab.~5.4]{BonSL2}. Then $\bL^F \cong \SL_2(q)\times \bZ_I^F$ by \Cref{lem33_loc}, in particular $h_0=\h_{I}(\varpi) \h_{e_i-e_{i'}}(-1)$. 
Because of $\psi(h_0)=- \psi(1)$ this implies $\psi (\h_{I}(\varpi)) = - (-1)^{\frac{q+1}{2}} \psi(1) =(-1)^{\frac{q-1}{2}} \psi(1) $.

Let $\kappa\in\Irr(\bZ_I^F)$ such that $\psi=\psi_0\times \kappa$. As $c_I$ acts by inverting on $\bZ_I$, $\kappa$ has multiplicative order 1 or 2.  The assumption $t_{I,2}^2\notin\ker(\psi)$ implies that $\kappa$ has order $2$. 
This proves that given $\psi_0$, the character $\kappa$ is uniquely determined by the fact that $\psi$ is $c_I$-stable and $t_{I,2}^2\not\in \ker(\psi)$.  Hence the only characters with the given properties are $\wh L $-conjugate. 
\end{proof}
Thanks to the above three statements we can now show \Cref{prop_struct2}.
\begin{proof}[Proof of \Cref{prop_struct2}]
Let $\nu\in\Irr(\spannh)$ be non-trivial, $d\in\DD_\even$, $I\in\cO_d$ and $\psi\in\cusp(\wh L_I\mid \nu)$ with $\psi^{c_I}=\psi$. If $d> 2$, then $t^2_{I,2}\in\ker(\psi)$ according to \Cref{lem513}. This shows part (a). 

Assume $d=2$ and $t^2_{I,2}\notin\ker(\psi)$. 
The set $\Irr(\restr \psi|{L_I})$ contains two characters according to \Cref{lem48_ersatz}.
Following Lemma \ref{lemVactKc} together with \Cref{prop43}(b) the character $\psi'\in \Irr(\restr \psi|{L_I})$ is cuspidal, satisfies $(\psi')^{c_I}=\psi'$ and $t^2_{I,2}\notin\ker(\psi')$.
Then there are exactly two $\wh L_I$-conjugate characters in  $\cusp(L_I\mid \nu)$ with  those properties, see \Cref{lem6_18}. 
Since $|\wh L_I:L_I|=2$ this implies that there is only one character $\psi$ with the given properties. This proves (b).
\end{proof}
\begin{lem}\label{prop_struct3}
If $\nu \in\Irr(\spannh)$ is non-trivial, then every $\psi\in\cusp(\wh L_{J_{-1}}\mid \nu)$ satisfies $\psi^{c_{J_{-1}}} \neq \psi$.
\end{lem} 
 \begin{proof} 
Note $\h_{J_{-1}}(\varpi)\in\Z(\bL_{-1})$ and $[c_{J_{-1}},\h_{J_{-1}}(\varpi)]=h_0$. No extension $\wh\nu\in\Irr(\Z(\bL_{-1})^F\mid \nu)$ is $c_{J_{-1}}$-stable. This implies that $\psi_{-1}$ is not $c_{J_{-1}}$-stable. 
 \end{proof}
The above leads to the following statement on the sets $\cO_c(\wh \la)$, $\cO_{c,-1}(\wh\la)$ and $Y(\wh \la)$ introduced earlier in \ref{lem510} and \ref{lem411}. We use the notation $\II ozmu@{o(\mu)}$ to denote the multiplicative order of a linear character $\mu$ of a finite group.
\begin{cor}\label{cor6_20}
Let $\la\in\cusp(L)$, $\wh \la$ and $\wh \la_I$ associated to $\la$ as in \Cref{lem411}. If $\wh \la$ is standardized, then: 
\begin{thmlist}
\item $\cO_c(\wh\la)\subseteq \bigcup_{d\in \DD_\even\cup \{1,-1\}} \cO_d$. 
\item If $h_0\in\ker(\la)$, then $\cO_{c,-1}(\wh\la)\subseteq \{J_{-1}\}\cup \{I\in\cO_1\mid o(\wh \la_I)\mid 2 \}$.
\item If $h_0\notin\ker(\la)$, then $\cO_{c,-1}(\wh\la)\subseteq \cO_2$ and all $\{\wh \la_I\mid I \in \cO_{c,-1}(\wh\la)\}$ are $V_S$-conjugate, i.e. $\cO_{c,-1}(\wh\la)\in Y(\wh \la)$. 
\end{thmlist}
\end{cor}
\begin{proof}
\Cref{lem6_13} implies that $\cO_{c,-1}(\wh \la)\cap \cO_d=\emptyset$ for every $d\in \DD_\odd \setminus\{\pm 1\}$. This gives (a).

For the proof of (b) assume $h_0\in\ker(\la)$. Then \Cref{prop_struct} implies $\cO_{c,-1}(\wh\la)\subseteq \cO_{-1}\cup \cO_1$. For $I\in\cO_1$, the character $\wh \la_I$ is $c_I$-stable if and only if $o(\wh \la_I)\mid 2$. 

For the proof of (c) assume $h_0\notin\ker(\la)$. Then $\cO_{c,-1}(\wh\la)\cap \cO_1=\emptyset$ and $\cO_{c,-1}(\wh\la)\subseteq \cO_2$ according to \Cref{prop_struct2}. 
\Cref{lem6_18} proves that $\{\wh \la_I\mid I \in \cO_{c,-1}(\wh \la)\}$ are $V_S$-conjugate. Hence the partition $Y(\wh \la)$ from \Cref{lem510} contains $\cO_{c,-1}(\wh \la)$. 
\end{proof}

Recall $K(\la):=\ov W_{\restr\la^{L\spann<F_p>}|L}$. For any $\ov W$-stable $L\leq J\leq \wt L$, $\kappa \in \Irr(J)$ and $Q\subseteq \cO$ let 
$\II WoverlineQ @{\ov W^Q(\kappa)}:=\ov W(\kappa)\cap \Sym_{\pm Q }$ and 
$\II WQ@{W^Q(\kappa)}:=\ov W^Q(\kappa)\cap W$.

\begin{prop}\label{cor_6B}
Let $ \la \in \TT$, $\wh \la \in \cusp(\wh L\mid \la)$, $\wh \la_I\in\Irr(\restr \wh \la| {\wh L_I})$ ($I\in \cO$) and $\wt \la\in\Irr(\wt L\mid \wh \la)$. Assume that $\wh \la$ is standardized in the sense of \ref{def_standardized}. We set 
 \begin{align*}
\II Q1hatlambda@{ Q^1(\pwh \lambda)}:=
 \begin{cases} 
 \{ I \in \cO_1\mid o(\wh \la_I)\mid 2\} \cup \cO_{-1} &\text{, if } h_0\in\ker(\la),\\
 \cO_{c,-1}(\wh \la) &\otw.
\end{cases}
\end{align*}
Then: 
\begin{thmlist}
\item $Q^1(\wh \la)$ is $K(\la)$-stable, and
\item $W(\wt \la)=W^1(\wt \lambda)\times W^2(\wt \lambda)$, where 
$\II Q2la@{Q^2(\wh \lambda)}:=\cO\setminus Q^1(\wh \la)$, and 
$\II Wjlambdatilde@{W^j(\pwt \lambda)}:= W^{Q^j(\wh \la)}(\wt \lambda)$ for $j\in\underline 2$.
\end{thmlist}
\end{prop}

\begin{proof}
Let $e\in \spann<F_0>$ such that $\la$ and $ \la^e$ are $\ov N$-conjugate. As $\wh \la$ is standardized, then $\wh \la^e$ is also standardized. 
As the orders of $\wh \la_I$ and $(\wh \la_I)^e$ coincide for every $I\in \cO$, we see that $Q^1(\wh \la)=Q^1(\wh\la^e)$ from the definition, whenever $h_0\in\ker(\la)=\ker(\wh\la^e)$ and hence $h_0\in\ker (\wh\la^e)$. 

Assume $h_0\notin\ker(\la)$. Then $h_0\notin\ker(\la^e)$. Let $I\in \cO\setminus\{J_{-1}\}$. 
Because of $c_I^e\in c_I\spannh$ we see 
$$ \wh \la_I^{c_I}=\wh \la_I \Leftrightarrow 
\Irr(\restr\wh \la^e|{\wh L_I})^{c_I} =
\Irr(\restr\wh \la^e|{\wh L_I}).$$
In the case of $\wh \la_I^{c_I}=\wh \la_I$ the character $\wh \la_I$ has some $c_I$-stable extension to $\wt L_I$ if and only if the unique character in $\Irr(\restr\wh \la^e|{\wh L_I})$ has some $c_I$-stable extension to $\wt L_I$. 
(The set $\Irr(\restr \phi|{\wh L_I})$ is a singleton for every $\phi\in\cusp(\wh L)$, since $\wh L$ is the central product of the groups $\wh L_I$ over $\spannh$.)
This shows $\cO_{c,-1}(\wh \la)=\cO_{c,-1}(\wh \la^e)$ and $Q^1(\wh \la)=Q^1(\wh \la^e)$ by the definition of those sets. 
Let $w\in \ov W$ and $e\in \spannFp$ with $we\in \wh K(\la)$. Then $\wh \la^w$ is standardized and $Q^1(\wh \la)^w= Q^1(\wh \la^w)$. Accordingly $w\in K(\la)$ stabilizes $Q^1(\wh \la)$. This implies part (a).

For part (b) recall the description of $\ov W(\wt\la)$ from 
\Cref{lem41}: 
\[\ov W(\wt\la)= 
\left(\Spann<(I,-I)|I\in \cO_{c,1}(\wh\la)> \Spann<(I,-I)(I',-I') | I,I'\in \cO_{c,-1}(\wh\la)>\right)
\rtimes \Young_{Y(\wh \la)\cup Y'(\wh \la)}.\]
First assume $h_0\in\ker(\la)$. By construction $\cO_c(\wh\la)\subseteq Q^1(\wh\la) \cup \bigcup_{d\in\DD_\even} \cO_d$ and hence $\ov W(\wt \la)=\ov W^1(\wt \la)\times \ov W^2(\wt \la)$. According to \Cref{cor6_20}(a) we observe $(I,-I)\in W$ for every $I\in \cO_c(\wh\la)\setminus Q^1(\wh\la)$. This implies $\ov W^2(\wt \la)\leq W$ and $W(\wt \la)=W^1(\wt \lambda)\times W^2(\wt \lambda)$ by the definition of $W$.

It remains to consider the case where $h_0\notin\ker(\la)$. Then $Q^1(\wh\la)=\cO_{c,-1}(\wh \la)\subseteq \cO_2$ by \Cref{cor6_20}(c) and hence $\ov W^1(\wt \la)\leq W$. By the structure of $\ov W(\wt \la)$ described in \Cref{lem41} we see  $W(\wt \la)=W^1(\wt \lambda)\times W^2(\wt \lambda)$.
\end{proof}

\subsection{Clifford theory for $W(\wt \la)\lhd W(\la)$ in the case of $\wh L\wt L_\la= \wt L$}
In this section we study the characters of $W(\wt \la)$, in particular their Clifford theory with respect to $K(\la)$. Assuming $\wh L\wt L_\la=\wt L$ we prove maximal extendibility \wrt $ W (\wt \la)\lhd K(\la)$. This result is required for a later application of \Cref{prop23}. We consider the following situation. 
\begin{notation}
Let $ \la \in \cusp(L)$, $\wh \la\in\Irr(\wh L\mid \la)$ and $\wt \la\in\Irr(\wt L\mid \wh \la)$ such that $\wh \la$ is standardized and $\wh L \wt L_\la=\wt L$ (or equivalently $\restr \wt \la|{\wh L}=\wh \la$). 
% and $\la$ extends to some $\wt \la\in\Irr(\wt L\mid \wh \la)$.
\end{notation}
For further computations we use the groups $\II Kjlambda@{K^j(\lambda)}$
associated to the subsets $Q^j(\la)\subseteq \cO$ from \Cref{cor_6B}, where $K^j(\la):= (K(\la) \Sym_{ \pm( \cO\setminus Q^j(\la))}) \cap  \Sym_{\pm Q^j(\la)}$ for $j\in \underline 2$. 
\begin{lem}\label{cor_4B}
If maximal extendibility holds \wrt $W^j(\wt \la)\lhd K^j(\la)$ for every $j\in \underline 2$, then maximal extendibility holds \wrt 
$W(\wt \la)\lhd K(\la)$, in particular for every $\eta_0 \in \Irr(W(\wt \la))$ there exists some $K(\la)_{\eta_0}$-stable $\eta \in\Irr(W (\la)\mid \eta_0)$.
\end{lem}
In that situation the above statement will ensure Assumption \ref{prop23_neu}(iii) for $\ov {(\la,\eta)}\in \calP(L)$ via \Cref{KlawhKla}.
\begin{proof} 
Recall $K( \la):= \ov W_{\restr \la^{L \spannFp}|L}$ by the definition in \ref{not52}. As $K(\la)$ stabilizes $Q^1(\wh \la)$ and $Q^2(\wh \la)$ by \Cref{cor_6B}, $K(\la)\leq \Sym_{\pm Q^1(\wh \la)}\times \Sym_{\pm Q^2(\wh \la)}$. This rewrites as $K(\la)\leq K^1(\la)\times K^2(\la)$. 
Recall that $\cO_{c,-1}(\wh \la) \subseteq Q^1(\wh \la)$.

Since maximal extendibility holds \wrt $W^j(\wt \la)\lhd K^j(\la)$ for $j \in \underline 2$ by assumption, maximal extendibility holds \wrt $$W(\wt \la)=W^1(\wt \la)\times W^2(\wt \la)\lhd K^1(\la)\times K^2( \la).$$ This implies the statement as $K(\la)\leq K^1(\la)\times K^2(\la)$. 
\end{proof}
\noindent 
For $\la\in \TT$ with $\wt L_\la=\wt L$  we study first the Clifford theory of $W^2(\wt\la)\lhd K^2(\la)$ for the groups from \Cref{cor_4B}.
\begin{lem}\label{CliffW2K2}
Let $W^1(\wt\la)$, $W^2(\wt\la)$, $K^1(\la)$ and $K^2(\la)$ be the groups from \Cref{cor_4B}. Then: 
\begin{thmlist}
\item maximal extendibility holds \wrt $W^{2}(\wt \la)\lhd K^2(\la)$, and 
\item maximal extendibility holds \wrt $W^{1}(\wt \la)\lhd K^1(\la)$, if $h_0\notin \ker(\la)$. \end{thmlist}
\end{lem} 
\begin{proof} Let $Y(\wh\lambda)\vdash \cO_c(\wh \la)$ and $Y'(\wh \la)\vdash \cO\setminus \cO_c(\wh \la)$ be the partitions from \Cref{lem510}. In order to prove part (a) we can assume $Q^2(\wh\la)=\cO$ without loss of generality. We have $\ov W(\wt \la) =\ov W^2(\la)=\Young_{\pm Y(\wh \la)}\times \Young_{Y'(\wh \la)}$, see \Cref{lem411}. 

If $h_0\in\ker(\la)$ then $\cO_c(\wh \la)\cap (\cO_1\cup \cO_{-1})=\emptyset$ by the choice of $Q^1(\wh\la)$ according to \Cref{cor6_20}. 
If $h_0\notin\ker(\la)$, \Cref{prop_struct3} implies $\cO_c(\wh \la)\cap \cO_{-1}=\emptyset$ and analogously we see $\cO_c(\wh \la)\cap \cO_{1}=\emptyset$.

This implies $\cO_c(\wh \la)\subseteq \bigcup_{d\in\DD_\even}\cO_d$. 
Accordingly $W(\wt\la)$ is the direct product of groups $W_d(\wt \la)$ for $d\in\DD$. It suffices to consider the case where $\cO=\cO_d=Q^2(\wh \la)$ for some $d\in \DD$ and $\cO_c(\wh \la)\in \{\cO,\emptyset\}$.
Additionally we can assume that $Y(\wh \la)$ and $Y'(\wh \la)$ are partitions whose elements have all the same cardinality. 
If $\cO_c(\wh \la)=\cO$, then $W^2(\la)\cong (\Cy_2\wr \Sym_k)^a$ for some positive integers $k$ and $a$. Then $K^2(\lambda)\cong (\Cy_2\wr \Sym_k)\wr \Sym_a$, and hence maximal extendibility holds \wrt $W^2(\wt \la)\lhd K^2(\la)$.

If $\cO_{c}(\wh\la)=\emptyset$, then $W^2(\la)=\Young_{Y'}$ and hence it is isomorphic to a direct product of symmetric groups. The group $K^2(\la)\leq \NNN_{\Sym_{\pm \cO_d}}(\Young_Y)$ is isomorphic to $(C \Young_Y)\rtimes \Sym_{Y}$, where $C:= \spann<\prod_{k\in y}( k,-k) \mid y\in Y>\leq \Sym_{\pm \cO_d}$. By \Cref{lem_wr} maximal extendibility holds with respect to $W_2(\wt \la)\lhd K_2(\la)$. This proves part (a).

For part (b) we assume $\cO_{c,-1}(\la)=\cO$, $h_0\notin \ker(\la)$, and as before $Q^1(\wh \la)=\cO$. By \Cref{cor6_20}(c) we have $K(\la)= W$ and $|W:W(\wt \la)|=2$, see \Cref{prop_struct}(c). As $W(\la)/W(\wt \la)$ is cyclic, maximal extendibility holds \wrt $W(\wt \la)\lhd K(\la)$.
\end{proof}
It remains to prove the following. 
\begin{prop}\label{CliffW1K1}
Maximal extendibility holds \wrt $W^1(\wt \la)\lhd K^1(\la)$, if $h_0\in \ker(\la)$. 
\end{prop}
\begin{proof}
Let $O_{1,i}=\{I\in \cO_1\mid o(\wh \la_I)=i\}$ for $i\in \underline 4$ and $l_i:= |O_{1,i}|$. 
By \Cref{lem510} and \Cref{lem411} we have $\ov W_{Q^1(\wh \la)}(\wt \la)\leq C\times \Sym_{\pm O_{1,1}}\times \Sym_{\pm O_{1,2}}$, where $C\leq W_{\cO_{-1}}(\wh\la)$ and $C$ is then either trivial or a cyclic group of order $2$. The group structures depend on $J_{-1}$ and those groups are described in Table \ref{Tab1}, where $W(\tB_j)$ and $W(\tD_j)$ are Coxeter groups of type $\tB_j$ and $\tD_j$, respectively. 

\begin{table}[h]
\begin{center}
\begin{tabular}{|c|c|c|c|}
	\hline
	% &&\\
 & $W^1(\wt \la)$& $K^1(\la)$ \\ 
	\hline 
	$J_{-1}\notin \cO_c(\wh \la)$ & $W(\tD_{l_1})\times W(\tD_{l_2})$ & 
	\small{ $\Cy_2\times (W(\tB_{l_1})\wr \Cy_2)$ \text{, if }$l_1=l_2$}\\
	&&$\Cy_2\times W(\tB_{l_1})\times 
	W(\tB_{l_2})$ \text{, if }$l_1\neq l_2$
	 \\ 
	\hline 
	$J_{-1}\in \cO_{c,1}$ & $ W(\tB_{l_1})\times W(\tD_{l_2})$ & $\Cy_2\times W(\tB_{l_1})\times W(\tB_{l_2})$\\
	\hline 
	$J_{-1}\in \cO_{c,-1}$ & $ W(\tD_{l_1})\times W(\tB_{l_2})$ & $\Cy_2\times W(\tB_{l_1})\times W(\tB_{l_2})$\\
		\hline 
\end{tabular} 
\end{center}
\caption{Isomorphism types of $W^1(\wt \la)$ and $K^1(\la)$}\label{Tab1}
\end{table}

Note that in all cases $\ov W_{O_{1,1}}(\wt \la)$ is a Coxeter group of type $\tB_{l_1}$. 
Considering the structure we observe that in all cases the statement holds according to \Cref{lem_wr}. 
\end{proof}
Recall that $\MM^{(X)}=\{\la\in\cusp(L)\mid \wt L_\la = X\}$ and $\MM_0:=\cusp(L)\setminus (\MM^{(L)}\cup \MM^{(\wt L)})$ for $X$ with $L\leq X\leq \wt L$. For characters in $\MM^{(\wt L)}\cup \MM_0$ the above proves the following: 
\begin{prop}\label{prop_sec4B}
Let $\la\in \MM^{(\wt L)}\cup \MM_0$, i.e. $\wh L \wt L_\la=\wt L$. 
For every $\wt \la\in\Irr(\wt L\mid \la)$ and $\eta_0\in\Irr(W(\wt \la))$ there exists some $K( \la)_{\eta_0}$-stable $\eta\in\Irr(W(\la)\mid \eta_0)$. 
\end{prop}
\begin{proof}
According to \Cref{cor_4B}, \Cref{CliffW1K1} and \Cref{CliffW2K2} imply the statement. 
\end{proof}

%%%%%%%%%%%%%%%%%%%%%%%%%%%%%%%%%%%%%%%
\subsection{Clifford theory for $W(\wh \la)\lhd W(\la)$ in the case of $\wt L_\la= \wh L$}\label{ssec5E}
We now study $W(\la)$ and $W(\wh \la)$ for characters $\la\in \MM^{(\wh L)}$, 
where $\wh\la\in\Irr(\wh L\mid\la)$ is standardized. 
We prove statements on the characters of $W(\wh \la)$ and their possible extensions to $W(\la)$. The results later imply that there exists some $\EL$-stable $\wt N$-transversal in $\Irr(N\mid \MM^{(\wh L)})$.

In the following we study the Clifford theory of $W(\wh \la)\lhd K(\la)$ for $\la\in \MM^{(\wh L)}$, where $K(\la)=\ov W_{\restr\la^{L\spannFp}|L}$. 

\begin{lem}\label{rem519}
Assume $|\Z(\GF)|=2$ or equivalently $q\equiv 3 (4)$ and $2\nmid l$. Every $\la\in\cusp(L)$ satisfies $\wh L\leq \wt L_\la$. Then $\cusp(N\mid \TT\cap \MM^{(\wh L)}) $ is an $\EL$-stable $\wt L$-transversal of $\cusp(N\mid \MM^{(\wh L)})$. 
\end{lem}
\begin{proof} The arguments given in \Cref{lem6_3} show the statement. \end{proof}
According to \Cref{rem519} we can now assume  $\Z(\GF)=\Z(\bG)$. We do that until the end of the section.
\begin{lem}\label{lem-1inDD}
If $|\Z(\GF)|=4$ and $\la\in\cusp(L)$ with $\wt L_\la=\wh L$, then $-1\in\DD$ and 
$\wh\la_{-1}^{t_{\ul,2}}\neq \wh\la_{-1}$. Moreover $\wh\la_{-1}^{\n_{e_1}(\varpi)}=\wh\la_{-1}$, if $h_0\in\ker(\la)$ or $\type(\Phi_{-1})$ is not of type $\tD_{|J_{-1}|}$. 
\end{lem}
\begin{proof}
Recall that maximal extendibility holds \wrt $L\lhd \wt L$, see \Cref{lem_ext_G_whG}. Accordingly $\wt L_\la=\wh L$ implies that $\la$ is not $t_{\ul,2}$-stable for the element $t_{\ul,2}\in \bT$ from \Cref{lem36}. If $\zeta\in\FF^\times$ with $\zeta^{(q-1)_2}=\varpi$ and $t_{I,2}:=\h_I(\zeta)$ as in \Cref{lem411}, then 
$t_{\ul,2}=\prod_{I\in\cO} t_{I,2}$. 
Recall that $\wt L=\wh L\spann<t_{\ul,2}>$. The character $\wh \la$ is $t_{\ul,2}$-stable, if $\wh \la_I$ is $t_{I,2}$-stable for every $I\in \cO$. For $I\in\cO\setminus \{ J_{-1}\}$ we see $t_{I,2}\in \Cent_{\bL_I}(\wh L_I)$ and hence $\wh \la_I$ is $t_{I,2}$-stable. As $\wh \la$ is not $t_{\ul,2}$-stable, $-1\in \DD$ and $\wh \la_{-1}$ is not $t_{\ul,2}$-stable. 

In the next step we prove  $\wh\la_{-1}^{\n_{e_1}(\varpi)}=\wh\la_{-1}$.
Since $\la_{-1}^{t_{\ul,2}}\neq \la_{-1}$, \Cref{prop5E} implies that $\wh \la_{-1}$ is $\gamma $-stable, if  $\type(\Phi_{-1})=\tD_{|J_{-1}|}$. We consider the other possible values of $\type(\Phi_{{-1}})$. We first assume $\type(\Phi_{{-1}})=\tA_1\times \tA_1$. Then $L_{-1}=\SL_2(q)\times \SL_2(q)$.
Let $\la_{-1,1},\la_{-1,2}\in \Irr(\SL_2(q))$ such that $\restr \la|{L_{-1}}=\la_{-1,1}\times \la_{-1,2}$. By the proof of \Cref{prop3D}, $\wt L_\la\leq \wh L$ implies that both characters $\la_{-1,1}$ and $\la_{-1,2}$ are not $\GL_2(q)$-stable. Additionally they are cuspidal. Following \cite[Tab.~5.4]{BonSL2}, the characters $\la_{-1,1}$ and $\la_{-1,2}$ are uniquely determined up to $\GL_2(q)$-conjugation. After applying some $\wt L$-conjugation we obtain that $\la_{-1,1}$ and $\la_{-1,2}$ are $\n_{e_1}(\varpi)$-conjugate. As $\wh L$ induces on the $\SL_2(q)$-factors of $L_{-1}$ simultaneous (non-inner) diagonal automorphisms, the set $\Irr(\wh L_{-1}|{\la_{-1}})$ contains only one character hence $\wh \la_{-1}$ is again $\n_{e_1}(\varpi)$-stable. 

It remains to consider the case where $\type(\Phi_{{-1}})=\tA_3$. Again the character $\la_{-1}$ is not $\wt L$-stable. Via the isomorphism $L_{-1}\cong \SL_4(q)$ we see that $t_{\ul,2}$ induces on $\SL_4(q)$ a diagonal automorphism corresponding to a generator of $\Z(\SL_4(q))$ in the sense of \ref{labeldiag}, see also the proof of \Cref{prop3D}. We take any $\chi\in\Irr(\GL_4(q)\mid \la_{-1})$. Then $\chi$ is cuspidal, see Lemma~\ref{lem_cuspc}. Using the description of cuspidal characters of general linear groups recalled in the proof of \Cref{prop51}, we let $s\in \GL_4(q)$ and $\zeta\in\FF^\times$ such that $\chi$ belongs to the rational Lusztig series of $s$ and $\zeta\in \FF_{q^4}\setminus \FF_{q^2}$ is an eigenvalue of $s$. Let $\det:\GL_4(\FF)\ra \FF^\times$ denote the determinant and  $\det^*$ the associated linear character of $\GL_4(\FF)$ with kernel $\SL_4(\FF)$. By the assumptions on $\chi$ we see that $\chi=\chi(\det^*)^{\frac {q-1}2}$ and hence $s$ and $-s$ are conjugate. Then $-\zeta \in \{\zeta, \zeta^q, \zeta^{q^2},\zeta^{q^3} \}$. Hence, using again $o$ to denote multiplicative order, $o(\zeta)_2=2(q^2-1)_2$, as 
$-\zeta \in \{\zeta, \zeta^q, \zeta^{q^3} \}$ would imply that $\zeta\in \FF_{q^2}$ or $\zeta\in \FF_{q^6}$, contradicting $\zeta\in\FF_{q^4}\setminus \FF_{q^2}$. In order to compute $\ker(\restr\chi|{\Z(\SL_4(q))})$ we see that  $\det s= \zeta^{\frac {q^4-1}{q-1}}$ is not a square in $\FF_q^\times$ since $o(\zeta)_2=2(q^2-1)_2$. This contradicts $h_0\in\ker(\la)$, as $h_0$ corresponds to the central involution of $\SL_4(q)$. Hence there exists no cuspidal character $\la_{-1}$ of $L_{-1}$ that satisfies $h_0\in\ker(\la)$ and $\la_{-1}^{t_{\ul,2}}\neq \la_{-1}$. This shows that  $\type(\Phi_{{-1}})=\tA_3$ is not possible. This finishes our proof.
\end{proof}
\begin{lem} \label{lem6Dh0}
Let $\la\in\cusp(L)$ and $\wh \la \in \Irr(\wh L\mid \la)$ such that $h_0\notin\ker(\la)$ and $W(\wh \la)\neq W(\la)$. Then maximal extendibility holds \wrt $W(\wh \la)\lhd K(\la)$.
\end{lem}
\begin{proof}
We first determine $W(\wh \la)$. Denote $c_{-1}:=c_{J_{-1}}$. As the character $\Irr(\restr\la|{\Z(L_{-1})})$ is not $c_{-1}$-stable, $\la_{-1}$ is not $c_{-1}$-stable. This implies $J_{-1}\notin \cO_c(\wh \la)$. 

If $I\in \cO_1$, $\wh \la_I$ is a linear character. Since 
$h_0\notin \ker(\la)$ and hence $h_0\notin \ker(\wh \la_I)$, 
the order of $\wh \la_I$ is divisible by $2(q-1)_2\geq 4$. Hence $\cO_{c}(\wh \la)\cap \cO_1= \emptyset$. 

Together with \Cref{lem6_12} and \Cref{cor6_20}(a) this leads to $\cO_c(\wh \la)\subseteq \bigcup_{d\in \DD_{\even}}\cO_d$. The structure of $W(\wh \la)$ is given by \Cref{lem411} and we observe $\ov W(\wh \la)=W(\wh \la)$. 
As in the proof of \Cref{CliffW2K2} we can apply \Cref{lem_wr}, and we see that maximal extendibility holds with respect to $W(\wh \la)\lhd \NNN_{\ov W}(W(\wh \la))$. Because of $K(\la)\leq \NNN_{\ov W}(W(\wh \la))$, this proves maximal extendibility \wrt $W(\wh \la)\lhd K(\la)$. 
\end{proof} 
As in \Cref{cor_4B} we associate to $\la$ subsets $Q^1(\wh \la)$ and $Q^2(\wh \la)$ of $\cO$. Recall $ K(\la)=\ov W_{\restr\la^{L\spannFp}|{L}}$, whenever $\bG$ is not of type $\tD_4$.

\begin{lem}\label{lem6_29}
Let $\la\in\MM^{(\wh L)}\cap \Irr(L\mid 1_{\spannh})$ and $\wh \la\in\Irr(\wh L\mid \la)$ with $W(\la)\neq W(\wh \la)$. Let 
\begin{align*}
\II Q1lambda@{Q^1(\pwh\lambda)}:=
\{ I \in \cO_1\mid o(\wh \la_I)\in \{1,2,4\}\} \cup \cO_{-1} \und
\II Q2lambda@ {Q^2(\pwh \lambda)}:=\cO\setminus Q^1(\wh\la).
\end{align*}
Let $\II Woverlinejlambdahat @{\pov W^j (\pwh \lambda)}:=\ov W(\wh \lambda)\cap \Sym_{\pm Q^j(\wh \lambda)} $, 
$\II{Wjlambdahat}@{W^j(\pwh\lambda)}:=\ov W^j(\wh \lambda)\cap W$, 
$\II Lhati@{\phat L^{(i)}}:=\spann<\wh L_I\mid I \in Q^i(\wh\la)>$
 and 
$\II Li@{L^{(i)}}:=L\cap \wh L^{(i)}$, for $i,j\in\underline{2}$.
\begin{thmlist}
\item Then $K(\la)$ stabilizes $Q^1(\wh\la)$ and  
$$W(\wh \la)=W^1(\wh \la^{(1)})\times W^2(\wh \la^{(2)}),$$
where 
$\wh \la^{(i)}\in \Irr(\restr \wh \la|{\wh L^{(i)}})$.
\item If $x\in W(\la) \setminus W(\wh \la)$, then $x=x_1x_2$ for some $x_i\in W^i(\la^{(i)})$ ($i\in\underline 2$), where $\la^{(i)}\in \Irr(\restr \la|{L^{(i)}})$. 
\item $|\cO_d\cap Q^2(\wh\la)|$ is even for every $d\in \DD_\odd\setminus\{-1\}$.
\item Let  
$\ov W^i:=\ov W \cap \Sym_{\pm Q^i(\wh \la)}$ and $\II Kilambda@{K^i( \lambda)}:= (\ov W^i)_{\restr(\la^{(i)})^{ L^{(i)}\spannFp}|{L^{(i)}}}$ for $i\in\underline 2$. If $\cO_1\cap Q^1(\wh \la)\neq \emptyset$, then $K(\la) \leq  K^1(\la)\times \NNN_{\ov W^2}( W^2(\wh \la))$.
\end{thmlist}
\end{lem}

\begin{proof} 
First note that $\wh N$ normalizes the groups $\bL_I$ ($I\in \cO$) and hence there is a well-defined action of $\wh W$ on $\cO$. Now $Q^1(\wh\la)$ is defined using $\wh \la$ (and is independent of the choice of $\wh \la\in\Irr(\wh L\mid \la)$). Note that by this definition any element in $\wh N_\la$ stabilizes $Q^1(\wh\la)$. Without loss of generality we can assume that $\wh \la$ is standardized and hence $\ov W(\la)$ is given in \Cref{lem411}. 

Accordingly $\ov W(\wh \la)= \ov W^1(\wh \la)\times \ov W^2(\wh \la)$. As $\la\in\Irr(L\mid 1_{\spannh})$,  \Cref{cor6_20} implies $\cO_c(\wh \la)\subseteq \bigcup_{d\in \DD_\even} \cO_d$. By the definition of $Q^2(\wh \la)$ we observe $\ov W^2(\wh \la)=W^2(\wh \la)$. 

According to \Cref{lem6_7}(d), $\ov W(\wh \la)$ is the direct product of the groups $\ov W_d(\wh \la)$, where $\ov W_d(\wh \la):=(\ov V_d)_{\wh \la}/H_d$. For $\ov W^1(\wh \la)$ we note that 
$$\ov W^1(\wh \la)= \ov W^{1,1}(\wh \la)\times \ov W^{1,2}(\wh \la)\times 
\ov W^{1,4}(\wh \la),$$ 
where $ \II Q1jlambd@{Q^{1,j}(\lambda)}:=\{I\in \cO_1\mid o(\wh \la_I)=j\}$
and $\ov W^{1,j}(\wh \la):=\ov W^1(\wh \la)\cap \Sym_{\pm  Q^{1,4}(\la)}$. 
This proves that 
$\ov W(\wh \la)=\ov W^1(\wh \la)\times \ov W^2(\wh \la)$. By the above $W^2(\wh \la)=\ov W^2(\wh \la)$ and hence $W(\wh \la)=\ov W^1(\wh \la)\times W^2(\wh \la)= W^1(\wh \la)\times W^2(\wh \la)$. 
Since $\wh \la=\wh \la^{(1)}\times \wh \la^{(2)}$ we note that $W^1(\wh \la^{(1)})=W^1(\wh \la^{(1)})$ and $W^2(\wh \la)=W^2(\wh \la^{(2)})$, proving (a). 

As $x\in K(\la)$ stabilizes $Q^1(\wh\la)$ by (a), it can be written as product $x_1x_2$ where $x_i\in \ov W^i( \la^{(i)})$. Since $\wh \la^x=\wh \la \mu$ for the faithful character $\mu$ of $\wh L/L$, it satisfies $(\wh \la ^{(i)})^x= \wh \la^{(i)} \mu^{(i)}$ where $\mu^{(i)}=\restr \mu |{\wh L^{(i)}}$ with $\wh L^{(i)}:=\spann<\wh L_I\mid I \in Q^i(\la)>$. Hence $(\wh \la ^{(i)})^{x_i}= \wh \la^{(i)} \mu^{(i)}$. 

In the following we show that any element $x_2 \in \ov W^2$ with $(\wh \la ^{(2)})^{x_i}= \wh \la^{(2)} \mu^{(i)}$ also satisfies $x_2\in W$. This then implies the statement in part (b). Recall that $\ov W_d\leq W$ for $d\in \DD_\even$. Hence, without loss of generality we can assume that $Q^2(\wh\la)\subseteq \cO_d$ for some $d\in\DD_\odd \setminus\{ -1\}$. 

For $I_0\in \cO$ and $\kappa\in\cusp(\wt L_{I_0})$ we set 
\begin{align*}
\II Okappalambad@{\cO_\kappa(\phat\lambda)}&:=\{I\in\cO\mid \wh \la_I \text{ is $V_S$-conjugate to }\kappa \text{ or }\kappa^{c_{I_0}} \}.
\end{align*}
Let $\mu_I$ be defined as in \Cref{lem6_12}. Then $x_2(\cO_\kappa(\wh \la))= \cO_{\kappa \mu_{I_0}}(\wh \la)$, see \ref{lem6_12}(a). With $\II oOkappahatlambda@ {\protect\overline \cO_\kappa(\protect\widehat \lambda)}:=\cO_\kappa(\wh \la)\cup \cO_{\kappa \mu_{I_0}}(\wh \la)$ the element $x$ can be written as product of $x_{\overline \cO_\kappa(\la)}\in \Sym_{\pm \overline \cO_\kappa(\la)}$ where $I\in \cO$ and $\kappa\in\cusp(\wh L_I)$ runs over the $\spann<\mu_I>\times \spann<c_I>$-orbits in $\cusp(\wh L_I)$.
To prove $x_2\in W$ it is sufficient to prove $x_{\overline \cO_\kappa(\la)}\in W$. Hence we assume $Q^2(\wh\la)= \cO_\kappa(\wh \la)\cup \cO_{\kappa \mu_{I_0}}(\wh \la)$ for some $\kappa\in\Irr(\wh L_{I_0})$.

If $I_0\in\cO_1$, we observe that $o(\kappa)\notin \{1,2,4\}$ by the definition of $Q^2(\wh\la)$. This implies $\kappa\mu_{I_0} \notin\{\kappa , \kappa^{c_{I_0}} \}$ and hence  $\cO_\kappa(\wh \la)\cap \cO_{\kappa \mu_{I_0}}(\wh \la)=\emptyset$. Note that $\ov W^2(\wh \la)\leq W$. The element $x_2$ satisfies $x_2(\cO_\kappa(\wh \la)) = \cO_{\kappa \mu_{I_0}}(\wh \la)$ as element of $\Sym_{Q^2(\wh\la)}$. Recall $\wh \la$ is standardized. Let $I\in \cO_\kappa(\wh \la)$ and $I'\in \cO_{\kappa \mu_{I_0}}(\wh \la)$. If $\wh \la_I$ and $\wh \la_{I'}$ are $V_S$-conjugate,
\begin{align}
x_2(\eps\cO_\kappa(\wh \la)) = \eps\cO_{\kappa \mu_{I_0}}(\wh \la) &\und x_2(\eps\cO_{\kappa \mu_{I_0}}(\wh \la)) = \eps\cO_\kappa (\wh \la) 
\end{align}
for every $\eps\in\{\pm 1\}$ as element of $\Sym_{\pm Q^2(\wh\la)}$. Otherwise
\begin{align}
x_2(\eps\cO_\kappa(\wh \la)) = -\eps\cO_{\kappa \mu_{I_0}}(\wh \la) &\und x_2(\eps\cO_{\kappa \mu_{I_0}}(\wh \la)) = -\eps\cO_\kappa (\wh \la)
\end{align}for every $\eps\in\{\pm 1\}$ as element of $\Sym_{\pm Q^2(\wh\la)}$.
In both cases we see $x_2\in W$. 

Assume $I_0\in\cO_d$ for $d\in \DD_\odd\setminus\{\pm 1\}$. Hence  $L_{I_0}\cong \GL_d(q)$ by \Cref{lem33_loc} and $c_{I_0}$ acts on $L_{I_0}$ as a graph automorphism by \Cref{lemVactK}(c). According to \Cref{lem6_13} we have $\kappa\neq \kappa^{c_{I_0}}$ and 
$\ov W^2(\wh \la)\leq W$. \Cref{prop51}(b) leads to $\kappa\mu_{I_0} \notin\{\kappa , \kappa^{c_{I_0}} \}$. We see again that $\cO_\kappa(\wh \la)$ and $ \cO_{\kappa \mu_{I_0}}(\wh \la)$ are disjoint. This implies again that there exists some $\eps'\in\{\pm 1\}$ such that
\begin{align}
x_2(\eps\cO_\kappa(\wh \la)) = \eps' \eps\cO_{\kappa \mu_{I_0}}(\wh \la) &\und x_2(\eps\cO_{\kappa \mu_{I_0}}(\wh \la)) = \eps'\eps\cO_\kappa (\wh \la)
\end{align}for every $\eps\in\{\pm 1\}$ as elements of $\Sym_{\pm Q^2(\wh\la)}$.
Again $x_2\in W$. Altogether this proves part (b). 

The considerations above imply $|\cO_\kappa(\wh\la)|=|\cO_{\kappa\mu_I}(\wh\la)|$ for every $I\in \cO$ and $\kappa \in \cusp(\wh L_I)$. If $\cO_\kappa(\wh\la)\subseteq Q^2(\wh\la)$, the sets are disjoint so that $2\mid |\ov \cO_\kappa(\wh \la)|$. 
This also applies if $I\in \cO_1$ and $\kappa\in \cusp(\wh L_I)$ with $o(\kappa)\mid 2$. This gives part (c).

Recall $K(\la)=\ov W_
{\restr\la^{L\spannFp }  | L }$. We see that $Q^1(\kappa)=Q^1(\lambda)$ and $Q^2(\kappa)=Q^2(\la)$ 
for every constituent $\kappa$ of $\restr\la^{L\spannFp}|{L}$.  Accordingly we see that $K(\la)\leq \Sym_{\pm Q^1(\la)}\times \Sym_{\pm Q^2(\la)} $.

By  definition, $\la^{(i)}$  is uniquely determined by $\la$. 
Let $w\in K(\la)$, $w_1\in \Sym_{\pm Q^1(\la)}$ and $w_2\in  \Sym_{\pm Q^2(\la)}$ with $w=w_1w_2$.
As $\la^{w_1w_2}$ is some $L\spann<F_p>$ conjugate of $\la$, the character  $(\la^{(1)})^{w}=(\la^{(1)})^{w_1}$ is a $L^{(1)} \spann<{F_p}>$-conjugate of $\la^{(1)}$. Hence $w_1\in K^ 1(\la)$. Analogously we can argue for $w_2$ and get  $w_2\in \NNN_{\ov W^2}(W^2(\wh \la))$,
as required in (d).
\end{proof}

We study first the Clifford theory for $W(\wh \la)\lhd W(\la)$ by considering subgroups associated to $Q^1(\wh\la)$ and $Q^2(\wh\la)$. 
\begin{lem}\label{lem521} In the situation of \Cref{lem6_29} maximal extendibility holds \wrt $W^2(\wh \la^{(2)})\lhd \NNN_{\ov W^2}(W(\wh\la^{(2)}))$.
\end{lem}
\begin{proof}
Without loss of generality we can assume that $\wh \la$ is standardized. The structure of $\ov W(\wh\la)$ is then given by \Cref{lem510}. 
%Let $\mu^{(2)} \in \Irr(\wh L^{(2)})$ with $\ker(\mu^{(2)})=L^{(2)}$ and $k\in K^2(\la)$. Note $\ov W({\wh \la^{(2)}}\mu^{(2)})=\ov W({\wh \la^{(2)}} )$. By the definition of $K^2({\la^{(2)}})$ there exists some $F'\in \spann<F_0>$ such that $({\la^{(2)}})^k=({\la^{(2)}})^{F'}$. Accordingly $\wh \phi:=(\wh \la^{(2)})^k \in \Irr(\wh L\mid ({\la^{(2)}})^{F'})$. This implies that 
%$$\ov W({\wh \la^{(2)}})^{k}=\ov W(({\wh \la^{(2)}})^k)=\ov W(\wh\phi).$$  We see that $\wh \phi^{(F')^{-1}} \in \Irr(\wh L\mid {\la^{(2)}})=\{{\wh \la^{(2)}},  {\wh\la^{(2)}}\mu^{(2)}\}$, and hence $\ov W(\wh \phi )= \ov W(\wh \phi )^{(F')^{-1}}= \ov W({\wh \la^{(2)}} )$. 
%This implies $W(\wh \la^{(2)})\lhd K^2(\la^{(2)})$ and hence $K^2(\la^{(2)})\leq \NNN_{\ov W^2}(W(\wh \la^{(2)}))$. 
As in the proof of \Cref{CliffW2K2}(a) the groups $W^2( \wh \la^{(2)})$ and $\NNN_{\ov W^2}(W(\wh \la^{(2)}))$ satisfy \Cref{lem_wr} and accordingly maximal extendibility holds. 
\end{proof}

\begin{prop}\label{lem522}
Let $\la\in\cusp(L\mid 1_{\spann<h_0>})$ , $\wh \la\in\Irr(\wh L\mid \la)$ and $\eta_0\in \Irr(W(\wh \la))$ with $\wt L_\la=\wh L$.
\begin{thmlist}
\item If $Q^1(\wh \la)=\cO$, then every $\eta\in\Irr(W(\la)\mid \eta_0)$ is $K(\la)_{\eta_0}$-stable.
\item Maximal extendibility holds \wrt $W^1(\wh \la^{(1)})\lhd K^1(\la^{(1)})$.
\end{thmlist}
\end{prop}
 
\begin{proof}
%\label{not5_34}
The statement in (a) is trivial if $W(\wh \la)= W(\la)$. Hence we assume in the following $W(\wh \la)\neq W(\la)$ and $Q^1(\wh\la)=\cO$. According to \Cref{lem-1inDD}, $-1\in\DD$ and $\wh \la_{-1}$ is $c_{J_{-1}}$-stable, i.e., $$|\ov W_{-1}(\wh \la)|=2.$$ 

In order to study those groups we introduce more notation: For $j\in\{1,2,4\}$ let $Q^{1,j}=\{I\in \cO_1\mid o(\wh \la_I)=j\}$ and $l_j:= |Q^{1,j}|$. Then $$\ov W^{Q^{1,j}}(\wh \la^{(1)}) =\begin{cases}
\Sym_{\pm Q^{1,j}} & \text{if }j\in \{1,2\},\\
\Sym_{ Q^{1,4}} & \text{if }j=4.
\end{cases}$$
Accordingly $\ov W^{1}(\wh \la)= \ov W_{-1}(\wh \la)\times \Sym_{\pm Q^{1,1}}\times \Sym_{\pm Q^{1,2}}\times \Sym_{ Q^{1,4}}$. Additionally $\ov W_{\restr\wh \la^{\wh L \EL }|{\wh L}}$ stabilizes the sets $J_{-1}$ and $Q^{1,j}$ ($j \in\{1,2,4\}$). If $W(\la)\neq W(\wh \la)$, every $x\in W(\la)\setminus W(\wh \la)$ satisfies $x(Q^{1,1})=Q^{1,2}$ as element of $\Sym_{Q^1(\wh\la)}$, see the proof of \Cref{lem6_29}. Hence in that case $l_1=l_2$. 

Following the arguments given in the proof of \Cref{lem6_29} the element $x_1\in W^1(\la)\setminus W^1(\wh \la)$ can be written as 
$x_{-1} x_{\{1,2\}} x_{4}$, where 
$x_{-1}\in \spann<(J_{-1},-J_{-1})>$, 
$x_{\{1,2\}}\in\Sym_{\pm (Q^{1,1}\cup Q^{1,2})}$ 
with $x_{\{1,2\}}(Q^{1,1})=Q^{1,2}$ as element of $\Sym_{(Q^{1,1}\cup Q^{1,2})}$ and $x_{4}\in \spann<x_4^\circ>$ with 
$$x_{4}^\circ:=\prod_{i\in Q^{1,4}} (i,-i)\in \Sym_{\pm Q^{1,4}}.$$

Note that $-1\in\DD$ according to \Cref{lem-1inDD} and hence $\gamma\in \EL$. In this notation we have 
$$K(\la)\leq K^1:=
\spann<(J_{-1},-J_{-1})> \times 
\left ((\Sym_{\pm Q^{1,1}}\times \Sym_{\pm Q^{1,2}})\rtimes \spann<x_{\{1,2\}}>\right )
\times \left (\Sym_{Q^{1,4}}\rtimes \spann<x_{4}^\circ>\right ).$$

We see $W(\wh \la)\cong\Sym_{\pm Q^{1,1}}\times \Sym_{\pm Q^{1,2}}\times \Sym_{Q^{1,4}}$. Since $ K(\la)=\ov W_{\restr\la^{L\spannFp}|L}$, we have
$$K(\la) \leq \spann<(J_{-1},-J_{-1})>\times \left ((\Sym_{\pm Q^{1,1}}\times \Sym_{\pm Q^{1,2}})\rtimes{\spann<x_{\{1,2\}}>} \right )\times \spann<x_{4}>\Sym_{Q^{1,4}}\leq W(\wh \la)\spann<x,c_{-1}, x_{4}^\circ>$$
 for the element $x$ from above and with $c_{-1}:=\rho_{\bT}(c_{J_{-1}})$. 
Note $W(\wh \la)\spann<x,c_{-1}, x_{4}>\leq  K^1=W( \la)\spann<c_{-1}, x_{4}^\circ >$. We observe $ c_{-1},x_{4}^\circ\in \Z(W(\la)\spann<c_{-1}, x_{4}^\circ>)$. This implies that every character $\eta\in \Irr(W^1(\la))$ is $K^1$-stable. This proves part (a), and even that every character of $\Irr(W(\la))$ extends to $K^1$. 

Now by the definition of $W^1(\wh \la^{(1)})$ and $K^1(\la^{(1)})$ we see that in the general case the groups obtained as $K^1(\la)$ coincide with $K(\la)$ for a group of smaller rank where for the character $\wh\la^{(1)}$  part (a) can be applied. This then proves part (b).\end{proof}
We consider the general case. 
\begin{prop}\label{prop5mixed}
Let $\la\in \cusp(L)$ with $\wt L_\la=\wh L$ and $\eta_0\in\Irr(W(\wh \la))$. Then every character in $\Irr(W(\la)\mid \eta_0)$ is $K(\la)_{\eta_0}$-stable.
\end{prop}
\begin{proof} 
Note that because of $|\wh L_\la:L|=2$ it is sufficient to prove that some character in 
 $\Irr(W(\la)\mid \eta_0)$ is $K(\la)_{\eta_0}$-stable. According to \Cref{lem521} and \Cref{lem522} we can assume $Q^1(\wh\la)\neq\cO\neq Q^2(\wh\la)$ for the sets $Q^1(\wh \la)$ and $Q^2(\wh \la)$ from \Cref{lem6_29}.

By \Cref{lem6Dh0} we can assume $h_0\in \ker(\la) $.  The groups $W^i(\wh \la)$ ($i\in\underline 2$)  satisfy $W(\wh \la)=W^1(\wh \la)\times W^2(\wh \la)$, see \Cref{lem6_29}. 

If $\cO_1\cap Q^1(\wh \la)\neq \emptyset$, then $K(\la)\leq  K^1(\la)\times \NNN_{\ov W^2}(W^2(\wh \la ^{(2)}))$, see \Cref{lem6_29}(d). Let $\eta_i\in\Irr(W^i(\wh \la))$ such that $\eta_0=\eta_1\times \eta_2$. According to \Cref{lem522}, $\eta_1$ has a $K^1(\la)_{\eta_1}$-stable extension to $W^1(\la)_{\eta_1}$ and  maximal extendibility holds \wrt $W^2(\wh \la)\lhd \NNN_{\ov W^2}(W^2(\wh \la ^{(2)}))$ according to \Cref{lem521}. This proves the statement in that case.

If $\cO_1\cap Q^1(\wh \la)= \emptyset$, then $Q^1(\wh \la)=\{J_{-1}\}$. Then $|W^1(\wh\la)|=1$ and therefore $W(\wh\la)=W^2(\wh\la)$. Then the stability statement follows by applying  again \Cref{lem521}. 
\end{proof}
Together with \Cref{prop_sec4B} this leads to the following statement. 
\begin{cor}\label{corsec5}
Let $\la\in\TT$ with $\wt L_\la\neq L$, $\wt  \la\in\Irr(\wt L_\la\mid \la)$ and $\eta_0\in\Irr(W(\wt \la))$. Then there exists some  $K(\la)_{\eta_0}$-stable $\eta\in\Irr(W(\la)\mid \eta_0)$.
\end{cor}
\begin{proof}For $\la\in \MM^{(\wt L)}\cup \MM_0$ this is  \Cref{prop_sec4B}. For  $\la\in\cusp(L)$ with $\wt L_\la=\wh L$ the statement follows from \Cref{prop5mixed}. 
\end{proof}

%%%%%%%%%%%%%%%%%%%%%%%%%%%%%%%%%%%%%%%%%%%%%%%%%%%%%%%
\section{Proof of Theorem A} \label{ssec6E}
In the following we explain how \Cref{corsec5} about the action of $K(\la)$ on $\Irr(W(\la))$ proves \Cref{thm_loc*}. As already sketched in the beginning of \Cref{secWla} based on \Cref{prop23}, knowing the action of $\wh K(\la)$ on $\Irr(W(\la))$ is crucial to verify \Cref{thm_loc*}. Unless $\bG$ is of type $\tD_4$ the action $\wh K(\la)$ on $\Irr(W(\la))$ is given by the action of $K(\la)$, see \Cref{KlawhKla}.

Via Harish-Chandra induction we transfer the result of \Cref{thm_loc*} on characters of $N$ to a weak version of \Cref{thm1}. Special considerations are needed to determine the stabilisers in $\wG E$  of characters $\chi\in\Irr(\GF\mid(L,\cusp(L)))$, whenever $L$ is not $E(\GF)$-stable. 

\begin{lem}\label{lem6_36}
Let $\la \in \cusp(L)$, $\wt \la\in\Irr(\wt L_\la\mid \la)$ and $\eta_0\in\Irr(W(\la))$. Then some character in $\Irr(W(\la)\mid \eta_0)$ is $\wh K(\la)_{\eta_0}$-stable. 
\end{lem}
\begin{proof}
According to  \Cref{corsec5} there exists some $K(\la)_{\eta_0}$-stable $\eta$ in $\Irr(W(\la)\mid \eta_0)$. According to \Cref{KlawhKla} the character $\eta$ is $\wh K(\la)_{\eta_0}\cap (W\rtimes \EL^\circ )$-stable, where $E^\circ:=\spann<F_p,\gamma_0>$ and $\EL^\circ:=E^\circ\cap \EL$. If $\bG$ is not of type $\tD_4$ or $\EL\leq \spann<F_p,\gamma_0>$ this is the above statement. 

Accordingly we can assume in the following that $\bG$ is of type $\tD_4$ and $L$ is $\gamma_3$-stable for the graph automorphism $\gamma_3$ of $\tD_4(\FF)$ from \ref{not_32}. 
If $L=T_0$ the statement follows from Theorem~3.7 of \cite{MS16}. Otherwise  $L$ is one of the other two possible $\gamma_3$-stable Levi subgroups. In both cases easy calculations show that $W(L)$ is a $2$-group. According to our considerations above we know that there is some $K(\la)_{\eta_0}$-stable $\eta\in\Irr(W(\la)\mid \eta_0)$. As $[K(\la):W(\la)]\in \{1,2\}$ the character $\eta_0$ extends to its inertia group $K(\la)_{\eta_0}$. This shows that maximal extendibility holds with respect to $W(\wt \la)\lhd K(\la)$. 
Let $K'(\la)=(\wh K(\la)\spann<F_p>)\cap (W\rtimes \spann<\gamma,\gamma_3>)$. When we identify $\ov W$ with $W\rtimes \spann<\gamma>$, we can see $K(\la)$ as a subgroup of $K'(\la)$ with index $1$ or $3$. Hence $K(\la)_{\eta_0}$ has index $1$ or $3$ in $K'(\la)_{\eta_0}$. The character $\eta_0$ extends to $K(\la)_{\eta_0}$ by the above. 

Assume that $K(\la)_{\eta_0}/ W(\wt\la) $ is a Sylow $2$-subgroup of $K'(\la)/W(\wt\la)$. Let $K_3$ be a subgroup of $K'(\la)_{\eta_0}$ with $W(\la)_{\eta_0}\leq K_3$ such that $K_3/W(\la)_{\eta_0}$ is a Sylow $3$-subgroup of $K'(\la)_{\eta_0}/W(\wt\la)$. The character $\eta_0$ extends to $K_3$ as $|W(\wt \la)|$ is coprime to $3$ according to \cite[(11.32)]{Isa}.  This implies that $\eta_0$ extends to $K(\la)_{\eta_0}$. 
Maximal extendibility holds \wrt $W(\wt \la)\lhd  K'(\la)$ as well. (This can be seen via an application of \cite[(11.31)]{Isa}.)

If $K(\la)_{\eta_0}/ W(\wt\la) $ is not a Sylow $2$-subgroup of $K'(\la)_{\eta_0}/W(\wt\la)$, the group $K(\la^{\gamma_3})_{\eta_0^{\gamma_3}}=(K(\la)_{\eta_0})^{\gamma_3}$ contains a Sylow $2$-subgroup of  $K'(\la^{\gamma_3})_{\eta_0^{\gamma_3}}$ and hence by the above $(\eta_0)^{\gamma_3}$ extends to $K'(\la^{\gamma_3})_{\eta_0^{\gamma_3}}$. Via conjugation this implies that $\eta_0$ extends to $K'(\la)_{\eta_0}$.
\end{proof}
We can now show \Cref{thm_loc*}.
\begin{proof}[Proof of \Cref{thm_loc*}]
Recall $\MM^{(X)}:=\{\la\in\cusp(L)\mid \wt L=X\}$ for the subgroups $L\leq X\leq \wt L$ and $\MM_0:=\cusp(L)\setminus (\MM^{(L)} \cup \MM^{(\wh L)} \cup \MM^{(\wt L)} )$, see before \Cref{lem6_3}. As the sets $\MM^{(L)}$, $\MM^{(\wh L)}$, $\MM^{(\wt L)}$ and $\MM_0$ are $\EL$-stable, it is sufficient to construct an $\wh N$-stable $\wt L$-transversal in $\Irr(N\mid \MM')$ for $\MM'\in \{\MM^{(L)}, \MM^{(\wh L)}, \MM^{(\wt L)}\}$. Note that since every character of $N$ is $N$-stable, one can equivalently also construct $\EL$-stable $\wt N'$-transversals. \Cref{lem6_3} provides an $\wh N$-stable $\wt L$-transversal in $\Irr(N\mid \MM^{(L)})$. 

\Cref{lem6_36} shows that for every  $\la\in\MM_0\cup \MM^{(\wt L)}\cup \MM^{(\wh L)}$ and every character $\eta_0\in\Irr(W(\wt \la))$ there is some $K(\wt \la)_{\eta_0}$-stable $\eta\in\Irr(W(\la)\mid \eta_0)$, where $\wt \la\in\Irr(\wt L_\la\mid \la)$. 

Assumptions (i) and (ii) of \Cref{prop23} are satisfied with $\TT':=\TT \cap (\MM_0\cup \MM^{(\wt L)}\cup\MM^{(\wh L)})$ from \Cref{def_TT} and the extension map $\Lambda$ from \Cref{thm_loc}. For every $\la\in\TT'$ and $\eta_0\in\Irr(W(\wt \la))$ there exists some $K(\la)_{\eta_0}$-stable $\eta\in\Irr(W(\la))$.
This allows us to apply \Cref{prop23_neu} and hence some $\wh N$-stable $\wt N$-transversal in $\cusp(N\mid \MM_0\cup \MM^{(\wt L)}\cup\MM^{(\wh L)} )$ exists. \end{proof}
\Cref{thm_loc*} implies according to \Cref{thm_MS} that the equation $(\wG \EL)_\chi=\wG_\chi(\EL)_{\chi}$ holds for every character $\chi$ of a $\wt G$-transversal in $ \Irr(\GF\mid (L,\cusp(L)))$. Accordingly we have constructed an $\EL$-stable $\wG$-transversal of $\Irr(\GF\mid(L,\cusp(L)))$, see \Cref{*cond_trans}. 
\begin{corollary}\label{AinftyE1}
 Let $\bG=\tDlsc(\FF)$, $F:\bG\ra\bG$ a standard Frobenius endomorphism and $E$ be defined as in \ref{not_32} and $\wG:=\calL\inv(\Z(\bG))$ for the Lang map $\calL$ defined by $F$ on $\bG$. Let $L$ be a standard Levi subgroup of $\GF$ and $\EL$ its stabilizer in $E(\GF)$. If \Cref{hyp_cuspD_ext} holds for every $l'< l$, then there exists an $\EL$-stable $\wG$-transversal in $\Irr(\GF\mid(L,\cusp(L)))$. 
\end{corollary}

\begin{proof}For a given fixed Levi subgroup $L$ we apply \Cref{thm_MS} whose assumptions follow from \Cref{thm_loc} and \Cref{thm_loc*}.\end{proof}

Condition $A'(\infty)$ from \ref{recallAinfty} and equivalently \Cref{thm1}  require to replace in the above statement  $\EL$ by $E$ and study $(\wt G E)_\chi$. Hence we study the stabilizers of characters in $\Irr(\GF\mid(L,\cusp(L)))$ in the case where $L$ is a standard Levi subgroup that is not $E$-stable. 

\begin{prop}\label{prop3_23} We keep $\bG=\tDlsc(\FF)$ and assume \Cref{hyp_cuspD_ext} holds for every $l'< l$. Let $\bT$ and $\bL$ as in Notation~\ref{not_15}. Let $E^\circ:=\spann<F_p,\gamma>\leq E$ in the notation of \ref{not_32}.
Assume that no $\NNN_\GF(\bT)$-conjugate of $\bL$ is $E$-stable. Let $\chi\in \Irr(\GF\mid(L,\cusp(L)))$. Then  $\wG_\chi=\wG$ or  $(\wt G E^\circ)_\chi\leq \wt G_\chi (E^\circ\cap \EL)$.
\end{prop}

\begin{proof}
Let $N_0:=\NNN_\GF(\bT)$. We consider first the possible structure of $L$, in particular the values of $\DD(L)$. Then we give the possible values of $\wt L_\chi$ via describing $W(\la)$.

We see that $L$  is $F_p$-stable. If $L$ is $\gamma $-stable, then $E(\GF)=\EL$. By our assumption $\EL\neq E$ we have $\gamma\notin \EL$. We observe that then $-1\notin\DD(L)$, as otherwise the system of simple roots $\Delta'$, associated to $L$ as in \ref{notationL} is $\gamma$-stable, which then implies $ \gamma \in E_L$. 

If $1\in \DD(L)$, then some $N_0$-conjugate of $L$ is $\gamma$-stable: The conjugation is given by some element $v\in N_0:=\NNN_\GF(\bT_0)$ that corresponds to some $\sigma\in \Sym_{\pm \underline l} $ with the following properties: $\si(\ul)=\ul$, $1\in \si(J_1)$ and $\sigma(\Delta')\subseteq \Delta$. The Levi subgroup $L$ satisfies accordingly $1\notin \DD(L)$. 

Let $W_0=N_0/\bT^F$. For the proof of the statement we consider $\chi\in\Irr(\GF\mid (L,\cusp(L)))$ with $(\wG E)_\chi\leq \wG E_L$. Then $\chi^{\gamma}$ and $\chi$ are $\wG E_L$-conjugate. For the statement we have to show that $\wG_\chi=\wG$. 
We assume that $\chi^{\gamma}$ and $\chi$ are $\wG E_L$-conjugate. Let  $\la\in\cusp(L)$ with $\chi\in \Irr(\GF\mid (L,\la))$. Then $(L,\la)$ and $(\gamma(L),\la^\gamma)$ are $\wG E_L$-conjugate, in particular $\gamma(L)$ and $L$ are $N_0$-conjugate.  
This shows $\ov W(L)\neq W(L)$, or equivalently $\DD_\odd\neq \emptyset$. Let $\II Oodd@{\cO_\odd}=\bigcup_{d\in \DD_\odd} \cO_d$, $\II Oeven@{\cO_\even}=\bigcup_{d\in \DD_\even} \cO_d$ and $I_0\in\cO_{\odd}$. Without loss of generality we assume $1\in  I_0$. Otherwise we  replace $L$ by some $N_0$-conjugate. Let $w_0:=\prod_{i\in I_0} (i,-i)\in W_0$ and $n_0\in \ov N=\NNN_{\ov \bG^F}(\bT)$ the corresponding element. Hence $w_0\in \gamma  N_0$. We note that $N$ induces on $L$ the outer automorphisms $W$, while any $w'\in \ov W\setminus W$ is induced by elements of $\NNN_{\GF  \spann<\gamma >}(L)\setminus N$. This proves that in this case $L$ and $\gamma (L)$ are actually $N_0$-conjugate. Hence the Harish-Chandra series satisfy $\Irr(\GF\mid(L,\cusp(L))=\Irr(\GF\mid(\gamma(L), \cusp(\gamma(L)))$. 

Let $\la\in\cusp(L)$. Assume that $(L,\la)$ and $(L,\la)^{e\gamma }$ are $\GF$-conjugate for some $e\in \spannFp$. This implies that  $(L,\la)=(L,\la)^{e \ov n}$ for some $\ov n\in \ov N\setminus N$.  Note that $W(\la)^{\ov n}=W(\la)$. Because of $-1\notin\DD(L)$ and $\DD_\odd\neq \emptyset$ we observe $\wt L_\la=\wt L$, as $\spann<\wh L_{I_0}, t_{\ul,2}>\leq \Cent_\bG(L)L$. 

Assume $h_0\in\ker(\la)$ and that some $\chi\in\Irr(\GF\mid(L,\la))$ satisfies $\wG_\chi\neq \wG$. According to \Cref{cor6_20} the equation $W(\wh \la)=W(\wt \la)$ holds, as $\{\pm 1\}\cap \DD(L)=\emptyset$.
If $W(\la)= W(\wh \la)$, then $\wG_\chi=\wG$. Hence we assume $W(\la)\neq W(\wh \la)$ in the following. Let $W^\odd(\la):= (W(\la) \Sym_{\pm \cO_\even})\cap\Sym_ {\pm \cO_\odd}$.
Without loss of generality we can assume $\la$ to be standardized, as in every $\ov N$-orbit in $\cusp(L)$ there is at least one standardized character, see after \ref{def_standardized}.
 As $W(\la)\neq W(\wh \la)$ and $\la$ is standardized, $W^{\odd}(\wh\la) \leq \Sym_{\cO_\odd}$ and $x\in W(\la)\setminus W(\wh \la)$ can be chosen as an involution with no fixed point in $\cO_\odd$. We note that $\NNN_{\Sym_{\pm \cO_\odd}}(W^\odd(\la)) \leq W$. This implies $\NNN_{\ov W}(W(\la)) \leq W$ and hence $(L,\la)$ and $(L,\la)^{e\gamma }$ are not $\GF$-conjugate for any $e\in \spannFp$.  This proves that $(L,\la)^\gamma$ is not $\GF$-conjugate to any element of the $\EL$-orbit of $(L,\la)$, when $h_0\in\ker(\la)$. 
 
Assume $h_0\notin\ker(\la)$ and $\nu\in\Irr(\restr \la|{\Z(\GF)})$. In the following we assume $|\Z(\GF)|=4$. Then we observe that $ E^\circ_\nu= \spannFp= \Cent_{E}(\Z(\GF))$ and hence $(\wG E)_\chi\leq \wG E^\circ_\nu$ for every $\chi\in\Irr(\GF\mid (L,\la))$. If $\Cent_{E}(\Z(\GF))=\EL$, this implies $(\wG E)_\chi \leq \wG \EL$ as required. Note that if $2\mid l$, then  $\Cent_{E}(\Z(\GF))=\EL$.   In the following we prove $2\mid |\cO_\odd|$ as this implies $2\mid l$.

For $I\in \cO$ let $\bZ_I$ be defined as in \ref{lem33_loc}, $Z_{I}:=\bZ_{I}^F$ and $\delta_I\in\Irr(\restr \la|{Z_I})$. Fix $d\in\DD_\odd$ and $I_d\in \cO_d$. For  $\kappa\in \Irr(Z_{I_d})$ we define 
$$a_{ \kappa}(\la):=|\{I \in \cO_d\mid \kappa \text{ and } \delta_I \text{ are }\ov V_d\text{-conjugate} \}|.$$

Recall $I_d\in \cO_\odd$ and $h_0\notin \ker(\la)$. Hence  $Z_{I_d}\cong \Cy_{q-1}$ and $o(\delta_{I_d})_2=(q-1)_2$. On the other hand we see that 
$$\sum_\kappa a_\kappa(\la)=|\cO_d|,$$
where $\kappa$ runs over the $\spann<c_{I_d}>$-orbits in $\Irr(Z_{I_d})$. 
By the above  $a_\kappa(\la)=0$ for every $\kappa\in\Irr(Z_{I_d})$ with $o(\kappa)_2\neq (q-1)_2$. If $\la'\in \cusp(L)$ is $\ov N$-conjugate to $\la$ then $a_\kappa(\la)=a_\kappa(\la')$ for every $\kappa\in\Irr(Z_{I_d})$ according to the action of $\ov V_d$ on the groups $Z_I$ ($I\in \cO_d$). Note that $a_\kappa(\la)=0$ as $o(\delta_I)_2=(q-1)_2$ for every $I\in \cO_d$.  

Recall that we assume $(L,\la)$ and $(\gamma(L),\la^{\gamma e})$ are $\GF$-conjugate. As the order of $\gamma e$ is even, we can choose some $F_0\in\spannFp$ such that $\spannFnull$ is a Sylow $2$-subgroup of $\spann<e>$. Then the $\ov N$-orbit of $\la$ is $F_0$-stable. Then $F_0$ acts on the characters of $\Irr(Z_{I_d})$, inducing an action on the set of $\spann<c_{I_d}>$-orbits in $\Irr(Z_{I_d})$. We denote this set by $\Irr(Z_{I_d})/\spann<c_{I_d}>$. If $\kappa\in\Irr(Z_{I_d})$ with $o(\kappa)_2=(q-1)_2$, the $\spann<c_{I_d}>$-orbit of $\kappa$ is not $F_0$-stable. Hence the $F_0$-orbit in $\Irr(Z_{I_d})/\spann<c_{I_d}>$ containing $\kappa$ has an even length. Since the $\ov N$-orbit of $\la$ is $F_0$-stable, we see that $a_\kappa(\la^{F_0})=a_{\kappa^{F_0}}(\la)$. 
Accordingly $$2\mid \sum_\kappa a_\kappa(\la),$$
whenever $\kappa$ runs over a $\spann<c_{I_d}>$-transversal in  $\{\kappa'\in \Irr(Z_{I_d})\mid o(\kappa')_2=(q-1)_2\}$.
By the above  $a_\kappa(\la)=0$ for every $\kappa\in\Irr(Z_{I_d})$ with $o(\kappa)_2\neq (q-1)_2$. Altogether this implies $2\mid  \sum_\kappa a_\kappa(\la)= |\cO_d|$, where $\kappa\in\Irr(Z_{I_d})$ runs over a $\spann<c_{I_d}>$-transversal. 

As $l=\sum_{d\in\DD} d |\cO_d|$ and hence $l\equiv |\cO_\odd| \mod 2$, the rank $l$ is even and $\Cent_E(\Z(\GF))=\spannFp=\EL$. As explained above this leads to $(\wG E)_\chi\leq \wG E_{\restr \chi|{\Z(\GF)}}=\wG E_L$ and hence a contradiction to the assumption on $\chi$. 

It remains to study the case of $|Z(\GF)|=2$. Then $2\nmid l$ and $4\nmid (q-1)$. Note that $4\nmid (q-1)$ implies $2\nmid |\EL|$ and hence $|\EL|$ and $o(\gamma)=2$ are coprime. 
If $(L,\la)$ and $(\gamma(L),\la^\gamma)$ are $\GF E_L$-conjugate, then the pairs are already $\GF$-conjugate. In the following we see that the $\GF$-orbit of $(L,\la)$ can not be $\gamma $-stable. By the above we have $\DD_\odd\neq \emptyset$ and $-1\notin\DD(L)$. According to \ref{lem_act_onGI}, $\la$ is $\wh L$-stable, as each $\la_I$ is $\wh L_I$-stable for $I\in \cO_\odd$.  Hence $\la$ is $\wt L$-stable, see also \Cref{lem_act_onGI}. According to \Cref{lem6_12} the assumption $2\nmid l$ implies $W(\la)=W(\wh \la)$, even more $\ov W(\la)=\ov W(\wh \la)\leq W$.  But this implies that $(L,\la)$ and $(L,\la)^\gamma$ are not $\GF$-conjugate. Hence $\gamma \notin(\wG E)_\chi$ and hence $(\wG E)_\chi\leq \wG (E_L)_{\chi}$.
\end{proof}
A last obstacle is formed by the groups $\tD_{4,sc}(q)$. We keep the same notation.
\begin{prop}\label{propD4}
If $\GF=\tD_{4,sc}(q)$, every $\wG$-orbit in $\Irr(\GF)$ contains some $\chi$ with $(\wG E)_\chi=\wG_\chi E_\chi$.
\end{prop}
\begin{proof} 
Let $\chi_0\in\Irr(\GF)$ and $E^\circ:=\spann<\gamma ,F_p>$. Then some Sylow $2$-subgroup of $E$ is contained in $E^\circ$. We can assume that $\wG E^\circ/\GF$ contains a Sylow $2$-subgroup of $(\wG E^\circ)_\chi/\GF$. (Otherwise we can replace $\chi_0$ by one of its $E$-conjugates.) Some $\wG$-conjugate $\chi$ of $\chi_0$ satisfies $(\wG E^\circ)_\chi=\wG_\chi E^\circ_\chi$ according to \Cref{prop3_23}. This proves the statement if  $(\wG E)_\chi\leq \wG E^\circ$. Additionally, $(\wG E)_\chi=\wG_\chi E_\chi$ holds if $\wG_\chi=\wG$. 

Accordingly there is some $f\in \spannFp$ and $t\in \wG$ such that $\chi$ is  $\gamma_3 ft$-stable and $\gamma_3 ft$ has $3$-power order in $\wG E/\GF$. If $t\in \GF$ the equation $(\wG E)_\chi=\wG_\chi E_\chi$ holds. 
Clearly $\wG_\chi\lhd (\wG E)_\chi$. Hence $\wG_\chi$ is normalized by $\gamma_3 f t$. But via the $\spann<\gamma_3 ,f>$-equivariant isomorphism $\wG/\GF\cong\Z(\GF)$ we see that $\wG_\chi=\GF$, as there is no $\gamma_3 f$-stable subgroups of $Z(\GF)$ apart from $\{1\}$ and $\Z(\GF)$. 

The element $\gamma_3 f$ acts on $\Z(\GF)$ such that only the trivial element is fixed by $\gamma_3 f$ and $[\gamma_3 f,\Z(\GF)]=\Z(\GF)\setminus \{1\}$. 

Hence some $\wG$-conjugate $\chi'$ of $\chi$ satisfies $\gamma_3f \in (\wG E)_{\chi'} $. We observe that $\gamma_3 f$ is a $3$-element and hence $o(f)$ is a power of $3$. Note that $f$ acts trivially on $\Z(\bG)$. Since $\chi$ satisfies $(\wG E^\circ)_\chi= \GF E^\circ_\chi$ and $[\Z(\GF),F_p]=1$, this leads to $(\wG E^\circ)_{\chi'}\in \{ G E^\circ_{\chi}, \ \GF \spann<F_p, \gamma  \wh t>_{\chi'}\}$ for some $\wh t \in \wG$ with $\calL(\wh t)=h_0$. Let $\wh G:=\calL\inv(\spannh)$. In the latter case  
$$ \wh t^{-1} \gamma_3(\wh t) f^2= \gamma_3 f (\gamma_3f)^{\gamma  \wh t}\in 
 (\wh G E)_{\chi'} .$$
Recalling that the orders of $\wt G/(\Z(\wt G)\GF)$ and $f$ are coprime, we get $\wh t^{-1} \gamma_3(\wh t)\in \spann<\wh t^{-1} \gamma_3(\wh t) f^2>$, but 
$\wh t^{-1} \gamma_3(\wh t)\in \wG_\chi$ and $\wh t^{-1} \gamma_3(\wh t)\notin \GF$. This leads to a contradiction and we see that $(\wG E^\circ)_{\chi'}=  \GF E^\circ_{\chi}$ and hence $(\wG E)_{\chi'}=\GF \spann<E^\circ_{\chi'}, \gamma_3f>$. 
\end{proof}
%%%%%%%%%%%%%%%%%%%%%%%%%%%%%%%%%%%%%%%%%%%%%%%
We can now deduce Theorem A from \Cref{thm_MS}. 
\begin{proof}[Proof of Theorem A] 
For a given fixed Levi subgroup $L$ of $\GF$ we apply \Cref{thm_MS} whose assumptions follow from \Cref{thm_loc} and \Cref{thm_loc*}. In this way we obtain an $\EL$-stable $\wt G$-transversal in $\Irr(\GF\mid(L,\cusp(L)))$, see \Cref{AinftyE1}. If $L$ is $E$-stable, $\EL=E$ and this gives the required statement. 
If $L$ has an $E$-stable $\GF$-conjugate $L'$, then we observe that $\Irr(\GF\mid(L,\cusp(L)))=\Irr(\GF\mid(L',\cusp(L')))$ and there exists an $E$-stable $\wG$-transversal in $\Irr(\GF\mid(L',\cusp(L')))$. 

It remains to consider the case where $\EL\neq E$ and  no $\GF$-conjugate of $L$ is $E$-stable. Then according to Propositions \ref{prop3_23} and  \ref{propD4} every $\chi'\in\Irr(\GF\mid (L,\cusp(L)))$ has some $\wG$-conjugate $\chi$ with $(\wG E)_\chi=\wG_\chi E_\chi$.
\end{proof}
%%%%%%%%%%%%%%%%%%%%%%%%%%%%%%%%%%%%%%%%%%%%%%%%%%%

\printindex

\end{document}